\documentclass[11pt]{amsart}
\usepackage{amssymb}
\usepackage{amsthm}
\usepackage{mathrsfs}
\usepackage{stmaryrd}
\usepackage{bbm}
\usepackage[frenchb,english]{babel}
\usepackage[utf8]{inputenc}
\usepackage[T1]{fontenc}
\usepackage{lmodern}
\usepackage[all]{xy}
\usepackage{hyperref}
\usepackage{upref}

\def\Rf{\mathbf{Rf}}
\def\Fft{\mathbf{Fft}}
\def\Zar{\mathbf{Zar}}

\def\PP{\textnormal{P}}
\def\QQ{\textnormal{Q}}
\def\XX{\mathfrak{X}}

\def\UU{\mathfrak{U}}
\def\YY{\mathfrak{Y}}
\def\ZZ{\mathfrak{Z}}
\def\SS{\mathfrak{S}}

\def\Coh{\mathbf{Coh}}

\def\Ob{\mathbf{Ob}}
\def\FHom{\mathscr{H}om}

\def\Ad{\mathbf{Ad}}
\def\MS{\mathbf{S}}
\def\MR{\mathbf{R}}
\def\MB{\mathbf{B}}

\DeclareMathOperator{\cris}{crys}
\DeclareMathOperator{\Cris}{Crys}

\DeclareMathOperator{\Cov}{Cov}
\DeclareMathOperator{\fppf}{fppf}
\DeclareMathOperator{\Supp}{Supp}

\DeclareMathOperator{\red}{red}

\DeclareMathOperator{\zar}{zar}
\DeclareMathOperator{\ad}{ad}

\DeclareMathOperator{\Ker}{Ker}

\DeclareMathOperator{\Hom}{Hom}

\DeclareMathOperator{\Spec}{Spec}
\DeclareMathOperator{\Spf}{Spf}

\DeclareMathOperator{\rH}{H}

\DeclareMathOperator{\rD}{D}
\DeclareMathOperator{\rW}{W}

\DeclareMathOperator{\rE}{E}
\DeclareMathOperator{\rG}{G}
\DeclareMathOperator{\rL}{L}
\DeclareMathOperator{\rM}{M}
\DeclareMathOperator{\rR}{R}

\DeclareMathOperator{\id}{id}
\DeclareMathOperator{\rig}{rig}
\DeclareMathOperator{\rconv}{rconv}
\DeclareMathOperator{\conv}{conv}
\DeclareMathOperator{\Conv}{Conv}
\DeclareMathOperator{\RConv}{RConv}
\DeclareMathOperator{\pConv}{pConv}
\DeclareMathOperator{\Iso}{Iso}

\newtheorem{theorem}{Theorem}[section]
\newtheorem{prop}[theorem]{Proposition}
\newtheorem{lemma}[theorem]{Lemma}

\newtheorem{coro}[theorem]{Corollary}

\theoremstyle{definition}
\newtheorem{rem}[theorem]{Remark}
\newtheorem{definition}[theorem]{Definition}
\newtheorem{nothing}[theorem]{}

\numberwithin{equation}{section}
\numberwithin{equation}{theorem}

\makeatletter
\newsavebox{\@brx}
\newcommand{\llangle}[1][]{\savebox{\@brx}{\(\m@th{#1\langle}\)}%
  \mathopen{\copy\@brx\kern-0.5\wd\@brx\usebox{\@brx}}}
\newcommand{\rrangle}[1][]{\savebox{\@brx}{\(\m@th{#1\rangle}\)}%
  \mathclose{\copy\@brx\kern-0.5\wd\@brx\usebox{\@brx}}}
\makeatother

\makeatletter
\newcommand*{\relrelbarsep}{.386ex}
\newcommand*{\relrelbar}{%
  \mathrel{%
    \mathpalette\@relrelbar\relrelbarsep
  }%
}
\newcommand*{\@relrelbar}[2]{%
  \raise#2\hbox to 0pt{$\m@th#1\relbar$\hss}%
  \lower#2\hbox{$\m@th#1\relbar$}%
}
\providecommand*{\rightrightarrowsfill@}{%
  \arrowfill@\relrelbar\relrelbar\rightrightarrows
}
\providecommand*{\leftleftarrowsfill@}{%
  \arrowfill@\leftleftarrows\relrelbar\relrelbar
}
\providecommand*{\xrightrightarrows}[2][]{%
  \ext@arrow 0359\rightrightarrowsfill@{#1}{#2}%
}
\providecommand*{\xleftleftarrows}[2][]{%
  \ext@arrow 3095\leftleftarrowsfill@{#1}{#2}%
}
\makeatother

\title{On higher direct images of convergent isocrystals}
\author{Daxin Xu}
\date{\today}

\AtEndDocument{\bigskip{\footnotesize%
  \textsc{Daxin Xu, Department of Mathematics, California Institute of Technology, Pasadena, CA 91125.} \par
  \textit{E-mail address}: \texttt{daxinxu@caltech.edu} \par
}}

\begin{document}
\selectlanguage{english}
\maketitle

\begin{abstract}
	Let $k$ be a perfect field of characteristic $p>0$ and $\rW$ the ring of Witt vectors of $k$.
	In this article, we give a new proof of the Frobenius descent for convergent isocrystals on a variety over $k$ relative to $\rW$. 
	This proof allows us to deduce an analogue of the de Rham complexes comparaison theorem of Berthelot \cite{Ber00} without assuming a lifting of the Frobenius morphism. 
	As an application, we prove a version of Berthelot's conjecture on the preservation of convergent isocrystals under the higher direct image by a smooth proper morphism of $k$-varieties. 
\end{abstract}

\tableofcontents

\section{Introduction}

\begin{nothing}
	Let $k$ be a perfect field of characteristic $p>0$. A good $p$-adic cohomology theory on a variety over $k$ is the rigid cohomology developed by Berthelot \cite{Ber86, Ber96}. 
	The coefficients for this theory are (over-)convergent $F$-isocrystals: they play a similar role of the lisse $\ell$-adic sheaves in $\ell$-adic cohomology. 
	In (\cite{Ber86} 4.3, \cite{Tsu03}), Berthelot and Tsuzuki conjectured that under a smooth proper morphism of varieties over $k$, the higher direct image of a (over-)convergent ($F$-)isocrystal is still a (over-)convergent ($F$-)isocrystal analogue to the $\ell$-adic case. 
	Various cases and variants of this conjecture have been proved by Tsuzuki \cite{Tsu03}, Shiho \cite{Shi08II}, \'Etess \cite{Ete12}, Caro \cite{Caro15}, etc. We refer to an article of Lazda \cite{Laz15} for a survey on these results and the relation between them. 
	The goal of this article is to prove a version of Berthelot's conjecture in the context of convergent topos developed by Ogus \eqref{intro Berthelots conj}.
\end{nothing}

\begin{nothing} \label{intro conv topos}
	In \cite{Ogus84,Ogus90}, Ogus introduced a crystalline-like site: convergent site and defined a \textit{convergent isocrystal} as a crystal on this site. Let us briefly recall his definition. 

	Let $\rW$ be the ring of Witt vectors of $k$, $K$ its fraction field and $X$ a scheme of finite type over $k$. We denote by $\Conv(X/\rW)$ the category of couples $(\mathfrak{T},u)$ consisting of a flat formal $\rW$-scheme of finite type $\mathfrak{T}$ and a $k$-morphism $u$ from the reduced subscheme $T_{0}$ of the special fiber of $\mathfrak{T}$ to $X$. Morphisms are defined in a natural way. A family of morphisms $\{(\mathfrak{T}_{i},u_{i})\to (\mathfrak{T},u)\}_{i\in I}$ is a covering if $\{\mathfrak{T}_{i}\to \mathfrak{T}\}_{i\in I}$ is a Zariski covering. 

	The functor $(\mathfrak{T},u)\mapsto \Gamma(\mathfrak{T}_{\zar},\mathscr{O}_{\mathfrak{T}}[\frac{1}{p}])$ is a sheaf of rings that we denote by $\mathscr{O}_{X/K}$. An $\mathscr{O}_{X/K}$-module amounts to give the following data: 
	\begin{itemize}
		\item[(i)] for every object $(\mathfrak{T},u)$ of $\Conv(X/\rW)$, an $\mathscr{O}_{\mathfrak{T}}[\frac{1}{p}]$-module $\mathscr{F}_{\mathfrak{T}}$ of $\mathfrak{T}_{\zar}$,

		\item[(ii)] for every morphism $f:(\mathfrak{T}_{1},u_{1})\to (\mathfrak{T}_{2},u_{2})$ of $\Conv(X/\rW)$, an $\mathscr{O}_{\mathfrak{T}_{1}}$-linear morphism $c_{f}:f^{*}(\mathscr{F}_{\mathfrak{T}_{2}})\to \mathscr{F}_{\mathfrak{T}_{1}}$,
	\end{itemize}
	satisfying a cocycle condition for the composition of morphisms as in (\cite{BO} 5.1). 	
	
	A convergent isocrystal on $\Conv(X/\rW)$ is a coherent crystal of $\mathscr{O}_{X/K}$-modules $\mathscr{F}$ on $\Conv(X/\rW)$, i.e. for every object $(\mathfrak{T},u)$ of $\Conv(X/\rW)$, $\mathscr{F}_{\mathfrak{T}}$ is coherent, and for every morphism $f$ of $\Conv(X/\rW)$, the transition morphism $c_{f}$ is an isomorphism.
	We denote by $\Iso^{\dagger}(X/\rW)$ the category of convergent isocrystals on $\Conv(X/\rW)$. 
	If $X$ is smooth over $k$, a convergent isocrystal can be viewed as the isogeny class of a crystal of $\mathscr{O}_{X/\rW}$-module on the crystalline site $\Cris(X/\rW)$ satisfying certain convergent conditions (cf. \cite{Ogus90} 0.7.2 and \cite{Ber96} 2.2.14). 
\end{nothing}
\begin{nothing} \label{intro Frob descent Ogus sec}
	In (\cite{Ogus84} 4.6), Ogus showed that the category $\Iso^{\dagger}(X/\rW)$ satisfies descent property under a proper and surjective morphism of $k$-schemes. 
	Then, if $X'$ denotes the base change of $X$ by the Frobenius morphism of $k$, the functorial morphism of convergent topoi induced by the relative Frobenius morphism $F_{X/k}:X\to X'$ gives an equivalence of categories:
	\begin{equation} \label{intro Frob descent Ogus}
		F_{X/k,\conv}^{*}: \Iso^{\dagger}(X'/\rW)\xrightarrow{\sim} \Iso^{\dagger}(X/\rW),
	\end{equation}
	that we call \textit{Frobenius descent}.

	A \textit{convergent $F$-isocrystal} on $\Conv(X/\rW)$ is a couple $(\mathscr{E},\varphi)$ of a convergent isocrystal $\mathscr{E}$ on $\Conv(X/\rW)$ and an isomorphism $\varphi$ between $\mathscr{E}$ and its pullback via the absolute Frobenius morphism of $X$ (cf. \ref{def F-isocrystal} for a precise definition). 
\end{nothing}

\begin{nothing}
	To study the higher direct image of convergent ($F$-)isocrystals, we need the notion of convergent topos over a $p$-adic base developed by Shiho \cite{Shi02,Shi08} \footnote{Actually, Shiho developed a theory of log convergent site and log convergent cohomology over a $p$-adic base with log structure.}. 
	Let $\SS$ be a flat formal $\rW$-scheme of finite type, $S_{0}$ the reduced subscheme of its special fiber and $X$ an $S_{0}$-scheme. We define the convergent site $\Conv(X/\SS)$ of $X$ relative to $\SS$ and the category $\Iso^{\dagger}(X/\SS)$ of convergent isocrystals on $\Conv(X/\SS)$ as in \ref{intro conv topos} (cf. \ref{def conv cat} and \ref{coh fppf descent}). 
	Shiho generalized Ogus' proper surjective descent for convergent isocrystals in this setting (\cite{Shi08II} 7.3). 

	We denote by $(X/\SS)_{\conv,\fppf}$ the topos of fppf sheaves on the category $\Conv(X/\SS)$ \eqref{def topology Conv}. As a first step towards Berthelot's conjecture, we show the following result. 
\end{nothing}

\begin{theorem}[\ref{Frob descent fppf topos}] \label{intro Frob descent fppf topos}
	Suppose that the Frobenius morphism $F_{S_{0}}:S_{0}\to S_{0}$ is flat.
	Let $X$ be an $S_{0}$-scheme locally of finite type, $X'=X\times_{S_{0},F_{S_{0}}}S_{0}$ and $F_{X/S_{0}}:X\to X'$ the relative Frobenius morphism. 
	The functorial morphism of topoi $F_{X/S_{0},\conv}:(X/\SS)_{\conv,\fppf}\to (X'/\SS)_{\conv,\fppf}$ is an equivalence of topoi. 
\end{theorem}

Our proof is inspired by a site-theoretic construction of the Cartier transform of Ogus-Vologodsky due to Oyama \cite{OV07,Oy} and its lifting modulo $p^{n}$ developed by the author \cite{Xu}. 
By Gabber-Bosch-G\"ortz's faithfully flat descent theory for coherent sheaves in rigid geometry \cite{BG98}, 
we obtain a new proof of the Frobenius descent \eqref{intro Frob descent Ogus}.

\begin{coro}[\ref{Frob descent isocrystals}]\label{intro Frob descent}
	Keep the hypotheses of \ref{intro Frob descent fppf topos}.
	The direct image and inverse image functors of $F_{X/S_{0},\conv}$ induce equivalences of categories quasi-inverse to each other: 
	\begin{equation}
		\Iso^{\dagger}(X/\SS) \leftrightarrows \Iso^{\dagger}(X'/\SS). 
	\end{equation}
\end{coro}


\begin{nothing}
	Keep the notation of \ref{intro Frob descent fppf topos} and suppose that there exists smooth liftings $\XX$ of $X$ and $\XX'$ of $X'$ over $\SS$ (in particular, $X$ is smooth over $S_{0}$). 
	We denote by $\widehat{\Omega}_{\XX/\SS}^{1}$ the $\mathscr{O}_{\XX}$-module of differentials of $\XX$ relative to $\SS$. 
	Given a convergent isocrystal $\mathscr{E}\in \Ob(\Iso^{\dagger}(X/\SS))$, there exists an integrable connection $\nabla:\mathscr{E}_{\XX}\to \mathscr{E}_{\XX}\otimes_{\mathscr{O}_{\XX}}\widehat{\Omega}_{\XX/\SS}^{1}$ on the coherent $\mathscr{O}_{\XX}[\frac{1}{p}]$-module $\mathscr{E}_{\XX}$ \eqref{convergent to isocry}. 
	We denote by $\mathscr{E}_{\XX}\otimes_{\mathscr{O}_{\XX}}\widehat{\Omega}_{\XX/\SS}^{\bullet}$ the associated de Rham complex. 
	We deduce from \ref{intro Frob descent fppf topos} the following result about comparing de Rham complexes for the Frobenius descent: 
\end{nothing}

\begin{coro}[\ref{Comp dR complexes}] \label{intro dR coh}
	Keep the above notation and let $f:X\to S_{0}$ be the canonical morphism. There exists a canonical isomorphism between de Rham complexes of $\mathscr{E}$ and of $F_{X/S_{0},\conv*}(\mathscr{E})$ in $\rD(X'_{\zar},f^{-1}(\mathscr{O}_{\SS}))$:
	\begin{equation}
		F_{X/S_{0}*}(\mathscr{E}_{\XX}\otimes_{\mathscr{O}_{\XX}}\widehat{\Omega}_{\XX/\SS}^{\bullet}) \xrightarrow{\sim} (F_{X/S_{0},\conv*}(\mathscr{E}))_{\XX'}\otimes_{\mathscr{O}_{\XX'}}\widehat{\Omega}_{\XX'/\SS}^{\bullet}.
	\end{equation}
\end{coro}

	In \cite{Ber96II}, Berthelot introduced a sheaf $\mathscr{D}^{\dagger}$ of differential operators over $\XX$ and described (over-)convergent isocrystals in terms of arithmetic $\mathscr{D}^{\dagger}$-modules. 
	If there exists a lifting $F:\XX\to \XX'$ of the relative Frobenius morphism $F_{X/S_{0}}$, he used it to establish a version of Frobenius descent for arithmetic $\mathscr{D}^{\dagger}$-modules and a comparison result for de Rham complexes (cf. \cite{Ber00} 4.2.4 and 4.3.5). 
	The above results can be viewed as a counterpart of Berthelot's results for convergent isocrystals without assuming a lifting of Frobenius morphism. 

Our main result is the following.

\begin{theorem}[\ref{coro main thm}]\label{intro Berthelots conj}
	Let $g:X\to Y$ be a smooth proper morphism of $k$-schemes locally of finite type. 
	The higher direct image of a convergent isocrystal (resp. $F$-isocrystal) on $\Conv(X/\rW)$ \eqref{intro Frob descent Ogus sec} is a convergent isocrystal (resp. F-isocrystal) on $\Conv(Y/\rW)$.
\end{theorem}

In \cite{Ogus84}, Ogus describle $\rR^{i}g_{\conv*}(\mathscr{O}_{X/K})$ in terms of relative crystalline cohomology. Recently, Morrow show that it is same as the higher direct images in crystalline cohomology when $Y$ is smooth \cite{Morrow}. Our proof follows a similar line of their work. 

\begin{nothing}
	In the following, we explain the structure of this article and the strategy for proving \ref{intro Berthelots conj}.

	Section \ref{Preliminary} contains general conventions and a review on formal geometry.
	In section \ref{rel crys coh}, we review (iso-)crystals on crystalline site and the finiteness, base change property of their crystalline cohomology due to Shiho \cite{Shi08}. 
	In section \ref{sec conv topos}, we recall the definition of the convergent topos over a $p$-adic base following Shiho \cite{Shi08}. 
	In section \ref{Higher direct image}, we show that under a smooth proper morphism of smooth $k$-schemes $X\to Y$, the higher direct image of a convergent isocrystal on $\Conv(X/\rW)$ is a ``$p$-adic convergent isocrystal'' on $\Conv(Y/\rW)$, i.e. it satisfies the property of coherent crystal in a certain subcategory of $\Conv(Y/\rW)$ \eqref{p-adic convergent isocry}.
	Section \ref{Frob descents} is devoted to the Frobenius descent (\ref{intro Frob descent fppf topos}, \ref{intro Frob descent}). Using Dworks' trick, we deduce \ref{intro Berthelots conj} in the case where $Y$ is smooth over $k$ \eqref{higher direct image of con Fiso weak}. 
	In section \ref{review rig}, we briefly review Raynaud's approach to rigid geometry following Abbes' book \cite{Ab10}. 
	Section \ref{rig conv topos} is devoted to a modification of convergent topos which allows us to apply the faithfully flat descent in rigid geometry in the full extent. 
	Based on previous results and Ogus' proper surjective descent, we complete the proof of \ref{intro Berthelots conj} in section \ref{final sec}. 
\end{nothing}

\textbf{Acknowledgement.} 
 I would like to thank Ahmed Abbes for discussions and his comments on an earlier version of this paper.
 I would like to thank Atsushi Shiho for his comments and suggestions to improve \ref{intro Berthelots conj}. 
I would like to thank Arthur Ogus for helpful discussions. 
Part of the work was done when the author was at the BICMR and IHÉS and the author would like to thank their hospitality.

\section{Preliminary} \label{Preliminary} 

\begin{nothing}	\label{def cat MS}
	In this article, $p$ denotes a prime number, $k$ denotes a perfect field of characteristic $p$, $\rW$ the ring of Witt vectors of $k$. 

	We denote by $\MS$ the category whose objects are formal $\rW$-scheme of finite type (\cite{Ab10} 2.3.13) and morphisms are adic morphisms (\cite{Ab10} 2.2.7). By (\cite{EGAInew} 6.1.5(v)), morphisms of $\MS$ are of finite type.
	We denote by $\MS^{\diamond}$ the full subcategory of $\MS$ consisting of \textit{flat} formal $\rW$-schemes of finite type. 
	
	Let $\XX$ be a $p$-adic formal $\rW$-scheme. For any $n\ge 1$, we denote by $\XX_{n}$ the reduction modulo $p^{n}$ of $\XX$. 
	If we use a gothic letter $\mathfrak{X}$ to denote a $p$-adic formal $\rW$-scheme, the corresponding roman letter $X$ will denote its special fiber $\XX_{1}$.
\end{nothing}

\begin{nothing} \label{cat up to isog}
	Let $\mathscr{A}$ be an abelian category. We denote by $\mathscr{A}_{\mathbb{Q}}$ the category with same objects as $\mathscr{A}$ such that the set of morphisms is given for any object $M,N$ of $\mathscr{A}$, by
	\begin{displaymath}
		\Hom_{\mathscr{A}_{\mathbb{Q}}}(M,N)=\Hom_{\mathscr{A}}(M,N)\otimes_{\mathbb{Z}}\mathbb{Q}.
	\end{displaymath}
	For any object $M$ of $\mathscr{A}$, we denote its image in $\mathscr{A}_{\mathbb{Q}}$ by $M_{\mathbb{Q}}$.
\end{nothing}

\begin{definition}[\cite{SGAVI} I 1.3.1] \label{def locally proj}
	Let $(\mathscr{T},A)$ be a ringed topos. We say that an $A$-module $M$ of $\mathscr{T}$ is \textit{locally projective of finite type} if the following equivalent conditions are satisfied:

	(i) $M$ is of finite type and the functor $\FHom_{A}(M,-)$ is exact;

	(ii) $M$ is of finite type and every epimorphism of $A$-modules $N\to M$ admits locally a section.

	(iii) $M$ is locally a direct summand of a free $A$-module of finite type.
\end{definition}

When $\mathscr{T}$ has enough points, and for every point $x$ of $\mathscr{T}$, the stalk of $A$ at $x$ is a local ring, the locally projective $A$-modules of finite type are locally free $A$-modules of finite
type (\cite{SGAVI} I 2.15.1).

\begin{nothing}
	In the following of this section, $\XX$ denotes an object of $\MS$ \eqref{def cat MS}. For any $\mathscr{O}_{\XX}$-module $\mathscr{F}$, we set
	\begin{equation} \label{OT mod inverse p}
		\mathscr{F}[\frac{1}{p}]=\mathscr{F}\otimes_{\mathbb{Z}_{p}}\mathbb{Q}_{p}.
	\end{equation}
	We denote by $\Coh(\mathscr{O}_{\XX})$ (resp. $\Coh(\mathscr{O}_{\XX}[\frac{1}{p}])$) the category of coherent $\mathscr{O}_{\XX}$-modules (resp. $\mathscr{O}_{\XX}[\frac{1}{p}]$-modules).
	The canonical functor $\Coh(\mathscr{O}_{\XX})\to \Coh(\mathscr{O}_{\XX}[\frac{1}{p}])$ defined by $\mathscr{F}\mapsto \mathscr{F}[\frac{1}{p}]$ induces an equivalence of categories (\cite{AGT} III.6.16)
	\begin{equation} \label{equi coh mod inverse p}
		\Coh(\mathscr{O}_{\XX})_{\mathbb{Q}}\xrightarrow{\sim} \Coh(\mathscr{O}_{\XX}[\frac{1}{p}]).	
	\end{equation}

	If $\XX=\Spf(A)$ is moreover affine, a coherent $\mathscr{O}_{\XX}[\frac{1}{p}]$-module $\mathscr{F}$ is locally projective of finite type if and only if $\Gamma(\XX,\mathscr{F})$ is a projective $A[\frac{1}{p}]$-module of finite type (\cite{AGT} III.6.17).
\end{nothing}

\begin{nothing} \label{recall fppf covering}
	We denote by $\Zar_{/\XX}$ (resp. $\XX_{\zar}$) the Zariski site (resp. topos) of $\XX$.

	Recall that a family of morphisms of schemes $\{f_{i}:T_{i}\to T\}_{i\in I}$ is called an \textit{fppf covering} if each morphism $f_{i}$ is flat and locally of finite presentation and if $|T|=\bigcup_{i\in I}f_{i}(|T_{i}|)$ (cf. \cite{SGAIII} IV 6.3.2).

	We say that a family of morphisms $\{f_{i}:\XX_{i}\to \XX\}_{i\in I}$ of $\MS$ is an \textit{fppf covering} if each morphism $f_{i}$ is flat and if $|\XX|=\bigcup_{i\in I}f_{i}(|\XX_{i}|)$. Since $\XX$ is quasi-compact, each fppf covering of $\XX$ admits a finite fppf sub-covering of $\XX$. 
	
	By (\cite{Ab10} 2.3.16, 5.1.2), a family of adic formal $\XX$-schemes $\{\XX_{i}\to \XX\}_{i\in I}$ is an fppf covering if and only if the family $\{\XX_{i,n}\to \XX_{n}\}_{i\in I}$ is an fppf covering of schemes for all integer $n\ge 1$. 

	Since fppf coverings of schemes are stable by base change and by composition, the same holds for fppf coverings of formal $\rW$-schemes of finite type. 

	We denote by $\Fft_{/\XX}$ the full subcategory of $\MS_{/\XX}$ consisting of flat adic formal $\XX$-schemes and by $\XX_{\fppf}$ the topos of sheaves of sets on $\Fft_{/\XX}$, equipped with the topology generated by fppf coverings.
\end{nothing}

\begin{nothing} \label{fppf to zar and functorial}
	The canonical functor $\Zar_{/\XX}\to \Fft_{/\XX}$ is continuous and left exact. Then it induces a morphism of topoi
	\begin{equation}
		\alpha_{\XX}:\XX_{\fppf}\to \XX_{\zar}.
	\end{equation}

	Given a morphism $f:\XX'\to \XX$ of $\MS$, the canonical functor $\Fft_{/\XX}\to \Fft_{/\XX'}$ (resp. $\Zar_{/\XX}\to \Zar_{/\XX'}$) defined by $\YY\mapsto \YY\times_{\XX}\XX'$ is continuous and left exact. For $\tau\in \{\zar,\fppf\}$, it induces morphisms of topoi 
	\begin{equation} \label{functorial ftau}
		f_{\tau}:\XX'_{\tau}\to \XX_{\tau}
	\end{equation}
	compatible with $\alpha_{\XX}$ and $\alpha_{\XX'}$.
	If $f$ is a morphism of $\Fft_{/\mathfrak{X}}$ (resp. $\Zar_{/\XX}$), in view of the description of direct image functors, one verifies that the above morphism coincides with the localization morphism at $\XX'$.
\end{nothing}

\begin{nothing} \label{def rigid points}
	We say that a closed formal sub-scheme $\mathscr{P}$ of $\XX$ is a \textit{rigid point} if $\mathscr{P}$ is affine and if $\Gamma(\mathscr{P},\mathscr{O}_{\mathscr{P}})$ is an $1$-valuative order (\cite{Ab10} 1.11.1, 3.3.1). We denote by $\langle\XX\rangle$ the set of rigid points of $\XX$. 

	Let $f:\mathfrak{X}\to \mathfrak{Y}$ be a morphism of $\MS$. Recall that $f$ is \textit{faithfully rig-flat} (\cite{Ab10} 5.5.9) if it is rig-flat (\cite{Ab10} 5.4.5) and if the associated map $\langle\XX\rangle\to \langle\YY\rangle$ on rigid points is surjective.

If $f$ is flat (resp. faithfully flat), then it is rig-flat (resp. faithfully rig-flat) (\cite{Ab10} 5.5.10). 
\end{nothing}

%
%
%

\begin{theorem}[Gabber, Bosch and G\"ortz \cite{BG98}; \cite{Ab10} 5.11.11] \label{descent fppf}
	Let $f:\XX'\to \XX$ be a faithfully rig-flat morphism of $\MS$. Then, the canonical functor from $\Coh(\mathscr{O}_{\XX}[\frac{1}{p}])$ to the category of coherent $\mathscr{O}_{\XX'}[\frac{1}{p}]$-modules endowed with a descent data in $\Coh(\mathscr{O}_{\XX'\times_{\XX}\XX'}[\frac{1}{p}])$ is an equivalence.
\end{theorem}

Although (\cite{Ab10} 5.11.11) is an assertion for coherent modules on rigid spaces, the same argument works for coherent $\mathscr{O}_{\XX}[\frac{1}{p}]$-modules.

Let $\mathscr{U}=\{\XX_{i}\to \XX\}_{i\in I}$ be a fppf covering. 
Since each fppf covering of $\XX$ admits a finite fppf sub-covering of $\XX$, we deduce that any descent datum on coherent $\mathscr{O}_{\XX_{i}}[\frac{1}{p}]$-modules for $\mathscr{U}$ is effective. 
The functor from $\Coh(\mathscr{O}_{\XX}[\frac{1}{p}])$ to the category of descent data with respect to $\mathscr{U}$ is fully faithful.

\begin{nothing} \label{coherent fppf}
	Let $\mathscr{F}$ be a coherent $\mathscr{O}_{\XX}[\frac{1}{p}]$-module. By fppf descent for coherent $\mathscr{O}_{\XX}[\frac{1}{p}]$-modules \eqref{descent fppf}, the presheaf on $\Fft_{/\XX}$
	\begin{displaymath}
		(f:\XX'\to \XX) \mapsto \Gamma(\XX',f^{*}_{\zar}(\mathscr{F}))
	\end{displaymath}
	is a sheaf for the fppf topology. In particular, $\mathscr{O}_{\XX}[\frac{1}{p}]$ defines a sheaf of rings of $\XX_{\fppf}$ that we still denote by $\mathscr{O}_{\XX}[\frac{1}{p}]$. We call abusively \textit{coherent $\mathscr{O}_{\XX}[\frac{1}{p}]$-module of $\XX_{\fppf}$} a sheaf of $\XX_{\fppf}$ associated to a coherent $\mathscr{O}_{\XX}[\frac{1}{p}]$-module.

	For $\tau\in \{\zar,\fppf\}$, the morphism of topoi $f_{\tau}$ \eqref{functorial ftau} is ringed by $\mathscr{O}_{\XX}[\frac{1}{p}]$ and $\mathscr{O}_{\XX'}[\frac{1}{p}]$. 
	For any $\mathscr{O}_{\XX}[\frac{1}{p}]$-module $M$ of $\XX_{\tau}$, we use $f^{-1}_{\tau}(M)$ to denote the inverse image in the sense of sheaves and we keep $f_{\tau}^{*}(M)$ for the inverse image in the sense of modules.
\end{nothing}

\begin{nothing} \label{def dilatation}
	Let $\mathscr{A}$ be an open ideal of finite type of $\XX$ (\cite{Ab10} 2.1.19). 
	We denote by $\XX'$ the admissible blow-up of $\mathscr{A}$ in $\XX$ (\cite{Ab10} 3.1.2). The canonical map $\XX'\to \XX$ is of finite type (\cite{Ab10} 2.3.13) and rig-flat (\cite{Ab10} 5.4.12).

	Suppose that $\XX$ is flat over $\rW$ and that $\mathscr{A}$ contains $p$. The ideal $\mathscr{A}\mathscr{O}_{\XX'}$ is invertible (\cite{Ab10} 3.1.4(i)), and $\XX'$ is flat over $\rW$ (\cite{Ab10} 3.1.4(ii)). We denote by $\XX_{(\mathscr{A}/p)}$ the maximal open formal subscheme of $\XX'$ on which 
	\begin{equation} \label{def dilatation formula}
		(\mathscr{A}\mathscr{O}_{\mathfrak{X}'})|\mathfrak{X}_{(\mathscr{A}/p)}=(p\mathscr{O}_{\mathfrak{X}'})|\mathfrak{X}_{(\mathscr{A}/p)}
	\end{equation}
	and we call it \textit{the dilatation of $\mathscr{A}$ with respect to $p$} (\cite{Ab10} 3.2.3.4 and 3.2.7). Note that $\XX_{(\mathscr{A}/p)}$ is the complement of $\Supp(\mathscr{A}\mathscr{O}_{\XX'}/p\mathscr{O}_{\XX'})$ in $\XX'$ (\cite{EGAInew} 0.5.2.2). 

	Let $S$ be a closed subscheme of $X$ and $\mathscr{I}$ the ideal sheaf associated to the canonical morphism $S\to \XX$. For any $n\ge 1$, we denote by $\mathfrak{T}_{S,n}(\XX)$ the dilatation of $\mathscr{I}^{n}+p\mathscr{O}_{\XX}$ with respect to $p$ (\cite{Ogus84} 2.5) \footnote{The union of rigid space $\cup_{n\ge 1}\mathfrak{T}_{S,n}(\XX)^{\rig}$ is same as the tube of $S$ in $\XX$ introduced by Berthelot (cf. \cite{Ber96} 1.1.2, 1.1.10).}. 
	By \eqref{def dilatation formula}, there exists a morphism from the reduced subscheme of $(\mathfrak{T}_{S,n}(\XX))_{1}$ to $S$ which fits into the following diagram:
	\begin{equation}
		\xymatrix{
			(\mathfrak{T}_{S,n}(\XX))_{1,\red} \ar[r] \ar[d] & \mathfrak{T}_{S,n}(\XX) \ar[d] \\
			S\ar[r] & \XX
		} \label{diagram red to S}
	\end{equation}

	The universal property of dilatation (\cite{Ab10} 3.2.6) can be reinterpreted in the following way. 
\end{nothing}

\begin{prop}[\cite{Xu} 3.6, 3.10] \label{univ dilatation}
	Keep the notations and assumptions of \ref{def dilatation}. 
	Let $\mathfrak{T}$ be an adic flat formal $\rW$-scheme, $\underline{T}$ the closed subscheme of $T$ defined by the ideal sheaf $\{x\in \mathscr{O}_{T}|x^{p}=0\}$ and $f:\mathfrak{T}\to \XX$ an adic morphism.
	Suppose that there exists a morphism $T\to S$ (resp. $\underline{T}\to S$) which fits into the following diagram 
	\begin{displaymath}
		\xymatrix{
			T\ar[r] \ar[d] & \mathfrak{T} \ar[d]^{f} \\
			S\ar[r] & \XX
		} \qquad
		\textnormal{(resp. }
		\xymatrix{
			\underline{T}\ar[r] \ar[d] & \mathfrak{T} \ar[d]^{f} \\
			S\ar[r] & \XX
		})
	\end{displaymath}
	Then there exists a unique adic morphism $g:\mathfrak{T}\to \mathfrak{T}_{S,1}(\XX)$ (resp. $g:\mathfrak{T}\to \mathfrak{T}_{S,p}(\XX)$) lifting $f$.
\end{prop}

\begin{nothing}\label{generality morphism topos}
	Let $\mathscr{C}$ and $\mathscr{D}$ be two categories, $\widehat{\mathscr{C}}$ (resp. $\widehat{\mathscr{D}}$) the category of presheaves of sets on $\mathscr{C}$ (resp. $\mathscr{D}$) and $u:\mathscr{C}\to \mathscr{D}$ a functor. We have a functor
	\begin{equation} \label{pull back presheaves}
		\widehat{u}^{*}:\widehat{\mathscr{D}}\to \widehat{\mathscr{C}}\qquad \mathscr{G}\mapsto \widehat{u}^{*}(\mathscr{G})=\mathscr{G}\circ u.
	\end{equation}
	It admits a right adjoint (\cite{SGAIV} I 5.1)
	\begin{equation}
		\widehat{u}_{*}:\widehat{\mathscr{C}}\to \widehat{\mathscr{D}}.
	\end{equation}

	Let $\mathscr{C}$ and $\mathscr{D}$ be two sites \footnote{We suppose that the site $\mathscr{C}$ is small.}. If $u:\mathscr{C}\to \mathscr{D}$ is a cocontinuous (resp. continuous) functor and $\mathscr{F}$ (resp. $\mathscr{G}$) is a sheaf on $\mathscr{C}$ (resp. $\mathscr{D}$), then $\widehat{u}_{*}(\mathscr{F})$ (resp. $\widehat{u}^{*}(\mathscr{G})$) is a sheaf on $\mathscr{D}$ (resp. $\mathscr{C}$) (\cite{SGAIV} III 1.2, 2.2).

	Let $\widetilde{\mathscr{C}}$ (resp. $\widetilde{\mathscr{D}}$) be the topos of the sheaves of sets on $\mathscr{C}$ (resp. $\mathscr{D}$) and $u:\mathscr{C}\to \mathscr{D}$ a cocontinuous functor. Then $u$ induces a morphism of topoi $g:\widetilde{\mathscr{C}}\to \widetilde{\mathscr{D}}$ defined by $g_{*}=\widehat{u}_{*}$ and $g^{*}=a\circ \widehat{u}^{*}$, where $a$ is the sheafification functor (cf. \cite{SGAIV} III 2.3).
\end{nothing}

\begin{prop}[\cite{Oy} 4.2.1] \label{lemma adjunction iso}
	Let $\mathscr{C}$ be a site, $\mathscr{D}$ a site whose topology is defined by a pretopology and $u:\mathscr{C}\to \mathscr{D}$ a functor. Assume that:
	\begin{itemize}
		\item[(i)] $u$ is fully faithful,
		\item[(ii)] $u$ is continuous and cocontinuous,
		\item[(iii)] For every object $V$ of $\mathscr{D}$, there exists a covering of $V$ in $\mathscr{D}$ of the form $\{u(U_{i})\to V\}_{i\in I}$ where $U_{i}$ is an object of $\mathscr{C}$.
	\end{itemize}
	Then the morphism of topoi $g:\widetilde{\mathscr{C}}\to \widetilde{\mathscr{D}}$ defined by $g^{*}=\widehat{u}^{*}$ and $g_{*}=\widehat{u}_{*}$ \eqref{generality morphism topos} is an equivalence of topoi.
\end{prop}

\section{Crystalline cohomology for isocrystals} \label{rel crys coh}
\begin{nothing} \label{setting crystalline}
	In this section, $\mathfrak{S}$ denotes a flat formal $\rW$-scheme of finite type and $X$ an $S$-scheme. 

	We equip $p\mathscr{O}_{\mathfrak{S}}$ with the canonical PD-structure $\gamma$. Recall that the \textit{crystalline site} $\Cris(X/\mathfrak{S})$ is defined as follows (\cite{BO} 7.17): an object is a quadraple $(U,T,\iota,\delta)$ consisting of an open subscheme $U$ of $X$, a scheme $T$ over $\mathfrak{S}_{n}$ for some integer $n\ge 1$, a closed immersion $\iota:U\to T$ and a PD-structure $\delta$ on $\Ker(\mathscr{O}_{T}\to \mathscr{O}_{U})$ compatible with $\gamma$.
	A morphism from $(U',T',\iota',\delta')$ to $(U,T,\iota,\delta)$ of $\Cris(X/\mathfrak{S})$ consists of an open immersion $U'\to U$ and an $\mathfrak{S}$-morphism $T'\to T$ compatible with $\iota', \iota$ and the PD-structures.
	A family of morphisms $\{(U_{i},T_{i})\to (U,T)\}_{i\in I}$ is a covering if each morphism $T_{i}\to T$ is an open immersion and $|T|=\bigcup_{i\in I} |T_{i}|$.
	We denote by $(X/\mathfrak{S})_{\cris}$ the topos of sheaves of sets on $\Cris(X/\mathfrak{S})$.

	The presheaf of rings defined by $(U,T) \mapsto \Gamma(T,\mathscr{O}_{T})$, is a sheaf that we denote by $\mathscr{O}_{X/\mathfrak{S}}^{\cris}$.
	For an $\mathscr{O}_{X/\SS}^{\cris}$-module $\mathscr{F}$ and an object $(U,T)$ of $\Cris(X/\SS)$, we denote by $\mathscr{F}_{T}$ the valuation of $\mathscr{F}$ at $(U,T)$ (\cite{BO} 5.1). 
\end{nothing}

\begin{definition}[\cite{BO} 6.1, \cite{Shi08} 1.8] \label{def crystal crystalline}
	(i) We say that an $\mathscr{O}_{X/\mathfrak{S}}^{\cris}$-module $\mathscr{F}$ is \textit{a crystal}, if for every morphism $f:(U',T')\to (U,T)$ of $\Cris(X/\mathfrak{S})$, the transition morphism $f^{*}(\mathscr{F}_{T})\to \mathscr{F}_{T'}$ is an isomorphism.

	(ii) We say that a crystal $\mathscr{F}$ is \textit{a crystal of $\mathscr{O}_{X/\mathfrak{S}}^{\cris}$-modules of finite presentation} if $\mathscr{F}_{T}$ is an $\mathscr{O}_{T}$-module of finite presentation for every object $(U,T)$ of $\Cris(X/\mathfrak{S})$.
	
	(iii) We denote by $\mathscr{C}(\mathscr{O}_{X/\mathfrak{S}}^{\cris})$ the category of crystals of $\mathscr{O}_{X/\mathfrak{S}}^{\cris}$-modules of finite presentation. 
	Objects of $\mathscr{C}(\mathscr{O}_{X/\mathfrak{S}}^{\cris})_{\mathbb{Q}}$ \eqref{cat up to isog} are called \textit{isocrystals}.
\end{definition}

\begin{nothing} \label{def Hopf alg}
	Crystals have an equivalent description in terms of modules equppied with hyper-PD-stratification and of modules with integrable connection. Let us briefly recall these notions. 

	Let $\XX$ be an adic formal $\SS$-scheme of finite type. 
	For any integer $r \ge 1$, let $\XX^{r+1}$ be the fiber product of $(r+1)$-copies of $\XX$ over $\SS$. 

	Let $\rG$ be an adic formal $\XX^{2}$-scheme and let $q_{1},q_{2}:\rG\to \XX$ be the canonical projections. 
	\textit{A formal $\XX$-groupoid struture over $\SS$ on $\rG$} are three adic morphisms
	\begin{equation}
	\alpha:\rG\times_{\XX}\rG\to \rG, \quad \iota:\XX\to \rG, \quad \eta:\rG\to \rG
	\end{equation}
	where the fibered product $\rG\times_{\XX}\rG$ is taken on the left (resp. right) for the $\XX$-structure defined by $q_{2}$ (resp. $q_{1}$), satisfying the compatibility conditions for groupoid (cf. \cite{Xu} 4.8).
	We set $q_{13}=\alpha$ and $q_{12},q_{23}:\rG\times_{\XX}\rG\to \rG$ the projection in the first and second component respectively. 
%
%
\end{nothing}

\begin{definition} \label{def stratification}
	Let $\rG$ be a formal $\XX$-groupoid over $\SS$ and $M$ an $\mathscr{O}_{\XX}$-module. 
	\textit{An $\mathscr{O}_{\rG}$-stratification on $M$} is an $\mathscr{O}_{\rG}$-linear isomorphism 
	\begin{equation}
		\varepsilon:q_{2}^{*}(M)\xrightarrow{\sim}q_{1}^{*}(M)
	\end{equation}
	satisfying $\iota^{*}(\varepsilon)=\id_{M}$ and the cocycle condition $q_{12}^{*}(\varepsilon)\circ q_{23}^{*}(\varepsilon)=q_{13}^{*}(\varepsilon)$. 	
\end{definition}

\begin{nothing}
	Let $f:\XX\to \SS$ be a smooth morphism, $M$ a coherent $\mathscr{O}_{\XX}$-module (resp. $\mathscr{O}_{\XX}[\frac{1}{p}]$-module) and $\widehat{\Omega}_{\XX/\SS}^{1}$ the $\mathscr{O}_{\XX}$-module of differentials of $\XX$ relative to $\SS$. \textit{A connection on $M$ relative to $\SS$} is an $f^{-1}(\mathscr{O}_{\SS})$-linear morphism
	\begin{equation}
		\nabla: M\to M\otimes_{\mathscr{O}_{\XX}}\widehat{\Omega}_{\XX/\SS}^{1}
	\end{equation}
	such that for every local sections $f$ of $\mathscr{O}_{\XX}$ and $e$ of $M$, we have $\nabla(fe)= e\otimes d(f)+ f\nabla(e)$. 

	For any $q\ge 0$, the morphism $\nabla$ extends to a unique $f^{-1}(\mathscr{O}_{\SS})$-linear morphism
\begin{equation}
	\nabla_{q}:M\otimes_{\mathscr{O}_{\XX}}\widehat{\Omega}_{\XX/\SS}^{q} \to M\otimes_{\mathscr{O}_{\XX}}\widehat{\Omega}_{\XX/\SS}^{q+1}.
\end{equation}
The composition $\nabla_{1}\circ \nabla$ is $\mathscr{O}_{\XX}$-linear. We say that $\nabla$ is \textit{integrable} if $\nabla_{1}\circ \nabla =0$. If $\nabla$ is integrable, we have $\nabla_{q+1}\circ \nabla_{q}=0$ for all $q\ge 0$ and we can associate to $(M,\nabla)$ a de Rham complex
\begin{equation}
	M\xrightarrow{\nabla} M\otimes_{\mathscr{O}_{\XX}}\widehat{\Omega}_{\XX/\SS}^{1} \xrightarrow{\nabla_{1}} M\otimes_{\mathscr{O}_{\XX}}\widehat{\Omega}_{\XX/\SS}^{2} \xrightarrow{\nabla_{2}}\cdots. \label{dR complex}
\end{equation}

When $M$ is a coherent $\mathscr{O}_{\XX}$-module, we say that an integrable connection $\nabla$ on $M$ is \textit{topologically quasi-nilpotent} if for every $n\ge 1$, its reduction modulo $p^{n}$ on $M/p^{n}M$ is quasi-nilpotent (\cite{BO} 6.1).
\end{nothing}

\begin{nothing}\label{crystal connection stratification}
	Suppose that $X$ is smooth over $S$ and admits a smooth lifting $\mathfrak{X}$ over $\mathfrak{S}$. 
	We denote by $\PP_{\XX_{n}/\SS_{n}}$ the PD-envelop of the diagonal immersion $\XX_{n}\to \XX^{2}_{n}$ compatible with the canonical PD-structure $\gamma$ \eqref{setting crystalline} for $n\ge 1$. The sequence $(\PP_{\XX_{n}/\SS_{n}})_{n\ge 1}$ forms an adic inductive $(\XX_{n}^{2})_{n\ge 1}$-system (cf. \cite{BO} 3.20.8) and we denote its inductive limit by $\PP_{\XX/\SS}$. We have $\XX_{\zar}=\PP_{\XX/\SS,\zar}$.
	By the universal property of PD-envelope, the formal $\XX^{2}$-schema $\PP_{\XX/\SS}$ is equipped with a formal $\XX$-groupoid structure \eqref{def Hopf alg}. 

	Given an object $\mathscr{F}$ of $\mathscr{C}(\mathscr{O}^{\cris}_{X/\mathfrak{S}})$ \eqref{def crystal crystalline}, the coherent $\mathscr{O}_{\XX}$-module $\mathscr{F}_{\XX}=\varprojlim_{n\ge 1}\mathscr{F}_{\XX_{n}}$ is equipped with an $\mathscr{O}_{\PP_{\XX/\SS}}$-stratification and then an integrable connection. 
	Moreover, the following categories are canonically equivalent (\cite{BO} 6.6, \cite{Stacks} 07JH):

	(i) The category $\mathscr{C}(\mathscr{O}_{X/\mathfrak{S}}^{\cris})$. 

	(ii) The category of coherent $\mathscr{O}_{\XX}$-modules equipped with an $\mathscr{O}_{\PP_{\XX/\SS}}$-stratification.

	(iii) The category of coherent $\mathscr{O}_{\XX}$-modules equipped with a topologically quasi-nilpotent integrable connection relative to $\SS$.
\end{nothing}

\begin{prop}[\cite{Ogus90} 0.7.5, \cite{Shi08} 1.23] \label{prop stratification}
	Keep the assumption of \ref{crystal connection stratification}. Let $M$ be a coherent $\mathscr{O}_{\XX}[\frac{1}{p}]$-module and $\varepsilon$ an $\mathscr{O}_{\PP_{\XX/\SS}}$-stratification on $M$. There exists a coherent $\mathscr{O}_{\XX}$-module $M^{\circ}$ and $\varepsilon^{\circ}$ an $\mathscr{O}_{\PP_{\XX/\SS}}$-stratification on $M^{\circ}$ such that $(M^{\circ}[\frac{1}{p}],\varepsilon^{\circ}\otimes\id)$ is isomorphic to $(M,\varepsilon)$. 
\end{prop}

\begin{lemma} \label{connection locally proj}
	Let $\XX$ be a smooth formal $\rW$-scheme, $M$ a coherent $\mathscr{O}_{\XX}$-module and $\nabla$ an integrable connection on $M$ relative to $\rW$. Then, $M[\frac{1}{p}]$ is a locally projective $\mathscr{O}_{\XX}[\frac{1}{p}]$-module of finite type. In particular, given a coherent $\mathscr{O}_{\XX}[\frac{1}{p}]$-module with an $\mathscr{O}_{\PP_{\XX/\rW}}$-stratification $(M,\varepsilon)$ (resp. an object $\mathscr{E}$ of $\mathscr{C}(\mathscr{O}_{X/\rW}^{\cris})$), $M$ (resp. $\mathscr{E}_{\XX}[\frac{1}{p}]$) is a locally projective $\mathscr{O}_{\XX}[\frac{1}{p}]$-module of finite type.
\end{lemma}

The first assertion is a standard result (cf. \cite{Ka71} 8.8, \cite{Ked} 1.2). Then the second assertion follows from \ref{crystal connection stratification} and \ref{prop stratification}.

\begin{nothing} \label{def relative cris coh}
	There exists a canonical functor $\pi:\Cris(X/\mathfrak{S})\to \Zar_{/X}$ defined by $(U,\mathfrak{T})\mapsto U$. It is cocontinuous and then induces a morphism of topoi (\cite{BO} 5.12)
	\begin{equation}
		u_{X/\mathfrak{S},\cris}: (X/\mathfrak{S})_{\cris}\to X_{\zar}.
	\end{equation}
	We denote by $g_{X/\mathfrak{S},\cris}$ the composition
	\begin{displaymath}
		g_{X/\mathfrak{S},\cris}: (X/\mathfrak{S})_{\cris}\to X_{\zar}\to \mathfrak{S}_{\zar},
	\end{displaymath}
	which is ringed by $\mathscr{O}_{X/\mathfrak{S}}^{\cris}$ and $\mathscr{O}_{\mathfrak{S}}$.
	We call $\rR^{\bullet}g_{X/\mathfrak{S},\cris*}(-)$ \textit{the relative crystalline cohomology}.

	For an isocrystal $\mathscr{E}=\mathscr{F}_{\mathbb{Q}}$ with $\mathscr{F} \in \Ob(\mathscr{C}(\mathscr{O}_{X/\mathfrak{S}}^{\cris}))$, we set \eqref{OT mod inverse p}
	\begin{eqnarray*}
		\rR^{q}g_{X/\mathfrak{S},\cris*}(\mathscr{E}) = \rR^{q}g_{X/\mathfrak{S},\cris*}(\mathscr{F})[\frac{1}{p}], \qquad \quad
		\rR g_{X/\mathfrak{S},\cris*}(\mathscr{E}) = \rR g_{X/\mathfrak{S},\cris*}(\mathscr{F})[\frac{1}{p}].
	\end{eqnarray*}
	It is clear that the above definition is independent of the choice of $\mathscr{F}$. 
\end{nothing}


\begin{theorem}[\cite{Shi08} 1.15] \label{finiteness crystalline coh}
	Suppose moreover that $\SS$ is separated and that $X$ is smooth and proper over $S$. Then, for any isocrystal $\mathscr{E}$ and $q\ge 0$, the relative crystalline cohomology $\rR^{q}g_{X/\mathfrak{S},\cris *} (\mathscr{E})$ is a coherent $\mathscr{O}_{\mathfrak{S}}[\frac{1}{p}]$-module.
\end{theorem}


\begin{definition}[\cite{Shi08} 1.2, 1.9] \label{notion MT locally proj}
	Let $\mathfrak{T}$ be an adic formal $\rW$-scheme. 
	
	(i) We denote by $\rM(\mathfrak{T})$ the category of projective systems of $(\mathscr{O}_{\mathfrak{T}_{n}})_{n\ge 1}$-modules $(M_{n})_{n\ge 1}$ such that each $M_{n}$ is an $\mathscr{O}_{\mathfrak{T}_{n}}$-module of finite presentation and that the canonical morphism $M_{n+1}\otimes_{\mathscr{O}_{\mathfrak{T}_{n+1}}}\mathscr{O}_{\mathfrak{T}_{n}}\xrightarrow{\sim} M_{n}$ is an isomorphism.

	(ii) We say that an object $M$ of $\rM(\mathfrak{T})_{\mathbb{Q}}$ \eqref{cat up to isog} is \textit{locally projective} if Zariski locally on $\mathfrak{T}$, $M$ is isomorphic to a direct summand of $((\mathscr{O}_{\mathfrak{T}_{n}}^{\oplus r})_{n\ge 1})_{\mathbb{Q}}$ for some integer $r$. 
\end{definition}

	If $\mathfrak{T}$ is of finite type over $\rW$, the canonical functor $(M_{n})_{n\ge 1}\mapsto \varprojlim_{n\ge 1}M_{n}$ induces an equivalence of categories $\rM(\mathfrak{T})\xrightarrow{\sim}\Coh(\mathscr{O}_{\mathfrak{T}})$.
	In this case, an object of $\rM(\mathfrak{T})_{\mathbb{Q}}$ is locally projective \eqref{notion MT locally proj} if and only if its image in $\Coh(\mathscr{O}_{\mathfrak{T}}[\frac{1}{p}])$ \eqref{equi coh mod inverse p} is locally projective of finite type \eqref{def locally proj}.

\begin{definition}[\cite{Shi08} 1.11, 1.12] \label{def crystal locally proj}
	(i) A projective system of objects $\{(U,T_{n},\iota_{n},\delta_{n})\}_{n\ge 1}$ of $\Cris(X/\mathfrak{S})$ is called \textit{a $p$-adic system}, if $\mathfrak{T}:= \varinjlim_{n} T_{n}$ is an adic formal $\mathfrak{S}$-scheme such that the canonical morphism $T_{n} \to\mathfrak{T}\otimes_{\mathbb{Z}_{p}}\mathbb{Z}/p^{n}\mathbb{Z}$ is an isomorphism.

	(ii) We say that an isocrystal $\mathscr{E}=\mathscr{F}_{\mathbb{Q}}$ \eqref{def crystal crystalline} is \textit{locally projective} if for every $p$-adic system $\mathfrak{T}=(U,\mathfrak{T}_{n})_{n\ge 1}$ of $\Cris(X/\mathfrak{S})$, $((\mathscr{F}_{\mathfrak{T}_{n}})_{n\ge 1})_{\mathbb{Q}} \in \rM(\mathfrak{T})_{\mathbb{Q}}$ is locally projective in the sense of \ref{notion MT locally proj}.
\end{definition}

\begin{theorem}[\cite{Shi08} 1.16]
	Keep the assumption of \ref{finiteness crystalline coh}. For any locally projective isocrystal $\mathscr{E}$, $\rR g_{X/\SS,\cris*}(\mathscr{E})$ is a perfect complex of $\mathscr{O}_{\SS}[\frac{1}{p}]$-modules, i.e. Zariski locally on $\SS$, it is isomorphic to a bounded complex of locally projective $\mathscr{O}_{\SS}[\frac{1}{p}]$-modules of finite type in the derived category of $\mathscr{O}_{\SS}[\frac{1}{p}]$-modules.
\end{theorem}

\begin{theorem}[\cite{Shi08} 1.19] \label{smooth base change cry coh}
	Let $\varphi:\mathfrak{S}'\to \mathfrak{S}$ be an adic morphism of seperated, flat formal $\rW$-schemes of finite type, $X$ a smooth proper $S$-scheme, $X'=X\times_{\mathfrak{S}}\mathfrak{S}'$ and $\varphi_{\cris}:(X'/\mathfrak{S}')_{\cris}\to (X/\mathfrak{S})_{\cris}$ the functorial morphism of crystalline topoi.
	Then, for any locally projective isocrystal $\mathscr{E}$ of $(X/\mathfrak{S})_{\cris}$, there exists a canonical isomorphism in the derived category of $\mathscr{O}_{\SS'}[\frac{1}{p}]$-modules:
	\begin{equation}
		\rL\varphi^{*}_{\zar}(\rR g_{X/\mathfrak{S},\cris*}(\mathscr{E}))\xrightarrow{\sim} \rR g_{X'/\mathfrak{S}',\cris *}(\varphi^{*}_{\cris}(\mathscr{E})).
	\end{equation}
\end{theorem}

\section{Convergent topos and convergent isocrystals} \label{sec conv topos}
	In this section, $\mathfrak{S}$ denotes a flat formal $\rW$-scheme of finite type and $X$ denotes an $S$-scheme.
	For any scheme $T$, we denote by $T_{0}$ the reduced subscheme of $T$.

	\begin{definition}[\cite{Ogus84} 2.1, \cite{Shi08} 2.4] \label{def conv cat}
	We define a category $\Conv(X/\SS)$ as follows.

	(i) An object of $\Conv(X/\SS)$ is a pair $(\mathfrak{T},u)$ consisting of a  formal $\SS$-scheme of finite type which is flat over $\rW$ and an $S$-morphism $u:T_{0}\to X$.

	(ii) Let $(\mathfrak{T}',u')$ and $(\mathfrak{T},u)$ be two objects of $\Conv(X/\SS)$. A morphism from $(\mathfrak{T}',u')$ to $(\mathfrak{T},u)$ is a $\SS$-morphism $f:\mathfrak{T}'\to \mathfrak{T}$ such that the induced morphism $f_{0}:T'_{0}\to T_{0}$ is compatible with $u'$ and $u$.
\end{definition}

We denote an object $(\mathfrak{T},u)$ of $\Conv(X/\SS)$ simply by $\mathfrak{T}$, if there is no risk of confusion.

It is clear that if $X\to Y$ is a nilpotent immersion of $S$-schemes, the category $\Conv(X/\mathfrak{S})$ is canonically equivalent to $\Conv(Y/\mathfrak{S})$.

\begin{nothing} \label{def fibered product}
	Let $f:(\mathfrak{T}',u')\to (\mathfrak{T},u)$ and $g:(\mathfrak{T}'',u'')\to (\mathfrak{T},u)$ be two morphisms of $\Conv(X/\SS)$. We denote by $\mathfrak{Z}$ the closed formal sub-scheme of $\mathfrak{T}'\times_{\mathfrak{T}}\mathfrak{T}''$ defined by the ideal of $p$-torsion elements of $\mathscr{O}_{\mathfrak{T}'\times_{\mathfrak{T}}\mathfrak{T}''}$.	
	The fibered product of $f$ and $g$ in $\Conv(X/\mathfrak{S})$ is represented by $\mathfrak{Z}$, which is flat over $\rW$, equipped with the composition $Z_{0}\to T'_{0}\times_{T_{0}}T''_{0}\to X$ induced by $u'$ and $u''$.

	If either $\mathfrak{T}''\to \mathfrak{T}$ or $\mathfrak{T}'\to \mathfrak{T}$ is flat, then $\mathfrak{Z}$ is equal to $\mathfrak{T}'\times_{\mathfrak{T}}\mathfrak{T}''$.
\end{nothing}

\begin{nothing} \label{def topology Conv}
	We say that a family of morphisms $\{(\mathfrak{T}_{i},u_{i})\to (\mathfrak{T},u)\}_{i\in I}$ is a Zariski (resp. fppf) covering if morphisms of formal schemes $\{\mathfrak{T}_{i}\to \mathfrak{T}\}_{i\in I}$ is a Zariski (resp. fppf) covering. 
	By \ref{recall fppf covering}, Zariski (resp. fppf) coverings form a pretopology. 
	For $\tau\in \{\zar,\fppf\}$, we denote by $(X/\SS)_{\conv,\tau}$ the topos of sheaves of sets on $\Conv(X/\SS)$, equipped with the $\tau$-topology.
\end{nothing}

\begin{rem}
	The above definition of convergent site is slightly different to that of \cite{Ogus84,Shi08} where they consider a category whose objects are triples $(\mathfrak{T},Z,u)$ where $\mathfrak{T}$ is the same as above, $Z$ is a closed subscheme of definition of $\mathfrak{T}$ such that $T_{0}\to \mathfrak{T}$ factors through $Z$ and $u:Z\to X$ is an $S$-morphism.	
	However, it follows from \ref{lemma adjunction iso} that the convergent topoi (with Zariski topology) defined by two different ways are equivalent and we freely use results of \cite{Shi08} in our setting. 
\end{rem}

	\begin{nothing} \label{construction descent data}
		Let $(\mathfrak{T},u)$ be an object of $\Conv(X/\mathfrak{S})$ and $\tau\in \{\zar,\fppf\}$. The canonical functor
	\begin{eqnarray}
		\Fft_{/\mathfrak{T}} \quad \textnormal{(resp. } \Zar_{/\mathfrak{T}}) &\to& \Conv(X/\mathfrak{S}) \\
		(f:\mathfrak{T}'\to \mathfrak{T}) &\mapsto& (\mathfrak{T}',u\circ f_{0}) \nonumber
	\end{eqnarray}
	is cocontinuous and it induces a morphism of topoi
	\begin{equation} \label{morphism of topoi sT}
		s_{\mathfrak{T}}: \mathfrak{T}_{\tau} \to (X/\mathfrak{S})_{\conv,\tau} \qquad \tau\in \{\zar,\fppf\}.
	\end{equation}
	For any sheaf $\mathscr{F}$ of $(X/\mathfrak{S})_{\conv,\tau}$, we set $\mathscr{F}_{\mathfrak{T}}=s_{\mathfrak{T}}^{*}(\mathscr{F})$.
	For any morphism $f:\mathfrak{T}'\to \mathfrak{T}$ of $\Conv(X/\mathfrak{S})$, we have a canonical morphism 
	\begin{equation}
		\beta_{f}: \mathscr{F}_{\mathfrak{T}}\to f_{\tau*}(\mathscr{F}_{\mathfrak{T}'}) \label{transition beta f}
	\end{equation}
	and we denote its adjoint by
	\begin{equation} \label{transition gamma f}
		\gamma_{f}:f_{\tau}^{*}(\mathscr{F}_{\mathfrak{T}}) \to \mathscr{F}_{\mathfrak{T}'}
	\end{equation}
	where $f_{\tau}:\mathfrak{T}_{\tau}'\to \mathfrak{T}_{\tau}$ denotes the functorial morphism for $\tau$-topology \eqref{fppf to zar and functorial}. It is clear that $\gamma_{\id}=\id$. 
	If $f$ is a morphism of $\Fft_{/\mathfrak{T}}$ (resp. $\Zar_{/\mathfrak{T}}$), $f_{\tau}$ is the localisation morphism at $\mathfrak{T}'$ \eqref{fppf to zar and functorial} and then $\gamma_{f}$ is an isomorphism. 
	If $g:\mathfrak{T}''\to \mathfrak{T}'$ is another morphism of $\Conv(X/\mathfrak{S})$, one verifies that $\gamma_{g\circ f}=\gamma_{f}\circ f^{*}_{\tau}(\gamma_{g})$.
\end{nothing}

\begin{prop} \label{descent data sheaf}
	For $\tau\in \{\zar,\fppf\}$, a sheaf $\mathscr{F}$ of $\Conv(X/\mathfrak{S})_{\conv,\tau}$ is equivalent to the following data:

	\textnormal{(i)} For every object $(\mathfrak{T},u)$ of $\Conv(X/\mathfrak{S})$, a sheaf $\mathscr{F}_{\mathfrak{T}}$ of $\mathfrak{T}_{\tau}$,

	\textnormal{(ii)} For every morphism $f:(\mathfrak{T}',u')\to (\mathfrak{T},u)$, a transition morphism $\gamma_{f}$ \eqref{transition gamma f},

	subject to the following conditions

	\textnormal{(a)} If $f$ is the identity morphism of $(\mathfrak{T},u)$, $\gamma_{f}$ is the identity morphism.

	\textnormal{(b)} If $f:\mathfrak{T}'\to \mathfrak{T}$ is a morphism of $\Zar_{/\mathfrak{T}}$ (resp. $\Fft_{/\mathfrak{T}}$), $\gamma_{f}$ is an isomorphism.

	\textnormal{(c)} If $f$ and $g$ are two composable morphisms, then we have $\gamma_{g\circ f}=\gamma_{f}\circ f^{*}_{\tau}(\gamma_{g})$.
\end{prop}
\begin{proof}
	Given a data $\{\mathscr{F}_{\mathfrak{T}},\gamma_{f}\}$ as in the proposition, for any morphism $f:\mathfrak{T}'\to \mathfrak{T}$ of $\Conv(X/\mathfrak{S})$, the morphism $\gamma_{f}$ induces a morphism
	\begin{displaymath}
		\mathscr{F}_{\mathfrak{T}}(\mathfrak{T})\to \mathscr{F}_{\mathfrak{T}'}(\mathfrak{T}').
	\end{displaymath}
	In view of conditions (a) and (c), the correspondence
\begin{displaymath}
	\mathfrak{T}\mapsto \mathscr{F}_{\mathfrak{T}}(\mathfrak{T})
\end{displaymath}
defines a presheaf $\mathscr{F}$ on $\Conv(X/\mathfrak{S})$. In view of condition (b), $\mathscr{F}$ is a sheaf and the above construction is quasi-inverse to \ref{construction descent data}. Then the proposition follows.
\end{proof}

\begin{nothing} \label{def alpha}
	Note that the fppf topology on $\Conv(X/\SS)$ is finer than the Zariski topology. Equipped with the fppf topology on the source and the Zariski topology on the target, the identical functor $\id:\Conv(X/\SS)\to \Conv(X/\SS)$ is cocontinuous. By \ref{generality morphism topos}, it induces a morphism of topoi
	\begin{equation} \label{morphism of topoi fppf zar}
	\alpha: (X/\SS)_{\conv,\fppf}\to (X/\SS)_{\conv,\zar}.
\end{equation}
If $\mathscr{F}$ is a sheaf of $(X/\SS)_{\conv,\fppf}$, $\alpha_{*}(\mathscr{F})$ is equal to $\mathscr{F}$ as presheaves. If $\mathscr{G}$ is a sheaf of $(X/\SS)_{\conv,\zar}$, then $\alpha^{*}(\mathscr{G})$ is the sheafification of $\mathscr{G}$ with respect to the fppf topology.
\end{nothing}

\begin{nothing} \label{def functorial functor}
	Let $g:\SS'\to \SS$ be a morphism of $\MS^{\diamond}$ \eqref{def cat MS}, $X'$ an $S'$-scheme and $f:X'\to X$ a morphism compatible with $g$, i.e. the diagram
	\begin{equation}
		\xymatrix{
			X'\ar[r] \ar[d]_{f} & S'\ar[r] \ar[d] & \SS' \ar[d]^{g}\\
			X \ar[r] & S \ar[r]& \SS
		}
	\end{equation}
	is commutative. 
	For any object $(\mathfrak{T},u)$ of $\Conv(X'/\SS')$, $(\mathfrak{T},f\circ u)$ defines an object of $\Conv(X/\SS)$.
	We obtain a functor that we denote by
	\begin{equation} \label{functorial functor}
		\varphi: \Conv(X'/\SS')\to \Conv(X/\SS), \qquad (\mathfrak{T},u) \mapsto (\mathfrak{T},f\circ u).
	\end{equation}
	It is clear that $\varphi$ commutes with the fibered product \eqref{def fibered product}.
\end{nothing}

\begin{lemma} \label{cont cocont func}
	\textnormal{(i)} Let $(\mathfrak{T},u)$ be an object of $\Conv(X'/\SS')$ and $g:(\mathfrak{Z},w)\to \varphi(\mathfrak{T},u)$ a morphism of $\Conv(X/\SS)$. Then there exists an object $(\mathfrak{Z},v)$ of $\Conv(X'/\SS')$ and a morphism $h:(\mathfrak{Z},v)\to (\mathfrak{T},u)$ of $\Conv(X'/\SS')$ such that $g=\varphi(h)$.

	\textnormal{(ii)} Equipped with the Zariski topology (resp. fppf topology) \eqref{def topology Conv} on both sides, the functor $\varphi$ is continuous and cocontinuous.
\end{lemma}

\begin{proof}
	(i) By considering compositions $\mathfrak{Z}\to \mathfrak{T}\to \SS'$ and $Z_{0}\to T_{0}\to X'$, we obtain an object $(\mathfrak{Z},v)$ of $\Conv(X'/\SS)$ and a morphism $h: (\mathfrak{Z},v)\to (\mathfrak{T},u)$ of $\Conv(X'/\SS')$ such that $g=\varphi(h)$.

	(ii) A family of morphisms $\{(\mathfrak{T}_{i},u_{i})\to (\mathfrak{T},u)\}_{i\in I}$ of $\Conv(X'/\SS')$ belongs to $\Cov_{\zar}(\mathfrak{T},u)$ (resp. $\Cov_{\fppf}(\mathfrak{T},u)$) if and only if, its image by $\varphi$ belongs to $\Cov_{\zar}(\varphi(\mathfrak{T},u))$ (resp. $\Cov_{\fppf}(\varphi(\mathfrak{T},u))$).
Since $\varphi$ commutes with the fibered product, the continuity of $\varphi$ follows from (\cite{SGAIV} III 1.6).

Let $\{(\mathfrak{T}_{i},u_{i})\to \varphi(\mathfrak{T},u)\}_{i\in I}$ be an element of $\Cov_{\zar}(\varphi(\mathfrak{T},u))$ (resp. $\Cov_{\fppf}(\varphi(\mathfrak{T},u))$). By (i), there exists an element $\{(\mathfrak{T}_{i},u_{i})\to (\mathfrak{T},u)\}_{i\in I}$ of $\Cov_{\zar}(\mathfrak{T},u)$ (resp. $\Cov_{\fppf}(\mathfrak{T},u)$) mapping by $\varphi$ to the given element. Then, $\varphi$ is cocontinuous by \textnormal{(\cite{SGAIV} III 2.1)}.
\end{proof}

\begin{nothing} \label{setting functorial morphism}
	By \ref{generality morphism topos} and \ref{cont cocont func}, the functor $\varphi$ \eqref{functorial functor} induces morphisms of topoi
	\begin{eqnarray}
		\label{morphism of topoi functorial} f_{\conv,\tau}:(X'/\SS')_{\conv,\tau}\to (X/\SS)_{\conv,\tau},
	\end{eqnarray}
	such that the pullback functor is induced by the composition with $\varphi$. For a sheaf $\mathscr{F}$ of $(X/\mathfrak{S})_{\conv,\tau}$ and an object $\mathfrak{T}$ of $\Conv(X'/\mathfrak{S}')$, we have \eqref{descent data sheaf}
	\begin{equation}
		(f_{\conv,\tau}^{*}(\mathscr{F}))_{\mathfrak{T}}=\mathscr{F}_{\varphi(\mathfrak{T})}.\qquad \label{description of pullback}
	\end{equation}
	For any morphism $g$ of $\Conv(X'/\mathfrak{S}')$, the transition morphism of $f_{\conv,\tau}^{*}(\mathscr{F})$ associated to $g$ \eqref{descent data sheaf} is equal to the transition morphism of $\mathscr{F}$ associated to $\varphi(g)$.
	
	By considering inverse image functors, one verifies that the following diagram commutes \eqref{morphism of topoi fppf zar}
	\begin{equation} \label{functorial map sigma fppf Zar}
		\xymatrixcolsep{5pc}\xymatrix{
			(X'/\SS')_{\conv,\fppf} \ar[r]^{f_{\conv,\fppf}} \ar[d]_{\alpha'}& (X/\SS)_{\conv,\fppf} \ar[d]^{\alpha} \\
			(X'/\SS')_{\conv,\zar} \ar[r]^{f_{\conv,\zar}} & (X/\SS)_{\conv,\zar}}
	\end{equation}
\end{nothing}

\begin{nothing}
	We define a presheaf of rings $\mathscr{O}_{X/\mathfrak{S}}[\frac{1}{p}]$ on $\Conv(X/\mathfrak{S})$ by
	\begin{equation}
		(\mathfrak{T},u)\mapsto \Gamma(\mathfrak{T},\mathscr{O}_{\mathfrak{T}}[\frac{1}{p}]).
	\end{equation}
	By fppf descent \eqref{descent fppf}, $\mathscr{O}_{X/\mathfrak{S}}[\frac{1}{p}]$ is a sheaf for the fppf topology. Since the fppf topology is finer than the Zariski topology, it is also a sheaf for the Zariski topology.

	For any object $(\mathfrak{T},u)$ of $\Conv(X/\mathfrak{S})$, we have $(\mathscr{O}_{X/\mathfrak{S}}[\frac{1}{p}])_{\mathfrak{T}}=\mathscr{O}_{\mathfrak{T}}[\frac{1}{p}]$. 
	If $\mathscr{F}$ is an $\mathscr{O}_{X/\mathfrak{S}}[\frac{1}{p}]$-module of $(X/\SS)_{\conv,\tau}$, $\mathscr{F}_{\mathfrak{T}}$ is an $\mathscr{O}_{\mathfrak{T}}[\frac{1}{p}]$-module of $\mathfrak{T}_{\tau}$. For any morphism $f:\mathfrak{T}'\to \mathfrak{T}$ of $\Conv(X/\mathfrak{S})$, the transition morphism $\gamma_{f}$ \eqref{descent data sheaf} extends to an $\mathscr{O}_{\mathfrak{T}'}[\frac{1}{p}]$-linear morphism \eqref{coherent fppf}
	\begin{equation} \label{lin transition morphism}
		c_{f}: f^{*}_{\tau}(\mathscr{F}_{\mathfrak{T}})\to \mathscr{F}_{\mathfrak{T}'}.
	\end{equation}
	
	In view of \ref{descent data sheaf}, we deduce the following description for $\mathscr{O}_{X/\SS}[\frac{1}{p}]$-modules.
\end{nothing}

\begin{prop} \label{description of mods}
		For $\tau\in\{\zar,\fppf\}$, an $\mathscr{O}_{X/\mathfrak{S}}[\frac{1}{p}]$-module of $(X/\mathfrak{S})_{\conv,\tau}$ is equivalent to the following data:

		\textnormal{(i)} For every object $\mathfrak{T}$ of $\Conv(X/\mathfrak{S})$, an $\mathscr{O}_{\mathfrak{T}}[\frac{1}{p}]$-module $\mathscr{F}_{\mathfrak{T}}$ of $\mathfrak{T}_{\tau}$,

		\textnormal{(ii)} For every morphism $f:\mathfrak{T}'\to \mathfrak{T}$ of $\Conv(X/\mathfrak{S})$, an $\mathscr{O}_{\mathfrak{T}'}$-linear morphism $c_{f}$ \eqref{lin transition morphism}.

	which is subject to the following conditions

	\textnormal{(a)} If $f$ is the identity morphism, then $c_{f}$ is the identity.

	\textnormal{(b)} If $f:\mathfrak{T}'\to \mathfrak{T}$ is a morphism of f $\Zar_{/\mathfrak{T}}$ (resp. $\Fft_{/\mathfrak{T}}$), then $c_{f}$ is an isomorphism.

	\textnormal{(c)} If f and g are two composable morphisms, then we have $c_{g\circ f}=c_{f}\circ f^{*}_{\tau}(c_{g})$.
\end{prop}
\begin{definition} \label{def convergent crystals}
	Let $\mathscr{F}$ be an $\mathscr{O}_{X/\mathfrak{S}}[\frac{1}{p}]$-module of $(X/\mathfrak{S})_{\conv,\tau}$.

	\textnormal{(i)} We say that $\mathscr{F}$ is \textit{coherent} if for every object $\mathfrak{T}$ of $\Conv(X/\mathfrak{S})$, $\mathscr{F}_{\mathfrak{T}}$ is a coherent \eqref{coherent fppf}.

	\textnormal{(ii)} We say that $\mathscr{F}$ is a \textit{crystal} if for every morphism $f$ of $\Conv(X/\mathfrak{S})$, $c_{f}$ is an isomorphism.
\end{definition}

With the notation of \ref{setting functorial morphism}, $f_{\conv,\tau}^{*}$ sends coherent $\mathscr{O}_{X/\SS}[\frac{1}{p}]$-modules (resp. crystals) to coherent $\mathscr{O}_{X'/\SS'}[\frac{1}{p}]$-modules (resp. crystals).

\begin{nothing} \label{coh fppf descent}
	Let $\mathscr{E}$ be a coherent crystal of $\mathscr{O}_{X/\mathfrak{S}}[\frac{1}{p}]$-modules of $(X/\mathfrak{S})_{\conv,\zar}$.
	For any object $\mathfrak{T}$ of $\Conv(X/\mathfrak{S})$, $\alpha^{*}(\mathscr{E})_{\mathfrak{T}}$ \eqref{morphism of topoi fppf zar} is the fppf sheaf associated to the presheaf
	\begin{equation}
		(f: \mathfrak{T}'\to \mathfrak{T}) \mapsto \Gamma(\mathfrak{T}',f^{*}_{\zar}(\mathscr{E}_{\mathfrak{T}}) ).
	\end{equation}
	Since $\mathscr{E}$ is a coherent crystal, $\alpha^{*}(\mathscr{E})$ is equal to $\mathscr{E}$ as presheaves on $\Conv(X/\mathfrak{S})$ by fppf descent.

	By \ref{def alpha}, the direct image and inverse image functors of $\alpha$ induce an equivalence between the category of coherent crystals of $\mathscr{O}_{X/\mathfrak{S}}[\frac{1}{p}]$-modules of $(X/\mathfrak{S})_{\conv,\zar}$ and the category of coherent crystals of $\mathscr{O}_{X/\mathfrak{S}}[\frac{1}{p}]$-modules of $(X/\mathfrak{S})_{\conv,\fppf}$.

	Following \cite{Ogus84,Shi08}, coherent crystal of $\mathscr{O}_{X/\mathfrak{S}}[\frac{1}{p}]$-modules of $(X/\SS)_{\conv,\tau}$ are called \textit{convergent isocrystals of $(X/\SS)_{\conv,\tau}$}. We denote the full category of $\mathscr{O}_{X/\mathfrak{S}}[\frac{1}{p}]$-modules consisting of these objects by $\Iso^{\dagger}(X/\mathfrak{S})$. 

	We say that a convergent isocrystal $\mathscr{F}$ is \textit{locally projective} if for every object $\mathfrak{T}$ of $\Conv(X/\mathfrak{S})$, $\mathscr{F}_{\mathfrak{T}}$ is locally projective of finite type \eqref{def locally proj}. In this case, $\mathscr{F}$ is locally projective $\mathscr{O}_{X/\SS}^{\cris}$-module of finite type.
\end{nothing}

\begin{prop}[\cite{Ogus90} 0.7.2, \cite{Shi08} 2.35] \label{convergent to isocry}
	Suppose that $X$ is smooth over $S$. There exists a canonical functor
	\begin{equation}
		\iota: \Iso^{\dagger}(X/\mathfrak{S})\to \mathscr{C}(\mathscr{O}_{X/\mathfrak{S}}^{\cris})_{\mathbb{Q}}.
	\end{equation}
\end{prop}

We briefly review the construction of the functor in the case where $X$ is separated and admits a smooth lifting $\XX$ over $\SS$. To do this, we need a formal $\XX$-groupoid constructed by admissible blow-up.

We set $\QQ_{\XX/\SS}=\mathfrak{T}_{X,p}(\XX^{2})$ the dilatation of the diagonal immersion $X\to \XX^{2}$ \eqref{def dilatation}. 
By \ref{univ dilatation}, the canonical morphism $\underline{\QQ_{\XX/\SS,1}}\to X\times_{S}X$ factors through the diagonal immersion. Then $\QQ_{\XX/\SS}$ defines an object of $\Conv(X/\SS)$.
Using its universal property \eqref{univ dilatation}, one verifies that $\QQ_{\XX/\SS}$ is equipped with a formal $\XX$-groupoid structure \eqref{def Hopf alg} (cf. \cite{Xu} 4.12).


Let $\mathscr{E}$ be an object of $\Iso^{\dagger}(X/\mathfrak{S})$. The canonical morphisms $p_{1},p_{2}:\QQ_{\XX/\SS}\to \XX$ give rise to morphisms of $\Conv(X/\SS)$. Since $\mathscr{E}$ is a crystal, $p_{1},p_{2}$ induce isomorphisms 
\begin{equation}
	p_{2}^{*}(\mathscr{E}_{\XX})\xrightarrow[c_{p_{2}}]{\sim} \mathscr{E}_{\QQ_{\XX/\SS}} \xleftarrow[c_{p_{1}}]{\sim} p_{1}^{*}(\mathscr{E}_{\XX}).
\end{equation}
By a standard argument, one verifies that $\varepsilon=c_{p_{1}}^{-1}\circ c_{p_{2}}$ defines an $\mathscr{O}_{\QQ_{\XX/\SS}}$-stratification on $\mathscr{E}_{\XX}$ \eqref{def stratification}.


Using \ref{univ dilatation}, the canonical morphism $\PP_{\XX/\SS}\to \XX^{2}$ induces a morphism of formal $\XX$-groupoids $\PP_{\XX/\SS}\to \QQ_{\XX/\SS}$ (cf \cite{Xu} 5.7).
By taking inverse image, we obtain an $\mathscr{O}_{\PP_{\XX/\SS}}$-strafitifcation on $\mathscr{E}_{\XX}$ and then a crystal of $\mathscr{O}_{X/\mathfrak{S}}^{\cris}$-modules of finite presentation up to isogeny by \ref{crystal connection stratification} and \ref{prop stratification}. The construction is clearly functorial.

\begin{nothing}\label{def uXS conv}
	There exists a canonical morphism of topoi $u_{X/\SS}:(X/\SS)_{\conv,\zar}\to X_{\zar}$ (cf. \cite{Ogus90} \S 4). 
	Suppose that $\SS$ is separated and that $X$ admits a smooth lifting $f: \XX \to \SS$. 
	Let $\mathscr{E}$ be a convergent isocrystal of $(X/\SS)_{\conv,\zar}$. By \ref{crystal connection stratification} and \ref{prop stratification}, there exists an integrable connection on $\mathscr{E}_{\XX}$ relative to $\SS$ and we denote by $\mathscr{E}_{\XX}\otimes_{\mathscr{O}_{\XX}}\widehat{\Omega}_{\XX/\SS}^{\bullet}$ the associated de Rham complex. Then there exists a canonical isomorphism in the derived category $\rD(\XX_{\zar},f^{-1}(\mathscr{O}_{\SS}[\frac{1}{p}]))$ (\cite{Shi08} 2.33)
	\begin{equation} \label{direct image u dR}
		\rR u_{X/\SS*}(\mathscr{E})\xrightarrow{\sim} \mathscr{E}_{\XX}\otimes_{\mathscr{O}_{\XX}}\widehat{\Omega}_{\XX/\SS}^{\bullet}.
	\end{equation}

	Based on the above isomorphism, Shiho showed the following results.
\end{nothing}


\begin{theorem}[\cite{Shi08} 2.36] \label{coh rig coh cristalline}
	Assume that $\SS$ is separated and that $X$ is smooth and proper over $S$.
	Let $\mathscr{E}$ be a convergent isocrystal of $(X/\SS)_{\conv,\zar}$ and $g_{\conv,\zar}:(X/\SS)_{\conv,\zar}\to (S/\SS)_{\conv,\zar}$ the functorial morphism. Then there exists a canonical isomorphism in the derived category of $\mathscr{O}_{\SS}[\frac{1}{p}]$-modules \textnormal{(\ref{def relative cris coh}, \ref{convergent to isocry})}
	\begin{equation}
		(\rR g_{\conv,\zar*}(\mathscr{E}))_{\mathfrak{S}}\xrightarrow{\sim} \rR g_{X/\mathfrak{S},\cris*}(\iota(\mathscr{E})).
	\end{equation}
	In particular, $(\rR^{i} g_{\conv,\zar*}(\mathscr{E}))_{\mathfrak{S}}$ is coherent for any $i\ge 0$. 
\end{theorem}

\begin{coro}[\cite{Shi08} 2.37]
	Keep the assumption of \ref{coh rig coh cristalline} and suppose moreover that $\mathscr{E}$ is locally projective \eqref{coh fppf descent}. Then, $(\rR g_{\conv,\zar*}(\mathscr{E}))_{\mathfrak{S}}$ is a perfect complex of $\mathscr{O}_{\SS}[\frac{1}{p}]$-modules. 
\end{coro}

\begin{coro}[\cite{Shi08} 2.38] \label{coro base change conv}
	Let $\varphi:\mathfrak{S}'\to \mathfrak{S}$ be an adic morphism of separated flat formal $\rW$-schemes of finite type, $X'=X\times_{\mathfrak{S}}\mathfrak{S}'$ and $\varphi_{\conv,\zar}:(X'/\mathfrak{S}')_{\conv,\zar}\to (X/\mathfrak{S})_{\conv,\zar}$ the functorial morphism of convergent topoi \eqref{morphism of topoi functorial}.	
	Then, for a locally projective convergent isocrystal $\mathscr{E}$ of $(X/\mathfrak{S})_{\conv,\zar}$, we have a canonical isomorphism in the derived category of $\mathscr{O}_{\SS'}[\frac{1}{p}]$-modules:
	\begin{equation}
		\rL\varphi_{\zar}^{*}(\rR g_{\conv,\zar*}(\mathscr{E}))\xrightarrow{\sim} \rR g'_{\conv,\zar*}(\varphi^{*}_{\conv,\zar}(\mathscr{E})).
	\end{equation}
\end{coro}

\section{Higher direct images of a convergent isocrystal are $p$-adically convergent} \label{Higher direct image}
\begin{nothing}\label{basic setting X to Y}
	In this section, we keep the notation of \S~\ref{sec conv topos} and let $g:X\to Y$ denote a morphism of $S$-schemes. 
	
	Let $\mathfrak{T}$ be an object of $\Conv(Y/\mathfrak{S})$ and $\tau\in \{\zar,\fppf\}$. By fppf descent for morphisms of formal $\rW$-schemes (\cite{Ab10} 5.12.1), the presheaf associated to $\mathfrak{T}$ is a sheaf for the fppf (resp. Zariski) topology that we denote by $\widetilde{\mathfrak{T}}$.
	We set $X_{T_{0}}=X\times_{Y}T_{0}$ and we denote by 
	\begin{eqnarray*}
		&g_{X/\mathfrak{T},\tau}:&(X_{T_{0}}/\mathfrak{T})_{\conv,\tau}\to (T_{0}/\mathfrak{T})_{\conv,\tau},\\
		&\omega_{\mathfrak{T}}:&(X_{T_{0}}/\mathfrak{T})_{\conv,\tau} \to (X/\mathfrak{S})_{\conv,\tau}
	\end{eqnarray*}
	the functorial morphisms of topoi \eqref{morphism of topoi functorial}.
\end{nothing}

\begin{lemma}[\cite{Ber} V 3.2.2] \label{localisation by a sheaf} 
	There exists a canonical equivalence of topoi:
	\begin{equation} \label{localisation of conv topos}
		(X/\mathfrak{S})_{\conv,\tau/g^{*}_{\conv,\tau}(\widetilde{\mathfrak{T}})}\xrightarrow{\sim} (X_{T_{0}}/\mathfrak{T})_{\conv,\tau}
	\end{equation}
	which identifies the localisation morphism and $\omega_{\mathfrak{T}}$.
\end{lemma}
The lemma can be verified in the same way as (\cite{Ber} V 3.2.2).


\begin{lemma}[\cite{Ber} V 3.2.3] \label{evaluation direct image}
	For any $\mathscr{O}_{X/\mathfrak{S}}[\frac{1}{p}]$-module $E$ of $(X/\mathfrak{S})_{\conv,\tau}$, there exists a canonical isomorphism in $\rD^{+}(\mathfrak{T}_{\tau},\mathscr{O}_{\mathfrak{T}}[\frac{1}{p}])$
	\begin{equation} \label{calcul of higher direct image evaluation}
		(\rR g_{\conv,\tau*}(E))_{\mathfrak{T}} \xrightarrow{\sim} (\rR g_{X/\mathfrak{T},\tau*}(\omega_{\mathfrak{T}}^{*}(E)))_{\mathfrak{T}}.
	\end{equation}
\end{lemma}
\begin{proof}
	Let $E$ be an abelian sheaf of $(X/\mathfrak{S})_{\conv,\tau}$ and $f:\mathfrak{T}'\to \mathfrak{T}$ a morphism of $\Conv(Y/\mathfrak{S})$. The morphism $f$ induces a functorial morphism of topoi $\varphi:(X_{T'_{0}}/\mathfrak{T}')_{\conv,\tau}\to (X_{T_{0}}/\mathfrak{T})_{\conv,\tau}$ which fits into the following commutative diagram
	\begin{equation} \label{diag funct morphisms of topoi}
		\xymatrix{
			(X_{T'_{0}}/\mathfrak{T}')_{\conv,\tau} \ar[r]^{\varphi} \ar[d]_{g_{X/\mathfrak{T}',\tau}} & (X_{T_{0}}/\mathfrak{T})_{\conv,\tau} \ar[d]^{g_{X/\mathfrak{T},\tau}} \\
			(T_{0}'/\mathfrak{T}')_{\conv,\tau} \ar[r]^{f_{\conv,\tau}} & (T_{0}/\mathfrak{T})_{\conv,\tau} }
	\end{equation}
	We have $\omega_{\mathfrak{T}'}=\omega_{\mathfrak{T}}\circ\varphi$. By \ref{localisation by a sheaf}, $\varphi$ (resp. $f_{\conv,\tau}$) coincides with the localisation morphism on the sheaf $g_{X/\mathfrak{T},\tau}^{*}(\widetilde{\mathfrak{T}}')$ (resp. $\widetilde{\mathfrak{T}}'$). 
	Then $g_{X/\mathfrak{T},\tau*}(\omega_{\mathfrak{T}}^{*}(E))$ is the sheaf associated to the presheaf on $\Conv(T_{0}/\mathfrak{T})$
	\begin{displaymath}
		(f:\mathfrak{T}'\to \mathfrak{T})\mapsto \Gamma((X_{T'_{0}}/\mathfrak{T}')_{\conv,\tau}, \omega_{\mathfrak{T}'}^{*}(E)).
	\end{displaymath}

	The sheaf $(g_{\conv,\tau*}(E))_{\mathfrak{T}}$ is associated to the presheaf 
	\begin{displaymath}
		(f:\mathfrak{T}'\to \mathfrak{T}) \mapsto \Gamma((X/\mathfrak{S})_{\conv,\tau / g_{\conv,\tau}^{*}(\widetilde{\mathfrak{T}}')}, E|_{g_{\conv,\tau}^{*}(\widetilde{\mathfrak{T}}')}).
	\end{displaymath}

%
	By \ref{localisation by a sheaf}, we deduce a canonical isomorphism of $\mathfrak{T}_{\tau}$
	\begin{equation} \label{calcul of higher direct image evaluation sheaf}
		(g_{\conv,\tau*}(E))_{\mathfrak{T}}\xrightarrow{\sim} (g_{X/\mathfrak{T},\tau*}(\omega_{\mathfrak{T}}^{*}(E)))_{\mathfrak{T}}.
	\end{equation}

	Since $\omega_{\mathfrak{T}}$ coincides with a localisation morphism, if $I^{\bullet}$ is an injective resolution of $E$, $\omega_{\mathfrak{T}}^{*}(I^{\bullet})$ is an injective resolution of $\omega_{\mathfrak{T}}^{*}(E)$. Then the isomorphism \eqref{calcul of higher direct image evaluation} follows from \eqref{calcul of higher direct image evaluation sheaf}.
\end{proof}

\begin{rem} \label{functorial evaluation direct image}
	Keep the notation of \ref{evaluation direct image}. 
	Let $f:\mathfrak{T}'\to \mathfrak{T}$ be a morphism of $\Conv(Y/\mathfrak{S})$. It induces morphisms of topoi \eqref{diag funct morphisms of topoi}. We consider canonical morphisms
	\begin{equation}
		f_{\tau}^{*} (\rR^{i} g_{X/\mathfrak{T},\tau*}(\omega_{\mathfrak{T}}^{*}(E))_{\mathfrak{T}}) 
		\to (f_{\conv,\tau}^{*} (\rR^{i} g_{X/\mathfrak{T},\tau*}(\omega_{\mathfrak{T}}^{*}(E)))_{\mathfrak{T}'}
		\xrightarrow{\sim} \rR^{i} g_{X/\mathfrak{T}',\tau*}(\omega_{\mathfrak{T}'}^{*}(E))_{\mathfrak{T}'}
	\end{equation}
	where the first morphism is the transition morphism of $\rR^{i} g_{X/\mathfrak{T},\tau*}(\omega_{\mathfrak{T}}^{*}(E))$ associated to $f$ and the second one an isomorphism because $\omega_{\mathfrak{T}'}=\omega_{\mathfrak{T}}\circ \varphi,\varphi,f_{\conv,\tau}$ are localisation morphisms (\cite{SGAIV} V 5.1). 
	
	In view of the proof of \ref{evaluation direct image}, via \eqref{calcul of higher direct image evaluation}, the above composition is compatible with the transition morphism of $\rR^{i}g_{\conv,\tau*}(E)$ associated to $f$
	\begin{equation}
		f_{\tau}^{*}((\rR^{i} g_{\conv,\tau*}(E))_{\mathfrak{T}}) \to (\rR^{i} g_{\conv,\tau*}(E))_{\mathfrak{T}'}.
	\end{equation}
\end{rem}

\begin{coro}[\cite{BBM} 1.1.19] \label{coh zar fppf}
	Let $\mathscr{E}$ be a convergent isocrystal of $(X/\mathfrak{S})_{\conv,\fppf}$. We have \eqref{def alpha}
	\begin{equation}
		\rR^{i}\alpha_{*}(\mathscr{E})=0,\qquad \forall~ i\ge 1.
	\end{equation}
\end{coro}
The assertion can be verified in the same way as (\cite{BBM} 1.1.19).

%

\begin{coro} \label{base change coro}
	Let $\SS'\to \SS$ be a morphism of $\MS^{\diamond}$ \eqref{def cat MS}, $Y'$ an $S'$-scheme and $h:Y'\to Y$ a morphism compatible with $S'\to S$. We set $X'=X\times_{Y}Y'$ and we denote by $g':X'\to Y'$ and $h':X'\to X$ the canonical morphisms: 
	\begin{displaymath}
		\xymatrix{
			X'\ar[r]^{h'} \ar[d]_{g'} \ar@{}[dr]|{\Box} & X \ar[d]^{g}\\
			Y'\ar[r]^{h} & Y
		}
	\end{displaymath}
	Then, for any $\mathscr{O}_{X/\SS}[\frac{1}{p}]$-module $E$ of $(X/\SS)_{\conv,\tau}$, the base change morphism
	\begin{equation} \label{base change morphism conv coh}
		h_{\conv,\tau}^{*}(\rR g_{\conv,\tau*}(E)) \xrightarrow{\sim} \rR g'_{\conv,\tau*}(h'^{*}_{\conv,\tau}(E)),
	\end{equation}
	is an isomorphism.
\end{coro}
\begin{proof}
	Let $\mathfrak{T}$ be an object of $\Conv(Y'/\SS')$. We denotes abusively the image of $\mathfrak{T}$ in $\Conv(Y/\SS)$ by $\mathfrak{T}$. We set $X_{T_{0}}=T_{0}\times_{Y}X(=T_{0}\times_{Y'}X')$. By applying \ref{evaluation direct image} to $g$ and $g'$, one verifies that the valuation of two sides of \eqref{base change morphism conv coh} at $\mathfrak{T}$ are both isomorphic to $(\rR g_{X/\mathfrak{T},\tau*}(\omega_{\mathfrak{T}}^{*}(E)))_{\mathfrak{T}}$ and that \eqref{base change morphism conv coh} induces an isomorphism between them. Then the assertion follows. 
\end{proof}

%

\begin{nothing} \label{padic enlargements}
	In the following, we consider the case where $\mathfrak{S}=\Spf(\rW)$ and $X,Y$ are schemes over $S=\Spec(k)$. 
	We denote by $\pConv(X/\rW)$ the full subcategory of $\Conv(X/\rW)$ consisting of objects $(\mathfrak{T},u)$ such that $u$ can be lifted to a $k$-morphism $\widetilde{u}:T\to X$. Given an object $\mathfrak{T}$ of $\pConv(X/\rW)$ and a morphism $f:\mathfrak{T}'\to \mathfrak{T}$ of $\Conv(X/\rW)$, then $\mathfrak{T}'$ is still an object of $\pConv(X/\rW)$.
	Objects of $\pConv(X/\rW)$ are called ``$p$-adic enlargements'' in \cite{Ogus84}.
\end{nothing}

\begin{lemma}\label{crystal base smooth locally free}
	Suppose that $X$ is smooth over $k$. 
	A convergent isocrystal $\mathscr{E}$ of $(X/\rW)_{\conv,\zar}$ is locally projective \eqref{coh fppf descent}. 
\end{lemma}
\begin{proof}
	The question being local, we may assume that $X$ admits a smooth lifting $\XX$ over $\rW$. By \ref{connection locally proj} and \ref{convergent to isocry}, $\mathscr{E}_{\XX}$ is locally projective. Since every object $\mathfrak{T}$ of $\Conv(X/\rW)$ locally admits a morphism to $\XX$, we deduce that $\mathscr{E}$ is locally projective.
\end{proof}

\begin{prop} \label{GM connection}
	Suppose that $Y$ is smooth over $k$ and admits a smooth lifting $\YY$ over $\rW$ and that $g:X\to Y$ is smooth and proper. Let $\mathscr{E}$ be a convergent isocrystal of $(X/\rW)_{\conv,\zar}$. Then there exists an $\mathscr{O}_{\PP_{\YY/\SS}}$-stratification on $(\rR^{i}g_{\conv,\zar*}(\mathscr{E}))_{\YY}$. In particular the later is locally projective of finite type \eqref{def locally proj}. 
\end{prop}
\begin{proof}
	We take again the notation of \ref{convergent to isocry} for $\YY\to \Spf(\rW)$ and we set $\ZZ=\QQ_{\YY/\rW}$ that we consider as an object of $\Conv(Y/\rW)$, and $\mathscr{F}=\rR^{i}g_{\conv,\zar*}(\mathscr{E})$. 
	By \ref{evaluation direct image}, we have canonical isomorphisms
	\begin{equation}
		\mathscr{F}_{\mathfrak{Y}}\xrightarrow{\sim} (\rR^{i}g_{X/\mathfrak{Y},\zar*}(\omega_{\mathfrak{Y}}^{*}(\mathscr{E})))_{\mathfrak{Y}}, \quad
		\mathscr{F}_{\mathfrak{Z}}\xrightarrow{\sim} (\rR^{i}g_{X/\mathfrak{Z},\zar*}(\omega_{\mathfrak{Z}}^{*}(\mathscr{E})))_{\mathfrak{Z}}.
	\end{equation}
	By \ref{coh rig coh cristalline}, $\mathscr{F}_{\YY}$ is coherent. 
	The projections $p_{1},p_{2}:\mathfrak{Z}\to \YY$ define two morphisms of $\Conv(Y/\rW)$ and induce two morphisms of topoi
	\begin{eqnarray*}
		(X_{Z_{0}}/\mathfrak{Z})_{\conv,\zar}\xrightarrow{\sim} (X\times_{Y,p_{1}}Z/\mathfrak{Z})_{\conv,\zar}\to (X/\YY)_{\conv,\zar}, \\
		(X_{Z_{0}}/\mathfrak{Z})_{\conv,\zar}\xrightarrow{\sim} (X\times_{Y,p_{2}}Z/\mathfrak{Z})_{\conv,\zar}\to (X/\YY)_{\conv,\zar}.
	\end{eqnarray*}
	Since $X$ is smooth over $k$, $\mathscr{E}$ is locally projective by \ref{crystal base smooth locally free}. 	
	The projections $p_{1},p_{2}$ are rig-flat (\cite{Ab10} 5.4.12). By \ref{coro base change conv} and \ref{functorial evaluation direct image}, $p_{1},p_{2}$ induce isomorphisms
	\begin{equation}
		p_{2}^{*}(\mathscr{F}_{\YY})\xrightarrow[c_{p_{2}}]{\sim} \mathscr{F}_{\mathfrak{Z}} \xleftarrow[c_{p_{1}}]{\sim} p_{1}^{*}(\mathscr{F}_{\YY}). 
	\end{equation}
	By a standard argument, the isomorphism $c_{p_{1}}^{-1}\circ c_{p_{2}}$ defines an $\mathscr{O}_{\QQ_{\YY/\rW}}$-stratification on $\mathscr{F}_{\YY}$. 
	Taking pull-back by $\PP_{\YY/\rW}\to \QQ_{\YY/\rW}$ \eqref{convergent to isocry}, we obtain an an $\mathscr{O}_{\PP_{\YY/\rW}}$-stratification on $\mathscr{F}_{\YY}$. The second assertion follows from \ref{connection locally proj}. 
\end{proof}

\begin{prop} \label{p-adic convergent isocry}
	Suppose that $Y$ is smooth over $k$ and that $g:X\to Y$ is smooth and proper.
	Let $\mathscr{E}$ be a convergent isocrystal of $(X/\rW)_{\conv,\tau}$ for $\tau\in \{\zar,\fppf\}$ and $i$ an integer $\ge 0$. We have: 

	\textnormal{(i)} For every object $\mathfrak{T}$ of $\pConv(Y/\rW)$ \eqref{padic enlargements}, $\rR^{i}g_{\conv,\tau*}(\mathscr{E})_{\mathfrak{T}}$ is coherent \eqref{coherent fppf}.

	\textnormal{(ii)} For every morphism $f$ of $\pConv(Y/\rW)$, the associated transition morphism $c_{f}$ of $\rR^{i}g_{\conv,\tau*}(\mathscr{E})$ is an isomorphism.
\end{prop}

By \ref{crystal base smooth locally free}, $\mathscr{E}$ is locally projective. 
We set $\mathscr{F}^{i}_{\tau}=\rR^{i}g_{\conv,\tau*}(\mathscr{E})$ and $\mathscr{G}^{i}_{\tau}=(\rR^{i}g_{X/\mathfrak{T},\tau*}(\omega_{\mathfrak{T}}^{*}(\mathscr{E})))$ \eqref{basic setting X to Y}. By \ref{evaluation direct image}, we have a canonical isomorphism
	\begin{equation} \label{calcul direct image coh}
		\mathscr{F}^{i}_{\tau,\mathfrak{T}}\xrightarrow{\sim} \mathscr{G}^{i}_{\tau,\mathfrak{T}}.
	\end{equation}

	\begin{nothing} Proof of \ref{p-adic convergent isocry} for Zariski topology. 
	
	(i) Since $(\mathfrak{T},u)$ is an object of $\pConv(Y/\rW)$, we take a lifting $\widetilde{u}:T\to Y$ of $u$ and we set $X_{T}=X\times_{Y}T$. Then, we have a canonical equivalence $(X_{T_{0}}/\mathfrak{T})_{\conv,\tau}\xrightarrow{\sim} (X_{T}/\mathfrak{T})_{\conv,\tau}$ \eqref{def conv cat} and the assertion follows from \ref{coh rig coh cristalline}.

	(ii) The question being local, by \ref{base change coro}, we may therefore assume that $Y$ is affien and admits a smooth lifting $\YY$ over $\rW$. Then $\mathscr{F}^{i}_{\zar,\YY}$ and $\mathscr{G}^{i}_{\zar,\YY}$ are locally projective of finite type by \ref{GM connection}. 
	
	We first prove assertion (ii) for a morphism $h:\mathfrak{T}\to \YY$ of $\pConv(Y/\rW)$ with target $\YY$.
	By \ref{coro base change conv}, we have a spectral sequence:
		\begin{equation}
			\rE^{i-j,j}_{2}=\rL_{i-j}h^{*}_{\zar} (\mathscr{G}_{\zar,\YY}^{j}) \Rightarrow \mathscr{G}_{\zar,\mathfrak{T}}^{i}.
		\end{equation}
		Since each $\mathscr{G}^{j}_{\zar,\YY}$ is locally projective of finite type, we deduce that $\rE^{i-j,j}_{2}=0$ for $i\neq j$. Then the transition morphism of $\mathscr{F}^{i}_{\zar,\YY}$ associated to $f$ is an isomorphism by \ref{functorial evaluation direct image} and \eqref{calcul direct image coh}.

	Since the question is local, for a general morphism $f:(\mathfrak{T}',u')\to (\mathfrak{T},u)$ of $\pConv(Y/\rW)$, we may assume that $u$ can be lifted to a morphism $h: \mathfrak{T}\to \YY$ of $\pConv(Y/\rW)$. By the previous result, $c_{h}$ and $c_{h\circ f}$ are isomorphisms. Then we deduce that $c_{f}$ is an isomorphism by \ref{description of mods}(c). 
\end{nothing}

\begin{nothing} Proof of \ref{p-adic convergent isocry} for fppf topology.
	We consider the presheaf $\mathscr{P}$ on $\Conv(Y/\rW)$ defined by 
	\begin{displaymath}
		(\mathfrak{T},u)\mapsto \rH^{i}( (X_{T_{0}}/\mathfrak{T})_{\conv,\fppf}, \omega_{\mathfrak{T}}^{*}(\mathscr{E})).
	\end{displaymath}
	By \ref{coh zar fppf}, the right hand side is isomorphic to $\rH^{i}( (X_{T_{0}}/\mathfrak{T})_{\conv,\zar}, \omega_{\mathfrak{T}}^{*}(\alpha_{*}(\mathscr{E})))$. We set $\mathscr{F}_{\zar}^{i}=\rR^{i}g_{\conv,\zar*}(\alpha_{*}(\mathscr{E}))$. By \ref{evaluation direct image}, the fppf (resp. Zariski) sheaf associated to $\mathscr{P}$ is $\mathscr{F}_{\fppf}^{i}$ (resp. $\mathscr{F}_{\zar}^{i}$). Then we deduce a canonical isomorphism \eqref{morphism of topoi fppf zar}
	\begin{equation} \label{rel cohomology associated}
		\alpha^{*}(\mathscr{F}_{\zar}^{i})\xrightarrow{\sim} \mathscr{F}_{\fppf}^{i}.
	\end{equation}
	
	Let $\mathfrak{T}$ be an object of $\pConv(Y/\rW)$. 
	By \ref{p-adic convergent isocry} for Zariski topology and fppf descent, we deduce that $\mathscr{F}_{\fppf,\mathfrak{T}}^{i}$ is the fppf sheaf associated to the coherent $\mathscr{O}_{\mathfrak{T}}[\frac{1}{p}]$-module $\mathscr{F}^{i}_{\zar,\mathfrak{T}}$ \eqref{coherent fppf} and hence is coherent. Assertion (i) follows.

	Since $\mathscr{F}_{\fppf,\mathfrak{T}}^{i}$ is the fppf sheaf associated to $\mathscr{F}_{\zar,\mathfrak{T}}^{i}$, assertion (ii) follows from \ref{p-adic convergent isocry}(ii) for Zariski topology and \eqref{rel cohomology associated}.
\end{nothing}

\section{Frobenius descents} \label{Frob descents}
\begin{nothing} \label{basic notation frob}
	In this section, $\SS$ denotes a flat formal $\rW$-scheme of finite type. Suppose that the Frobenius morphism $F_{S_{0}}:S_{0}\to S_{0}$ of the reduced subscheme of $S$ is \textit{flat} (and hence faithfully flat). 
	
	Let $X$ be an $S_{0}$-scheme locally of finite type. 
	We denote by $X'$ the base change of $X$ by $F_{S_{0}}$ and by $F_{X/S_0}:X\to X'$ the relative Frobenius morphism of $X$ relative to $S_{0}$. Then we have a commutative diagram
	\begin{equation}
		\xymatrix{
			X \ar[r]^{F_{X/S_{0}}} \ar[rd] & X' \ar[r] \ar@{}[dr]|{\Box} \ar[d]& X \ar[d] \\
			& S_{0} \ar[r]^{F_{S_{0}}} & S_{0}
		}
	\end{equation}

	We study the functorial morphism of convergent topoi induced by the relative Frobenius morphism $F_{X/S_{0}}:X \to X'$.
	We denote the functor \eqref{functorial functor} induced by $F_{X/S_0}$ by:
	\begin{eqnarray*} \label{functorial functor Frob}
	\rho: \Conv(X/\SS)&\to& \Conv(X'/\SS), \\
	(\mathfrak{T},u) &\mapsto& (\mathfrak{T},F_{X/S_{0}}\circ u).
\end{eqnarray*}
Note that $F_{X/S_{0}}\circ u=u'\circ F_{T_{0}/S_{0}}$.
\end{nothing}

\begin{lemma}\label{lemma full faithful of rho}
	Let $Y$ be a reduced $S_{0}$-scheme, $Z$ an $S_{0}$-scheme and $g_{1},g_{2}:Y\to Z$ two $S_{0}$-morphisms. We put $h_{i}=g'_{i}\circ F_{Y/S_{0}}:Y\to Y'\to Z'$ for $i=1,2$. If $h_{1}=h_{2}$, then $g_{1}=g_{2}$.
\end{lemma}
\begin{proof}
	Since $F_{Y/S_0}$ is a homeomorphism and $h_{1}=h_{2}$, then $|g_{1}|=|g_{2}|$ on the underlying topological spaces. 
	Since the question is local, we can reduce to the case where $Y,Z,S_{0}$ are affine.

	Since $Y$ is reduced and separated over $S_{0}$, $F_{Y/S_{0}}$ is schematically dominant (\cite{EGAInew} 5.4.2) and we deduce that $g_{1}'=g_{2}'$ (\cite{EGAInew} 5.4.1).
	The Frobenius morphism $F_{S_{0}}$ is faithfully flat. Then the functor $Y\mapsto Y'$ from the category of affine $S_{0}$-schemes to itself is faithful. The lemma follows.
\end{proof}

\begin{lemma} \label{functor rho fully faithful}
	The functor $\rho$ is fully faithful.
\end{lemma}
\begin{proof}
	The functor $\rho$ is clearly faithful. We prove its fullness.
	Let $(\mathfrak{T}_{1},u_{1})$, $(\mathfrak{T}_{2},u_{2})$ be two objects of $\Conv(X/\SS)$ and $g:\rho(\mathfrak{T}_{1},u_{1})\to \rho(\mathfrak{T}_{2},u_{2})$ a morphism of $\Conv(X'/\SS)$. We set $g_{0}:T_{1,0}\to T_{2,0}$ the induced morphism. To show that the morphism $\mathfrak{T}_{1}\to \mathfrak{T}_{2}$ define a morphism of $\Conv(X/\SS)$ which is sent to $g$ by $\rho$, it suffices to show that $u_{1}=u_{2}\circ g_{0}$. Since $g$ is a morphism of $\pConv(X'/\SS)$, we have a commutative diagram
	\begin{displaymath}
		\xymatrix{
			T_{1,0}\ar[rr]^{g_{0}} \ar[d]_{F_{T_{1,0}/S_0}} &&T_{2,0} \ar[d]^{F_{T_{2,0}/S_0}} \\
			(T_{1,0})' \ar[rd]_{u_{1}'} \ar[rr]^{g_{0}'} && (T_{2,0})' \ar[ld]^{u_{2}'} \\
			& X' &
		}
	\end{displaymath}
	Then the assertion follow from \ref{lemma full faithful of rho} applied to $u_{1}$ and $u_{2}\circ g_{0}$.
\end{proof}

\begin{lemma} \label{lemma condition 4}
	\textnormal{(i)} Let $(\mathfrak{T},u)$ be an object of $\Conv(X'/\SS)$ such that $\mathfrak{T}$ is affine and that $u:T_{0}\to X'$ factor through an affine open subscheme $U'$ of $X'$. Then there exists an object $(\mathfrak{Z},v)$ of $\Conv(X/\SS)$ and a fppf covering $\{f:\rho(\mathfrak{Z},v)\to (\mathfrak{T},u)\}$ in $\Conv(X'/\SS)$.

	\textnormal{(ii)} Keep the assumption and notation of \textnormal{(i)}.
	Let $g:(\mathfrak{T}_{1},u_{1})\to (\mathfrak{T},u)$ be a morphism of $\Conv(X'/\SS)$. Then there exists a morphism $h:(\mathfrak{Z}_{1},v_{1})\to (\mathfrak{Z},v)$ of $\Conv(X/\SS)$ and a fppf covering $\{\varphi:\rho(\mathfrak{Z}_{1},v_{1})\to (\mathfrak{T}_{1},u_{1})\}$ such that the following diagram is Cartesian:
	\begin{equation}\label{diag rho U1U2 pre}
	\xymatrix{
		\rho(\mathfrak{Z}_{1},v_{1})\ar[r]^{\varphi} \ar[d]_{\rho(h)} \ar@{}[dr]|{\Box}& (\mathfrak{T}_{1},u_{1}) \ar[d]^{g}\\
		\rho(\mathfrak{Z},v)\ar[r]^{f} & (\mathfrak{T},u) }
	\end{equation}

	\textnormal{(iii)} Every object of $\Conv(X'/\SS)$ admits a Zariski covering whose objects satisfying conditions of \textnormal{(i)}.
\end{lemma}

\begin{proof}
	(i) We set $U=F_{X/S_0}^{-1}(U')$ which is an affine $S_{0}$-scheme of finite type and we take a closed $S_{0}$-immersion $\iota_{0}: U\to Y_{0}=\Spec(\mathscr{O}_{S_{0}}[T_{1},\cdots,T_{d}])$.
	We denote by $\YY=\Spf(\mathscr{O}_{\SS}\{T_{1},\cdots,T_{d}\})$ and $F:\YY\to \YY$ the $\SS$-morphism defined by sending each $T_{i}$ to $T_{i}^{p}$. 
	
	Note that $Y_{0}'=Y_{0}$ and the restriction of $F$ on $Y_{0}$ is same the relative Frobenius morphism $F_{Y_{0}/S_{0}}$. 
	We have a commutative diagram 
	\begin{equation}
		\xymatrix{
			U\ar[r]^{\iota_{0}} \ar[d]_{F_{U/S_{0}}} & Y_{0} \ar[d]^{F|_{Y_{0}}} \\
			U' \ar[r]^{\iota_{0}'} & Y_{0}
		}
		\label{diagram iota Frob}
	\end{equation}
	and a canonical morphism $U\to U'\times_{Y_{0},F}Y_{0}$. 
	We denote the composition of $\iota_{0}':U'\to Y_{0}$ and $Y_{0}\to \YY$ by $\iota'$.
	Since $\YY$ is smooth over $\SS$, there exists an $\SS$-morphism $\tau:\mathfrak{T}\to \YY$ lifting $\iota'\circ u:T_{0}\to \YY$. 
We consider the following commutative diagram:
\begin{equation} \label{diag lifting Frob}
\xymatrix{
	&& (T_{0}\times_{Y_{0},F}Y_{0})_{0}  \ar[rr]\ar'[d][dd] \ar[ld] && \mathfrak{T}\times_{\YY,F}\YY \ar[dd] \ar[ld]\\
	&T_{0} \ar[rr]\ar[dd]
	&& \mathfrak{T}\ar[dd]\\
	U_{0} \ar'[r][rr] \ar[rd] && (U'\times_{Y_{0},F}Y_{0})_{0} \ar[ld] \ar'[r][rr]
	&& \YY\ar[ld]^{F}\\
	&U'  \ar[rr]^{\iota'}&& \YY}
\end{equation}
where $U_{0}\to U'$ is induced by $F_{U/S_{0}}$. 

If the ideal sheaf associated to $\iota_{0}:U\hookrightarrow Y_{0}$ is locally generated by polynomials $\{f_{1},\cdots,f_{n}\}$ of $\mathscr{O}_{S_{0}}[T_{1},\cdots,T_{d}]$, the ideal sheaf associated to $\iota'_{0}\times_{Y_{0},F}Y_{0}: U'\times_{Y_{0},F}Y_{0}\hookrightarrow Y_{0}$ \eqref{diagram iota Frob} is locally generated by $\{f_{1}^{p},\cdots,f_{n}^{p}\}$. 
	Then the canonical morphism $U\to U'\times_{Y_{0},F}Y_{0}$ induces an isomorphism 
\begin{displaymath}
	U_{0}\xrightarrow{\sim} (U'\times_{Y_{0},F}Y_{0})_{0}.
\end{displaymath}
By \eqref{diag lifting Frob} and (\cite{EGAInew} 4.5.11), we obtain an object $(\mathfrak{T}\times_{\YY,F}\YY,v)$ of $\Conv(X/\SS)$ and a morphism $f:\rho(\mathfrak{T}\times_{\YY,F}\YY,v)\to (\mathfrak{T},u)$ of $\Conv(X'/\SS)$. Since the reduction modulo $p$ of $F$ is faithfully flat of finite type (\cite{Hodge} 3.2), so is $F$ (cf. \cite{Xu} 7.2). Then, $f$ is a fppf covering \eqref{def topology Conv}.

(ii) We denote by $(\mathfrak{Z}_{1},w)$ the fibered product $\rho(\mathfrak{Z},v)\times_{(\mathfrak{T},u)}(\mathfrak{T}_{1},u_{1})$ in $\Conv(X'/\SS)$. By applying \ref{cont cocont func} to the projection $(\mathfrak{Z}_{1},w)\to \rho(\mathfrak{Z},v)$, we obtain the Cartesian diagram \eqref{diag rho U1U2 pre}. Since $\varphi$ is the base change of $f$, $\varphi$ is a fppf covering.

(iii) Let $(\mathfrak{T},u)$ be an object of $\Conv(X'/\SS)$ and $U'$ an affine open subscheme of $X'$. We denote by $\mathfrak{T}_{U'}$ the open formal subscheme of $\mathfrak{T}$ associated to the open subset $u^{-1}(|U'|)$ of $|T_{0}|=|T|$. The assertion follows by taking an affine covering of $\mathfrak{T}_{U'}$ for every $U'$.
\end{proof}

\begin{lemma} \label{lemma rho 4 comp}
	Let $\mathfrak{T}$ an object of $\Conv(X'/\SS)$, $\mathfrak{Z}$ an object of $\Conv(X/\SS)$ and $\{\rho(\mathfrak{Z})\to \mathfrak{T}\}$ a morphism of $\Conv(X'/\SS)$. Then there exists an object $\mathfrak{Z}\times_{\mathfrak{T}}\mathfrak{Z}$ of $\Conv(X/\SS)$ and two morphisms $p_{1},p_{2}:\mathfrak{Z}\times_{\mathfrak{T}}\mathfrak{Z}\to \mathfrak{Z}$ of $\Conv(X/\SS)$ such that $\rho(\mathfrak{Z}\times_{\mathfrak{T}}\mathfrak{Z})=\rho(\mathfrak{Z})\times_{\mathfrak{T}}\rho(\mathfrak{Z})$ and that $\rho(p_{1})$ (resp. $\rho(p_{2})$) is the projection $\rho(\mathfrak{Z})\times_{\mathfrak{T}}\rho(\mathfrak{Z})\to \rho(\mathfrak{Z})$ on the first (resp. second) component.
\end{lemma}

\begin{proof}
	By applying \ref{cont cocont func}(i) to the projection $\rho(\mathfrak{Z})\times_{\mathfrak{T}}\rho(\mathfrak{Z})\to \rho(\mathfrak{Z})$ on the first component, we obtain an object $\mathfrak{Z}\times_{\mathfrak{T}}\mathfrak{Z}$ of $\Conv(X/\SS)$ and a morphism $p_{1}:\mathfrak{Z}\times_{\mathfrak{T}}\mathfrak{Z}\to \mathfrak{Z}$ as in the proposition. The existence of $p_{2}$ follows from the fullness of $\rho$ \eqref{functor rho fully faithful}.
\end{proof}

\begin{theorem} \label{Frob descent fppf topos}
	Suppose that the Frobenius morphism $F_{S_{0}}:S_{0}\to S_{0}$ flat.
	For every $S_{0}$-scheme locally of finite type $X$, the relative Frobenius morphism $F_{X/S_{0}}$ induces an equivalence of topoi.
	\begin{equation} \label{rel Frob functorial}
		F_{X/S_{0},\conv,\fppf}:(X/\SS)_{\conv,\fppf}\xrightarrow{\sim} (X'/\SS)_{\conv,\fppf}.
	\end{equation}
\end{theorem}
\begin{proof}
	The theorem follows from \ref{lemma adjunction iso}, \ref{cont cocont func}, \ref{functor rho fully faithful} and \ref{lemma condition 4}.
\end{proof}

\begin{prop} \label{Frob descent isocrystals}
	Keep the assumption of \ref{Frob descent fppf topos}.
	The inverse image and the direct image functors of $F_{X/S_0,\conv,\zar}$ induce equivalences of categories quasi-inverse to each other \eqref{def convergent crystals}
	\begin{equation} \label{equi crystal cat}
		\Iso^{\dagger}(X/\SS)\rightleftarrows\Iso^{\dagger}(X'/\SS).
	\end{equation}
\end{prop}
\begin{proof}
	By \ref{coh fppf descent}, convergent isocrystals are sheaves for fppf topology and we work with fppf topology in this proof. We write simply \eqref{rel Frob functorial} for $F_{X/S_0}$ and we will show that the direct image and inverse image functors of $F_{X/S_0}$ send coherent crystals of $\mathscr{O}_{X/\SS}[\frac{1}{p}]$-modules to coherent crystals of $\mathscr{O}_{X'/\SS}[\frac{1}{p}]$-modules. The assertion for inverse image follows from \eqref{description of pullback} and we will prove it for direct image.

Let $\mathscr{F}$ be coherent crystal of $\mathscr{O}_{X/\SS}[\frac{1}{p}]$-modules and $(\mathfrak{T},u)$ an object of $\Conv(X'/\SS)$. We first show that $(F_{X/S_0*}(\mathscr{F}))_{\mathfrak{T}}$ is coherent.
By \ref{lemma condition 4}(iii), we may assume that $(\mathfrak{T},u)$ satisifies conditions of \ref{lemma condition 4}(i).
Then, by \ref{lemma condition 4}(ii) and \ref{lemma rho 4 comp}, there exist objects $\mathfrak{Z}$ and $\mathfrak{Z}\times_{\mathfrak{T}}\mathfrak{Z}$ of $\Conv(X/\SS)$, a fppf covering $\{f:\rho(\mathfrak{Z})\to \mathfrak{T}\}$ and two morphisms $p_{1},p_{2}:\mathfrak{Z}\times_{\mathfrak{T}}\mathfrak{Z}\to \mathfrak{Z}$ such that $\rho(\mathfrak{Z}\times_{\mathfrak{T}}\mathfrak{Z})=\rho(\mathfrak{Z})\times_{\mathfrak{T}}\rho(\mathfrak{Z})$ and that $\rho(p_{1})$ and $\rho(p_{2})$ are the canonical projections of $\rho(\mathfrak{Z})\times_{\mathfrak{T}}\rho(\mathfrak{Z})$. In particular, the morphism of formal schemes $\mathfrak{Z}\times_{\mathfrak{T}}\mathfrak{Z}\to \mathfrak{Z}$ attached to $p_{1}$ (resp. $p_{2}$) is the projection on the first (resp. second) component.

Since the adjunction morphism $F_{X/S_0}^{*}F_{X/S_0*}\to \id$ is an isomorphism \eqref{Frob descent fppf topos}, we have \eqref{description of pullback}
\begin{equation}\label{C-1 calcul app}
	(F_{X/S_0*}(\mathscr{F}))_{\rho(\mathfrak{Z})}=\mathscr{F}_{\mathfrak{Z}}, \qquad (F_{X/S_0*}(\mathscr{F}))_{\rho(\mathfrak{Z}\times_{\mathfrak{T}}\mathfrak{Z})}=\mathscr{F}_{\mathfrak{Z}\times_{\mathfrak{T}}\mathfrak{Z}}.
\end{equation}

Since $\mathscr{F}$ is a crystal, we have $\mathscr{O}_{\mathfrak{Z}\times_{\mathfrak{T}}\mathfrak{Z}}$-linear isomorphisms
\begin{equation}
	p_{2}^{*}(\mathscr{F}_{\mathfrak{Z}})\xrightarrow[\sim]{c_{p_{2}}} \mathscr{F}_{\mathfrak{Z}\times_{\mathfrak{T}}\mathfrak{Z}}\xleftarrow[\sim]{c_{p_{1}}} p_{1}^{*}(\mathscr{F}_{\mathfrak{Z}}).
\end{equation}
Then we obtain a descent datum $(\mathscr{F}_{\mathfrak{Z}},c_{p_{1}}^{-1}\circ c_{p_{2}})$ for the fppf covering $\{f:\mathfrak{Z}\to \mathfrak{T}\}$.
By \ref{descent fppf}, there exists a coherent $\mathscr{O}_{\mathfrak{T}}[\frac{1}{p}]$-module $\mathscr{M}$ and a canonical $\mathscr{O}_{\mathfrak{Z}}$-linear isomorphism $f^{*}(\mathscr{M})\xrightarrow{\sim} \mathscr{F}_{\mathfrak{Z}}$.

On the other hand, since $F_{X/S_0*}(\mathscr{F})$ is a sheaf fppf topology, there exists an exact sequence
\begin{equation}
	0\to (F_{X/S_0*}(\mathscr{F}))(\mathfrak{T})\to (F_{X/S_0*}(\mathscr{F}))(\rho(\mathfrak{Z}))\to (F_{X/S_0*}(\mathscr{F}))(\rho(\mathfrak{Z}\times_{\mathfrak{T}}\mathfrak{Z})).
\end{equation}
By \eqref{C-1 calcul app}, we deduce an $\mathscr{O}_{\mathfrak{T}}$-linear isomorphism $\mathscr{M}\xrightarrow{\sim}(F_{X/S_0*}(\mathscr{F}))_{\mathfrak{T}}$. In particular, $(F_{X/S_0*}(\mathscr{F}))_{\mathfrak{T}}$ is coherent. Hence $F_{X/S_0*}(\mathscr{F})$ is coherent.

Following the same argument as in the second part of the proof of (\cite{Xu} 8.15), we show that for every morphism $g$ of $\Conv(X/\SS)$, the transition morphism $c_{g}$ associated to $F_{X/S_0*}(\mathscr{F})$ is an isomorphism, i.e. $F_{X/S_0*}(\mathscr{F})$ is a crystal.
\end{proof}

\begin{prop} \label{thm comp coh}
	We consider the following diagram \eqref{def uXS conv}:
	\begin{equation}
		\xymatrixcolsep{4pc}\xymatrix{
			(X/\SS)_{\conv,\zar} \ar[r]^{F_{X/S_0,\conv,\zar}} \ar[d]_{u_{X/\SS}} & (X'/\SS)_{\conv,\zar} \ar[d]^{u_{X'/\SS}} \\
			X_{\zar}\ar[r]^{F_{X/S_0}}& X'_{\zar} }
	\end{equation}
	Let $\mathscr{E}$ be a convergent isocrystal of $(X/\rW)_{\conv,\zar}$ and denote the structure morphism $X\to S_{0}$ by $f$. Then there exists a canonical isomorphism in the derived category $\rD(X_{\zar},f^{-1}(\mathscr{O}_{\SS}))$
	\begin{equation}
		F_{X/S_0*}(\rR u_{X/\SS *}(\mathscr{E}))\xrightarrow{\sim} \rR u_{X'/\SS *}(F_{X/S_0\conv,\zar*}(\mathscr{E})).
	\end{equation}
\end{prop}
\begin{proof}
	We consider $\mathscr{E}$ as a coherent crystal of $\mathscr{O}_{X/\SS}[\frac{1}{p}]$-module of $(X/\SS)_{\conv,\fppf}$. Then, $\alpha_{*}(\mathscr{E})$ and $\mathscr{E}$ are equal as presheaves and $\rR^{i}\alpha_{*}(\mathscr{E})=0$ for $i\ge 1$ \eqref{coh zar fppf}. 
	The assertion follows from \ref{Frob descent fppf topos} and the fact that $F_{X/S_{0}}:X_{\zar}\to X'_{\zar}$ is an equivalence of topoi.
\end{proof}
\begin{coro} \label{Comp dR complexes}
	Keep the assumption of \ref{thm comp coh} and suppose that there exists smooth liftings $\XX$ of $X$ and $\XX'$ of $X'$ over $\SS$. 
	Let $f:X\to S_{0}$ be the canonical morphism.
	Then there exists a canonical isomorphism between the de Rham compelexes of $\mathscr{E}$ and of $F_{X/S_{0},\conv,\zar*}(\mathscr{E})$ in $\rD(X'_{\zar},f^{-1}(\mathscr{O}_{\SS}))$
	\begin{equation}
		F_{X/S_{0}*}(\mathscr{E}_{\XX}\otimes_{\mathscr{O}_{\XX}}\widehat{\Omega}_{\XX/\SS}^{\bullet}) \xrightarrow{\sim} (F_{X/S_{0},\conv,\zar*}\mathscr{E})_{\XX'}\otimes_{\mathscr{O}_{\XX'}}\widehat{\Omega}_{\XX'/\SS}^{\bullet}.
	\end{equation}
\end{coro}
\begin{proof}
	It follows from \eqref{direct image u dR} and \ref{thm comp coh}.
\end{proof}

\begin{theorem} \label{higher direct image of con Fiso weak}
	Let $g:X\to Y$ a smooth proper morphism of smooth $k$-schemes and $\mathscr{E}$ be a convergent isocrystal of $\Conv(X/\rW)_{\conv,\tau}$. Then $\rR^{i}g_{\conv,\tau*}(\mathscr{E})$ is a convergent isocrystal of $\Conv(Y/\rW)_{\conv,\tau}$ for every $i\ge 0$. 
\end{theorem}

Inspired by Ogus' arguments in \cite{Ogus84}, we use \ref{p-adic convergent isocry} and Dwork's trick to prove \ref{higher direct image of con Fiso weak}. To do this, we introduce certain subcategories of $\Conv(X/\rW)$.

\begin{definition} \label{n p-adic enlargements}
	(i) Let $n$ be an integer $\ge 0$ and $T$ a $k$-scheme. 
	We denote by $T^{(n)}$ the closed subscheme of $T$ defined by the ideal sheaf $\{x\in \mathscr{O}_{T}| x^{p^{n}}=0\}$.

	(ii) We denote by $\Conv^{(n)}(X/\rW)$ the full subcategory of $\Conv(X/\rW)$ consisting of objects $(\mathfrak{T},u)$ such that $u:T_{0}\to X$ can be lifted to a $k$-morphism $\widetilde{u}:T^{(n)}\to X$.
\end{definition}

	Given an object $(\mathfrak{T},u)$ of $\Conv^{(n)}(X/\rW)$ and a morphism $(\mathfrak{T}',u')\to (\mathfrak{T},u)$ of $\Conv(X/\rW)$, then $(\mathfrak{T}',u')$ is also an object of $\Conv^{(n)}(X/\rW)$. In particular, $T^{(0)}=T$ and $\Conv^{(0)}(X/\rW)$ coincides with $\pConv(X/\rW)$ \eqref{padic enlargements}. 

\begin{lemma} \label{lemma rho Dwork's trick}
	The functor $\rho$ \eqref{functorial functor Frob} sends $\Conv^{(n+1)}(X/\rW)$ to $\Conv^{(n)}(X'/\rW)$.
\end{lemma}
\begin{proof}
	Let $(\mathfrak{T},u)$ be an object of $\Conv^{(n+1)}(X/\rW)$ and $\widetilde{u}:T^{(n+1)}\to X$ a lifting of $u$. The absolute Frobenius morphism $T^{(n)}\to T^{(n)}$ factors through the closed subscheme $T^{(n+1)}$ and then the relative Frobenius morphism $F_{T^{(n)}/k}$ factors through $(T^{(n+1)})'$. We have a commutative diagram
	\begin{equation}
		\xymatrixcolsep{4pc}\xymatrix{
			& T_{0} \ar[ld]_{u} \ar[d] \ar[rd]& \\
			X \ar[d]_{F_{X/k}} & T^{(n+1)} \ar[l]_{\widetilde{u}} \ar@{^{(}->}[r] \ar[d]_{F_{T^{(n+1)}/k}} &
			T^{(n)} \ar[d]^{F_{T^{(n)}/k}} \ar[ld]\\
			X' & (T^{(n+1)})' ~\ar[l]_{\widetilde{u}'} \ar@{^{(}->}[r] & (T^{(n)})' } \label{diagram rho}
	\end{equation}
	Then the morphism $F_{X/k}\circ u$ can be lifted to a $k$-morphism $T^{(n)}\to X'$ and the lemma follows.
\end{proof}

\begin{nothing} \label{pf of higher direct image of con Fiso weak}
	\textit{Proof of \ref{higher direct image of con Fiso weak}.}
	By \ref{coh zar fppf}, it suffices to prove the assertion for fppf topology. 
	There exists a convergent isocrystal $\mathscr{G}$ on $\Conv(X'/\rW)$ with $F_{X/k,\conv,\fppf}^{*}(\mathscr{G})\simeq \mathscr{E}$ \eqref{Frob descent isocrystals}. 
	If we set $\mathscr{F}=\rR^{i}g_{\conv,\fppf*}(\mathscr{E})$, $g'=g\otimes_{k,\sigma}k$ and $\mathscr{H}=\rR^{i}g'_{\conv,\fppf*}(\mathscr{G})$, then we have $F_{Y/k,\conv,\fppf}^{*}(\mathscr{H})\simeq \mathscr{F}$ by \ref{Frob descent fppf topos}. 

Each object (resp. morphism) of $\Conv(Y/\rW)$ belongs to a subcategory $\Conv^{(n)}(Y/\rW)$ \eqref{n p-adic enlargements} for some integer $n$. We prove the following assertions by induction:

\textnormal{(i)} For every object $\mathfrak{T}$ of $\Conv^{(n)}(Y/\rW)$, $\mathscr{F}_{\mathfrak{T}}$ is coherent.

\textnormal{(ii)} For every morphism $f$ of $\Conv^{(n)}(Y/\rW)$, the transition morphism $c_{f}$ associated to $\mathscr{F}$ is an isomorphism.

Assertions for $n=0$ are proved in \ref{p-adic convergent isocry}. Suppose that assertions hold for $n\ge 0$ and we prove them for $n+1$. 
Let $(\mathfrak{T},u)$ be an object of $\Conv^{(n+1)}(Y/\rW)$. By \eqref{description of pullback}, we deduce that
\begin{displaymath}
	\mathscr{H}_{\rho(\mathfrak{T})} \xrightarrow{\sim} \mathscr{F}_{\mathfrak{T}}.
\end{displaymath}
By induction hypotheses, for any object $\mathfrak{Z}$ of $\Conv^{(n)}(X'/\rW)$, $\mathscr{H}_{\mathfrak{Z}}$ is coherent. 
Then assertion (i) follows from \ref{lemma rho Dwork's trick} and the induction hypotheses.

Assertion (ii) can be verified in the same way by \ref{setting functorial morphism} and \ref{lemma rho Dwork's trick}. 
\end{nothing}

\begin{nothing} \label{def F-isocrystal}
	We denote by $\sigma:\rW\to \rW$ the Frobenius homomorphism.
	It induces a morphism of topoi $(X'/\rW)_{\conv,\tau}\to (X/\rW)_{\conv,\tau}$ by functoriality. For any sheaf $\mathscr{E}$ of $(X/\rW)_{\conv,\tau}$, we denote by $\mathscr{E}'$ the inverse image of $\mathscr{E}$ to $(X'/\rW)_{\conv,\tau}$. 

	For $\tau\in \{\zar,\fppf\}$, \textit{a convergent $F$-isocrystal of $\Conv(X/\rW)_{\conv,\tau}$} is a pair $(\mathscr{E},\varphi)$ consisting of a convergent isocrystal $\mathscr{E}$ of $(X/\rW)_{\conv,\tau}$ and an isomorphism, called \textit{Frobenius structure}
	\begin{equation}
		\varphi:F_{X/k,\conv,\tau}^{*}(\mathscr{E}')\xrightarrow{\sim} \mathscr{E}.
	\end{equation}
\end{nothing}
\begin{coro} \label{higher direct image of con Fiso weak Frob}
	Keep the assumption of \ref{higher direct image of con Fiso weak} and let $\varphi$ be a Frobenius structure on $\mathscr{E}$. Then, for any $i\ge 0$, the pair $(\rR^{i}g_{\conv,\tau*}(\mathscr{E}),\rR^{i}g_{\conv,\tau*}(\varphi))$ is a convergent $F$-isocrystal of $\Conv(Y/\rW)_{\conv,\tau}$.
\end{coro}
\begin{proof}
	Consider the isomorphism
	\begin{equation}
		\rR^{i}g_{\conv,\fppf*}(\varphi): \rR^{i}g_{\conv,\fppf*}(F_{X/k,\conv,\fppf}^{*}(\mathscr{E}'))\xrightarrow{\sim} \rR^{i}g_{\conv,\fppf*}(\mathscr{E}).
	\end{equation}
	By \ref{base change coro} and \ref{Frob descent fppf topos}, the left hand side is isomorphic to $F_{Y/k,\conv,\fppf}^{*}( (\rR^{i}g_{\conv,\fppf*}(\mathscr{E}))')$. Then the assertion follows.
\end{proof}

\section{Review on rigid geometry} \label{review rig}

\begin{nothing}\label{def category R}
	Recall that $\MS$ denotes the category of formal $\rW$-scheme of finite type whose morphisms are $\rW$-morphisms of finite type \eqref{def cat MS}. 
	The set $\MB$ of admissible blow-ups in $\MS$ forms a right multiplicative system in $\MS$ (\cite{Ab10} 4.1.4). We denote by $\MR$ the localized category of $\MS$ relative to $\MB$. 
	Objects of $\MR$ are called \textit{coherent rigid spaces (over $K=\rW[\frac{1}{p}]$)}. For any object $\XX$ (resp. morphism $f$) of $\MS$, its image in $\MR$ is denoted by $\XX^{\rig}$ (resp. $f^{\rig}$).

	For any object $\XX$ of $\MS$, we set $\MB_{\XX}$ the full subcategory of $\MS_{/\XX}$ consisting of admissible blowups.
\end{nothing}
\begin{nothing} \label{S diamond}
	Recall that $\MS^{\diamond}$ denotes the full subcategory of $\MS$ consisting of \textit{flat} formal $\rW$-schemes of finite type \eqref{def cat MS}. For any object $\XX$ of $\MS^{\diamond}$ and any admissible blow-up $\varphi:\XX'\to \XX$, $\XX'$ is still an object of $\MS^{\diamond}$ (\cite{Ab10} 3.1.4). Then the set $\MB^{\diamond}$ of admissible blow-ups in $\MS^{\diamond}$ forms a right multiplicative system in $\MS^{\diamond}$. By (\cite{Ab10} 4.1.15(iii)), the canonical functor $\MS^{\diamond}\to \MR$ is essentially surjective and hence induces an equivalence of categories between the localized category of $\MS^{\diamond}$ relative to $\MB^{\diamond}$ and $\MR$.
\end{nothing}

\begin{nothing} \label{def ad covering}
	Let $\mathcal{X}$ be a coherent rigid space. We denote by $\langle\mathcal{X}\rangle$ the set of rigid points of $\mathcal{X}$ (\cite{Ab10} 4.3.1). 
	We say that a family of morphisms $(f_{i}: \mathcal{X}_{i}\to \mathcal{X})_{i\in I}$ is \textit{a covering for rigid points} if $\bigcup_{i\in I} f_{i}(\langle\mathcal{X}_{i}\rangle)=\langle\mathcal{X}\rangle$. 
	
	Recall (\cite{Ab10} 4.3.8) that a family $(\mathcal{X}_{i}\to \mathcal{X})_{i\in I}$ of open immersions of coherent rigid spaces (\cite{Ab10} 4.2.1) is an \textit{admissible covering} if it admits a finite sub-covering for rigid points. 
	We denote by $\Ad_{/\mathcal{X}}$ the full subcategory of $\MR_{/\mathcal{X}}$ consisting of open immersions to $\mathcal{X}$ and by $\mathcal{X}_{\ad}$ the topos of sheaves of sets on $\Ad_{/\mathcal{X}}$ for the admissible topology.
\end{nothing}

\begin{nothing} \label{def rhoX}
	In the following of this section, $\XX$ denotes an object of $\MS$. 
	The functor $\Zar_{/\XX}\to \Ad_{/\XX^{\rig}}$ defined by $\UU\mapsto \UU^{\rig}$ is continuous and left exact and induces a morphism of topoi (\cite{Ab10} 4.5.2)
\begin{equation}
	\rho_{\XX}:\XX_{\ad}^{\rig} \to \XX_{\zar}.
\end{equation}
	For any object $(\XX',\varphi)$ of $\MB_{\XX}$, we denote by $\mu_{\varphi}$ the composition
\begin{equation}
		\mu_{\varphi}:\XX_{\ad}^{\rig}\xrightarrow{\sim} \XX'^{\rig}_{\ad} \xrightarrow{\rho_{\XX'}} \XX'_{\zar}.
\end{equation}

Let $\mathscr{F}$ be an $\mathscr{O}_{\XX}$-module. We denote by $\mathscr{F}^{\rig}$ the rigid fiber associated to $\mathscr{F}$ (\cite{Ab10} 4.7.4) which is a sheaf of $\XX_{\ad}^{\rig}$. We have a functorial isomorphism (\cite{Ab10} 4.7.4.2)
\begin{equation}
	\mathscr{F}^{\rig} \xrightarrow{\sim} \varinjlim_{(\XX',\varphi)\in \MB_{\XX}^{\circ}} \mu_{\varphi}^{*}( (\varphi_{\zar}^{*}(\mathscr{F}))[\frac{1}{p}]).
\end{equation}

	In particular, $(\mathscr{O}_{\XX})^{\rig}$ is a ring that we also denote by $\mathscr{O}_{\XX^{\rig}}$. 
	We have canonical morphisms of ringed topoi (\cite{Ab10} 4.7.5)
	\begin{equation}
		\rho_{\XX}:(\XX_{\ad}^{\rig},\mathscr{O}_{\XX^{\rig}})\to (\XX_{\zar},\mathscr{O}_{\XX}[\frac{1}{p}]), \qquad
		\mu_{\varphi}:(\XX_{\ad}^{\rig},\mathscr{O}_{\XX^{\rig}}) \xrightarrow{\rho_{\XX'}} (\XX'_{\zar},\mathscr{O}_{\XX'}[\frac{1}{p}]) ).
	\end{equation}
	
	If $\mathscr{F}$ is moreover coherent, we have a canonical isomorphism $\rho_{\XX}^{*}(\mathscr{F}[\frac{1}{p}])\xrightarrow{\sim} \mathscr{F}^{\rig}$ (\cite{Ab10} 4.7.2.8). 
%
\end{nothing}

\begin{prop}\label{rhoXX coh equi}
	Let $\Coh(\mathscr{O}_{\XX^{\rig}})$ be the category of coherent $\mathscr{O}_{\XX^{\rig}}$-module over $\XX_{\ad}^{\rig}$ \textnormal{(\cite{Ab10} 4.8.16)}. 	
	The inverse image and direct image functors of $\rho_{\XX}$ induces equivalences of categories quasi-inverse to each other
	\begin{equation}
		\Coh(\mathscr{O}_{\XX}[\frac{1}{p}])\rightleftharpoons\Coh(\mathscr{O}_{\XX^{\rig}}).
	\end{equation}
\end{prop}
\begin{proof}
	We write simply $\rho$ for $\rho_{\XX}$. By (\cite{Ab10} 4.7.8.1), $\rho_{*}$ send coherent $\mathscr{O}_{\XX^{\rig}}$-modules to coherent $\mathscr{O}_{\XX}[\frac{1}{p}]$-modules.
	By \eqref{equi coh mod inverse p} and (\cite{Ab10} 4.8.18), the inverse image functor $\rho^{*}$ is essentially surjective. By (\cite{Ab10} 4.7.8.2 and 4.7.29.2), $\rho^{*}$ is fully faithful. 
	In view of the canonical isomorphisms $\rho^{*}\xrightarrow{\sim} \rho^{*}\rho_{*}\rho^{*}$ and $\rho^{*}\rho_{*}\rho^{*}\xrightarrow{\sim} \rho^{*}$, we deduce that $\rho_{*}$ is a quasi-inverse. 
\end{proof}

\begin{nothing}
	Let $f:\XX\to \YY$ be a morphism of $\MS$. It induces a morphism of ringed topoi $f_{\ad}^{\rig}:(\XX^{\rig}_{\ad},\mathscr{O}_{\XX^{\rig}}) \to (\YY^{\rig}_{\ad},\mathscr{O}_{\YY^{\rig}})$ (\cite{Ab10} 4.7.2.1). The diagram 
	\begin{equation}
		\xymatrix{
			(\XX^{\rig}_{\ad},\mathscr{O}_{\XX^{\rig}}) \ar[d]_{\rho_{\XX}} \ar[r]^{f^{\rig}_{\ad}} & (\YY^{\rig}_{\ad},\mathscr{O}_{\YY^{\rig}}) \ar[d]^{\rho_{\YY}} \\
			(\XX_{\zar},\mathscr{O}_{\XX}[\frac{1}{p}]) \ar[r]^{f_{\zar}} & (\YY_{\zar},\mathscr{O}_{\YY}[\frac{1}{p}]) }
	\end{equation}
	is commutative up to canonical isomorphisms (\cite{Ab10} 4.7.24.2). 
	
	Let $\mathscr{F}$ be a coherent $\mathscr{O}_{\YY^{\rig}}$-module. By \ref{rhoXX coh equi}, there exists canonical isomorphisms
	\begin{displaymath}
		\rho_{\XX}^{*}(f_{\zar}^{*}(\rho_{\YY_{*}}(\mathscr{F}))) \xrightarrow{\sim} f^{\rig*}_{\ad}(\rho_{\YY}^{*}(\rho_{\YY*}(\mathscr{F})))\xrightarrow{\sim} f^{\rig*}_{\ad}(\mathscr{F}).
	\end{displaymath}
	Then we deduce that the following base change morphism is an isomorphism
	\begin{equation} \label{base change rho coherent}
		f_{\zar}^{*}\rho_{\YY*}(\mathscr{F})\xrightarrow{\sim} \rho_{\XX*}f_{\ad}^{\rig*}(\mathscr{F}).
	\end{equation}
\end{nothing}

\begin{nothing}\label{def rfppf covering}
	Following (\cite{Ab10} 5.10.1), we say that a morphism $f:\mathcal{X}\to \mathcal{Y}$ of coherent rigid spaces is \textit{flat} if the associated morphism of ringed topoi $f_{\ad}:(\mathcal{X}_{\ad},\mathscr{O}_{\mathcal{X}})\to (\mathcal{Y}_{\ad},\mathscr{O}_{\mathcal{Y}})$ is flat (\cite{SGAIV} V 1.7).
	We say that $f$ is \textit{faithfully flat} if $f$ is flat and $f(\langle\mathcal{X}\rangle)=\langle\mathcal{Y}\rangle$ (\cite{Ab10} 4.3.1 and 5.10.11).

	A morphism $f:\XX\to \YY$ of $\MS$ is rig-flat \eqref{def rigid points} if and only if $f^{\rig}$ is flat (\cite{Ab10} 5.5.8). 

	We say that a family $(\mathcal{X}_{i}\to \mathcal{X})_{i\in I}$ of flat morphisms of $\MR$ is a \textit{fppf covering} if it admits a finite sub-covering for rigid points \eqref{def ad covering}.
	In view of (\cite{Ab10} 5.10.12), fppf coverings are stable by composition and by base change in $\MR$. 

	Let $\mathcal{X}$ be a coherent rigid space. We denote by $\Rf_{/\mathcal{X}}$ the full subcategory of $\MR_{/\mathcal{X}}$ consisting of flat morphisms to $\mathcal{X}$. We call \textit{fppf topology} the topology on $\Rf_{/\mathcal{X}}$ generated by the pretopology for which coverings are fppf coverings. We denote by $\mathcal{X}_{\fppf}$ the topos of sheaves of sets on this site.

	By fppf descent of morphisms (\cite{Ab10} 5.12.4), the fppf topology on $\Rf_{/\mathcal{X}}$ is sub-canonical, i.e. the presheaf associated to each object of $\Rf_{/\mathcal{X}}$ is a sheaf for the fppf topology. 
\end{nothing}

\begin{nothing} \label{fppf to ad and functorial}
	The canonical functor $\Ad_{/\mathcal{X}}\to \Rf_{/\mathcal{X}}$ is continuous and left exact. Then it induces a morphism of topoi
	\begin{equation}
		\alpha_{\mathcal{X}}:\mathcal{X}_{\fppf}\to \mathcal{X}_{\ad}.
	\end{equation}

	Given a morphism $f:\mathcal{X}'\to \mathcal{X}$ of $\MR$, the canonical functor $\Rf_{/\mathcal{X}}\to \Rf_{/\mathcal{X}'}$ (resp. $\Ad_{/\mathcal{X}}\to \Ad_{/\mathcal{X}'}$) defined by $\mathcal{Y}\mapsto \mathcal{Y}\times_{\mathcal{X}}\mathcal{X}'$ is continuous and left exact. It induces morphisms of topoi 
	\begin{equation} \label{functorial ftau rig}
		f_{\tau}:\mathcal{X}'_{\tau}\to \mathcal{X}_{\tau}, \quad \tau\in \{\ad,\fppf\},
	\end{equation}
	compatible with $\alpha_{\mathcal{X}}$ and $\alpha_{\mathcal{X}'}$.
	If $f$ is a morphism of $\Rf_{/\mathcal{X}}$ (resp. $\Ad_{/\mathcal{X}}$), in view of the description of direct image functors, one verifies that the above morphism coincides with the localization morphism at $\mathcal{X}'$.
\end{nothing}


\begin{nothing} \label{coherent fppf rig}
	Let $\mathscr{F}$ be a coherent $\mathscr{O}_{\mathcal{X}}$-module. The presheaf on $\Rf_{/\mathcal{X}}$
	\begin{displaymath}
		(f:\mathcal{X}'\to \mathcal{X}) \mapsto \Gamma(\mathcal{X}',f^{*}_{\ad}(\mathscr{F}))
	\end{displaymath}
	is a sheaf for the fppf topology by fppf descent for coherent modules on rigid spaces (\cite{Ab10} 5.11.11). In particular, $\mathscr{O}_{\mathcal{X}}$ defines a sheaf of rings of $\mathcal{X}_{\fppf}$ that we still denote by $\mathscr{O}_{\mathcal{X}}$. We call abusively \textit{coherent $\mathscr{O}_{\mathcal{X}}$-module of $\mathcal{X}_{\fppf}$} a sheaf of $\mathcal{X}_{\fppf}$ associated to a coherent $\mathscr{O}_{\mathcal{X}}$-module of $\mathcal{X}_{\ad}$.

	For $\tau\in \{\ad,\fppf\}$, the morphism of topoi $f_{\tau}$ \eqref{functorial ftau rig} is ringed by $\mathscr{O}_{\mathcal{X}}$ and $\mathscr{O}_{\mathcal{X}'}$. 
	For any $\mathscr{O}_{\mathcal{X}}$-module $\mathscr{M}$ of $\mathcal{X}_{\tau}$, we use $f^{-1}_{\tau}(\mathscr{M})$ to denote the inverse image in the sense of sheaves and we keep $f_{\tau}^{*}(\mathscr{M})$ for the inverse image in the sense of modules.
\end{nothing}



\section{Rigid convergent topos and convergent isocrystals} \label{rig conv topos}
\begin{nothing}
	In this section, $\mathfrak{S}$ denotes a flat formal $\rW$-scheme of finite type and $X$ an $S$-scheme. 
	
	We will introduce a full subcategory of $(X/\SS)_{\conv,\zar}$ consisting of sheaves $\mathscr{F}=\{\mathscr{F}_{\mathfrak{T}},\beta_{f}\}$ \eqref{construction descent data} such that the morphism $\beta_{f}$ is an isomorphism if the underlying morphism of formal schemes of $f$ is an admissible blowup. 
	It turns out that this category forms a topos $(X/\SS)_{\rconv,\ad}$ (\ref{def RConv rconv}, \ref{rconv description}) and admits a canonical morphism to $(X/\SS)_{\conv,\zar}$ \eqref{rho XS rconv}. 
	Convergent isocrystals lie in $(X/\SS)_{\rconv,\ad}$ and their cohomologies remain unchanged in this topos (\ref{rho conv coh}, \ref{coh zar ad iso}). 
	
	We begin by introducing $(X/\SS)_{\rconv,\ad}$ and its fppf variant.
\end{nothing}

	\begin{lemma} \label{right multiplicative in Conv}
	We denote by $\MB_{X/\mathfrak{S}}$ the set of morphisms in $\Conv(X/\mathfrak{S})$ \eqref{def conv cat} whose underlying morphism on formal schemes is an admissible blow-up. Then, it forms a right multiplicative system in $\Conv(X/\SS)$. 
	\end{lemma}
	\begin{proof}
	For any object $(\mathfrak{T},u)$ of $\Conv(X/\SS)$, we have a canonical functor $s_{\mathfrak{T}}:\MB_{\mathfrak{T}}\to \Conv(X/\SS)$ sending $(\mathfrak{T}',\varphi)$ to $(\mathfrak{T}',u\circ \varphi_{0})$.
	Then the assertion follow from the facts that admissible blow-ups form a right multiplicative system in $\MS^{\diamond}_{/\mathfrak{S}}$ \eqref{S diamond} and that the canonical functor $\Conv(X/\SS)\to \MS_{\SS}^{\diamond}$ is faithful.
	\end{proof}

	\begin{nothing} \label{def RConv rconv}
	We denote by $\RConv(X/\mathfrak{S})$ the localized category of $\Conv(X/\mathfrak{S})$ relative to $\mathbf{B}_{X/\mathfrak{S}}$. 
	More precisely, objects of $\RConv(X/\mathfrak{S})$ are same as those of $\Conv(X/\mathfrak{S})$. For two objects $(\mathfrak{Z},v),(\mathfrak{T},u)$ of $\RConv(X/\mathfrak{S})$, we have
	\begin{equation}\label{Hom RConv}
	\Hom_{\RConv(X/\mathfrak{S})}( (\mathfrak{Z},v),(\mathfrak{T},u))=\varinjlim_{ (\mathfrak{Z}',\varphi)\in \mathbf{B}_{\mathfrak{Z}}^{\circ}} \Hom_{\Conv(X/\mathfrak{S})}( (\mathfrak{Z}',v\circ \varphi_{0}),(\mathfrak{T},u)).
	\end{equation}

	We denote by $\QQ_{X/\mathfrak{S}}$ the canonical functor
	\begin{equation} \label{functor QXS}
		\QQ_{X/\mathfrak{S}}: \Conv(X/\mathfrak{S})\to \RConv(X/\mathfrak{S}).
	\end{equation}
	For an object $\mathfrak{T}$ (resp. a morphism $f$) of $\Conv(X/\SS)$, we write $\mathfrak{T}^{\rig}=\QQ_{X/\SS}(\mathfrak{T})$ (resp. $f^{\rig}=\QQ_{X/\SS}(f)$), if there is no risk of confusion.
\end{nothing}	
\begin{nothing}
	We denote by $\widehat{\Conv}(X/\mathfrak{S})$ (resp. $\widehat{\RConv}(X/\mathfrak{S})$) the category of presheaves on $\Conv(X/\mathfrak{S})$ (resp. $\RConv(X/\mathfrak{S})$) and by $\QQ_{X/\mathfrak{S}}^{*}:\widehat{\RConv}(X/\mathfrak{S}) \to \widehat{\Conv}(X/\mathfrak{S})$ the functor defined by $\mathscr{F}\mapsto \mathscr{F}\circ \QQ_{X/\mathfrak{S}}$.
	The functor $\QQ_{X/\SS*}$ admits a left adjoint $\QQ_{X/\mathfrak{S}!}$ (\cite{SGAIV} I 5.1) defined as follows. 
	
	For any object $\mathfrak{T}^{\rig}$ of $\RConv(X/\SS)$, we denote by $I_{\QQ}^{\mathfrak{T}^{\rig}}$ the category whose object are pairs $(\mathfrak{Z},g)$ consisting of an object $\mathfrak{Z}$ of $\Conv(X/\SS)$ and a morphism $g:\mathfrak{T}^{\rig}\to \mathfrak{Z}^{\rig}$ of $\RConv(X/\SS)$. A morphism $(\mathfrak{Z}',g')\to (\mathfrak{Z},g)$ is given by a morphism $\mu:\mathfrak{Z}'\to \mathfrak{Z}$ of $\Conv(X/\SS)$ such that $g=\mu^{\rig}\circ g'$. 
	Then we have (\cite{SGAIV} I 5.1.1) 
	\begin{equation} \label{QXSS!}
		\QQ_{X/\SS!}(\mathscr{F})(\mathfrak{T}^{\rig})=\varinjlim_{(\mathfrak{Z},g)\in (I_{\QQ}^{\mathfrak{T}^{\rig}})^{\circ}} \mathscr{F}(\mathfrak{Z}).
	\end{equation}

	Moreover, we have a commutative diagram (\cite{SGAIV} I 1.5.4)
	\begin{equation} \label{commutative diagram Q!}
		\xymatrix{
			\Conv(X/\mathfrak{S}) \ar[r]^{\QQ_{X/\mathfrak{S}}} \ar[d] & \RConv(X/\mathfrak{S}) \ar[d] \\
			\widehat{\Conv}(X/\mathfrak{S}) \ar[r]^{\QQ_{X/\mathfrak{S}!}} & \widehat{\RConv}(X/\mathfrak{S}) 
		}
	\end{equation}
	where the vertical functors are the canonical functors. 
\end{nothing}

\begin{prop} \label{QXSS! exact}
	\textnormal{(i)} The category  $(I_{\QQ}^{\mathfrak{T}^{\rig}})^{\circ}$ is filtered \textnormal{(\cite{SGAIV} I 2.7)}.

	\textnormal{(ii)} The functor $\QQ_{X/\mathfrak{S}!}$ is left exact (and hence is exact).

	\textnormal{(iii)} Fiber products are representable in $\RConv(X/\SS)$ and $\QQ_{X/\SS}$ commutes with fiber products.
\end{prop}
\begin{proof} We verify following conditions of (\cite{SGAIV} I 2.7) for $(I_{\QQ}^{\mathfrak{T}^{\rig}})^{\circ}$:

	(PS1) Given two morphisms $u:(\ZZ_{1},g_{1})\to (\ZZ_{0},g_{0})$ and $v:(\ZZ_{2},g_{2})\to (\ZZ_{0},g_{0})$ of $I_{\QQ}^{\mathfrak{T}^{\rig}}$, by \eqref{Hom RConv} and \eqref{right multiplicative in Conv}, there exists an admissible blow-up $\mathfrak{T}'$ of $\mathfrak{T}$ and morphisms $\mathfrak{g}_{i}:\mathfrak{T}'\to \ZZ_{i}$ of $\Conv(X/\SS)$ such that $\mathfrak{g}_{i}^{\rig}=g_{i}$ and that $u\circ \mathfrak{g}_{1}=\mathfrak{g}_{0}=v\circ \mathfrak{g}_{2}$. 
	Then we obtain a morphism $\mathfrak{h}:\mathfrak{T}'\to \ZZ_{1}\times_{\ZZ_{0}}\ZZ_{2}$ of $\Conv(X/\SS)$ \eqref{def fibered product} and an object $(\ZZ_{1}\times_{\ZZ_{0}}\ZZ_{2},\mathfrak{h}^{\rig})$ of $I_{\QQ}^{\mathfrak{T}^{\rig}}$ dominant $(\ZZ_{i},g_{i})$ for $i=1,2$. The diagram
	\begin{displaymath}
		\xymatrix{
			(\ZZ_{1}\times_{\ZZ_{0}}\ZZ_{2},\mathfrak{h}^{\rig}) \ar[r] \ar[d] & (\ZZ_{2},g_{2}) \ar[d]\\
			(\ZZ_{1},g_{1}) \ar[r] & (\ZZ_{0},g_{0})
		}
	\end{displaymath}
	commutes. 
	Then condition (PS1) follows.

	(PS2) Let $u,v:(\mathfrak{Y},g)\to (\mathfrak{Z},h)$ be two morphisms of $I_{\QQ}^{\mathfrak{T}^{\rig}}$. There exists an admissible blow-up $(\mathfrak{T}',\varphi)$ of $\mathfrak{T}$ and a morphism $\mathfrak{g}:\mathfrak{T}'\to \YY$ of $\Conv(X/\SS)$ such that 
	$u\circ \mathfrak{g}=v\circ \mathfrak{g}$ in $\Conv(X/\SS)$, denoted by $\mathfrak{h}$.
	We have $\mathfrak{g}^{\rig}=g, \mathfrak{h}^{\rig}=h$.
	Then $(\mathfrak{T}',\varphi^{\rig})$ defines an object of $I_{\QQ}^{\mathfrak{T}^{\rig}}$ and $\mathfrak{h}$ (resp. $\mathfrak{g}$) defines morphism from $(\mathfrak{T}',\varphi^{\rig})$ to $(\mathfrak{Y},g)$ (resp. $(\mathfrak{Z},h)$). 
	Condition (PS2) follows from $u\circ \mathfrak{g}=v\circ \mathfrak{g}=\mathfrak{h}$. 

	It is clear that $I_{\QQ}^{\mathfrak{T}^{\rig}}$ is non-empty. Given two objects $(\mathfrak{Z}_{1},g_{1})$ and $(\mathfrak{Z}_{2},g_{2})$, there exists an admissible blow-up $\mathfrak{T}'$ of $\mathfrak{T}$ and morphisms $\mathfrak{g}_{i}:\mathfrak{T}'\to \ZZ_{i}$ of $\Conv(X/\SS)$ for $i=1,2$ such that $\mathfrak{g}_{i}^{\rig}=g_{i}$. Hence, $I_{\QQ}^{\mathfrak{T}^{\rig}}$ is connected. Then assertion (i) follows. 

Assertion (ii) follows from (i). Assertion (iii) follows from (ii), \eqref{commutative diagram Q!} and the fact that fiber product is representable in $\Conv(X/\mathfrak{S})$.
\end{proof}

\begin{nothing} \label{def fibered product rig}	
	The canonical functor $\Conv(X/\mathfrak{S})\to \mathbf{S}^{\diamond}_{/\mathfrak{S}}$ defined by $(\mathfrak{T},u)\mapsto \mathfrak{T}$ induces a functor
	\begin{equation}
		\RConv(X/\mathfrak{S})\to \mathbf{R}_{/\mathfrak{S}^{\rig}}.
	\end{equation}
	In view of the definition of fiber product in $\mathbf{R}_{/\mathfrak{S}^{\rig}}$ (\cite{Ab10} 4.1.13), the above functor commutes with fiber products. 
\end{nothing}

\begin{nothing} \label{def topology Conv rig}
	We say that a family of morphisms  $\{(\mathfrak{T}_{i},u_{i})^{\rig}\to (\mathfrak{T},u)^{\rig}\}_{i\in I}$ of $\RConv(X/\mathfrak{S})$ is an admissible (resp. fppf) covering if its image in its image $\{\mathfrak{T}_{i}^{\rig}\to \mathfrak{T}^{\rig}\}_{i\in I}$ in $\MR$ is an admissible (resp. fppf) covering (\ref{def ad covering}, \ref{def rfppf covering}).
	By \ref{def rfppf covering} and \ref{def fibered product rig}, admissible (resp. fppf) coverings form a pretopology. 
	For $\tau=\ad$ (resp. $\fppf$), we call \textit{rigid convergent topos of $X$ over $\SS$} and denote by $(X/\SS)_{\rconv,\tau}$ the topos of sheaves of sets on $\RConv(X/\SS)$, equipped with the topology associated to the pretopology defined by admissible (resp. fppf) coverings.
\end{nothing}

\begin{nothing} \label{descent datum rig}
	Let $(\mathfrak{T},u)$ be an object of $\Conv(X/\mathfrak{S})$. The canonical functor \eqref{S diamond}
	\begin{displaymath}
		r_{\mathfrak{T}}: \mathbf{S}_{/\mathfrak{T}}^{\diamond} \to \Conv(X/\mathfrak{S}) \qquad (f:\mathfrak{T}'\to \mathfrak{T})\mapsto (\mathfrak{T}',u\circ f_{0}).
	\end{displaymath}
	sends admissible blow-ups to $\mathbf{B}_{X/\mathfrak{S}}$ and hence induces a functor 
	\begin{equation} \label{RT to RConv}
		r_{\mathfrak{T}^{\rig}}: \mathbf{R}_{/\mathfrak{T}^{\rig}} \to \RConv(X/\mathfrak{S}).
	\end{equation}

	The restriction of \eqref{RT to RConv} to $\Ad_{/\mathfrak{T}^{\rig}}$ (resp. $\Rf_{/\mathfrak{T}^{\rig}}$) is cocontinuous for the admissible (resp. fppf) topology and it induces a morphism of topoi
	\begin{equation}
		s_{\mathfrak{T}^{\rig}}: \mathfrak{T}^{\rig}_{\tau} \to (X/\mathfrak{S})_{\conv,\tau}, \qquad \tau\in \{\ad,\fppf\}. 
	\end{equation}

	For any sheaf $\mathscr{F}$ of $(X/\mathfrak{S})_{\conv,\tau}$, we set $\mathscr{F}_{\mathfrak{T}^{\rig}}=s_{\mathfrak{T}^{\rig}}^{*}(\mathscr{F})$.
	For any morphism $f:\mathfrak{T}'^{\rig}\to \mathfrak{T}^{\rig}$ of $\RConv(X/\mathfrak{S})$, we have a canonical morphism 
	\begin{equation}
		\mathscr{F}_{\mathfrak{T}^{\rig}}\to f_{\tau*}(\mathscr{F}_{\mathfrak{T}'^{\rig}})
	\end{equation}
	and we denote its adjoint by
	\begin{equation} \label{transition gamma f rig}
		\gamma_{f}:f_{\tau}^{*}(\mathscr{F}_{\mathfrak{T}^{\rig}}) \to \mathscr{F}_{\mathfrak{T}'^{\rig}},
	\end{equation}
	where $f_{\tau}:\mathfrak{T}'^{\rig}_{\tau}\to \mathfrak{T}^{\rig}_{\tau}$ denotes the functorial morphism for $\tau$-topology \eqref{functorial ftau rig}. 
	If the morphism of underlying rigid spaces of $f$ belongs to $\Ad_{/\mathfrak{T}^{\rig}}$ (resp. $\Rf_{/\mathfrak{T}^{\rig}}$), $f_{\tau}$ is the localisation morphism at $\mathfrak{T}'$ \eqref{fppf to ad and functorial} and then $\gamma_{f}$ is an isomorphism. 
	If $g:\mathfrak{T}''^{\rig}\to \mathfrak{T}'^{\rig}$ is another morphism of $\RConv(X/\mathfrak{S})$, one verifies that $\gamma_{g\circ f}=\gamma_{f}\circ f^{*}_{\tau}(\gamma_{g})$.
	
	By repeating the proof of \ref{descent data sheaf}, we have the following description for a sheaf of $(X/\mathfrak{S})_{\rconv,\tau}$.
\end{nothing}

\begin{prop} \label{descent data sheaf rig}
	For $\tau\in \{\ad,\fppf\}$, a sheaf $\mathscr{F}$ of $(X/\mathfrak{S})_{\rconv,\tau}$ is equivalent to the following data:

	\textnormal{(i)} For every object $\mathfrak{T}^{\rig}$ of $\RConv(X/\mathfrak{S})$, a sheaf $\mathscr{F}_{\mathfrak{T}^{\rig}}$ of $\mathfrak{T}^{\rig}_{\tau}$,

	\textnormal{(ii)} For every morphism $f:\mathfrak{T}'^{\rig}\to \mathfrak{T}^{\rig}$ of $\RConv(X/\SS)$, a morphism $\gamma_{f}$ \eqref{transition gamma f rig},

	subject to the following conditions

	\textnormal{(a)} If $f$ is the identity morphism of $(\mathfrak{T},u)$, $\gamma_{f}$ is the identity morphism.

	\textnormal{(b)} If the underlying morphism $f:\mathfrak{T}'^{\rig}\to \mathfrak{T}^{\rig}$ of coherent rigid spaces is a morphism of $\Ad_{/\mathfrak{T}^{\rig}}$ (resp. $\Rf_{/\mathfrak{T}^{\rig}}$), then $\gamma_{f}$ is an isomorphism.

	\textnormal{(c)} If $f$ and $g$ are two composable morphisms, then we have $\gamma_{g\circ f}=\gamma_{f}\circ f_{\tau}^{*}(\gamma_{g})$.
\end{prop}

\begin{nothing}
	Note that the fppf topology on $\RConv(X/\SS)$ is finer than the admissible topology. Equipped with the fppf topology on the source and the admissible topology on the target, the identical functor $\id:\RConv(X/\SS)\to \RConv(X/\SS)$ is cocontinuous. By \ref{generality morphism topos}, it induces a morphism of topoi
	\begin{equation} \label{alpha ad fppf}
		\alpha_{r}: (X/\SS)_{\rconv,\fppf}\to (X/\SS)_{\rconv,\ad}.
	\end{equation}
If $\mathscr{F}$ is a sheaf of $(X/\SS)_{\rconv,\fppf}$, $\alpha_{r*}(\mathscr{F})$ is equal to $\mathscr{F}$ as presheaves. If $\mathscr{G}$ is a sheaf of $(X/\SS)_{\rconv,\ad}$, then $\alpha_{r}^{*}(\mathscr{G})$ is the sheafification of $\mathscr{G}$ with respect to the fppf topology.
\end{nothing}

\begin{nothing} \label{rho XS rconv}
	Equipped with the Zariski topology on the source and the admissible on the target, the canonical functor $\QQ_{X/\mathfrak{S}}$ \eqref{functor QXS} is clearly continuous. Since the functor $\QQ_{X/\mathfrak{S}!}$ and the shefification functor are exact \eqref{def fibered product rig}, then we have a morphism of topoi
	\begin{equation}
		\rho_{X/\mathfrak{S}}: (X/\mathfrak{S})_{\rconv,\ad} \to (X/\mathfrak{S})_{\conv,\zar}
	\end{equation}
	defined by $\rho_{X/\mathfrak{S}*}=\QQ^{*}_{X/\mathfrak{S}}$ and $\rho_{X/\mathfrak{S}}^{*}=a\circ \QQ_{X/\mathfrak{S}!}$ \eqref{def fibered product rig}, where $a$ denotes the sheafification functor.
	
	For any object $\mathfrak{T}$ of $\Conv(X/\mathfrak{S})$ and any sheaf $\mathscr{F}$ of $(X/\mathfrak{S})_{\rconv,\ad}$, we have \eqref{def rhoX}
	\begin{equation} \label{description pushforward rho}
		(\rho_{X/\mathfrak{S}*}(\mathscr{F}))_{\mathfrak{T}}=\rho_{\mathfrak{T}*}(\mathscr{F}_{\mathfrak{T}^{\rig}}).
	\end{equation}
	Let $f:\mathfrak{Z}\to \mathfrak{T}$ be a morphism of $\Conv(X/\mathfrak{S})$ and $\beta_{f^{\rig}},\gamma_{f^{\rig}}$ (resp. $\beta_{f}$, $\gamma_{f}$) transition morphisms of $\mathscr{F}$ associated to $f^{\rig}$ (resp. $\rho_{X/\mathfrak{S}*}(\mathscr{F})$ associated to $f$) (\ref{construction descent data}, \ref{descent datum rig}). 
	Via \eqref{description pushforward rho}, we have $\beta_{f}=\rho_{\mathfrak{T}*}(\beta_{f^{\rig}})$. 
	Then we deduce that $\gamma_{f}$ coincides with the composition 
	\begin{equation} \label{description pushforward rho gamma}
	f^{*}_{\zar}(\rho_{\mathfrak{T}*}(\mathscr{F}_{\mathfrak{T}^{\rig}}))\to \rho_{\mathfrak{Z}*}(f^{\rig*}_{\ad}(\mathscr{F}_{\mathfrak{T}^{\rig}})) \xrightarrow{\rho_{\mathfrak{Z}*}(\gamma_{f^{\rig}})} \rho_{\mathfrak{Z}*}(\mathscr{F}_{\mathfrak{Z}^{\rig}}). 
\end{equation}
\end{nothing}

\begin{prop} \label{lemma pullback rhoXS}
	Let $\mathscr{F}$ be a sheaf of $(X/\mathfrak{S})_{\conv,\zar}$ and $(\mathfrak{T},u)$ an object of $\Conv(X/\mathfrak{S})$. There exists a canonical isomorphism
	\begin{equation} \label{cal pullback by rho}
		(\rho_{X/\mathfrak{S}}^{*}(\mathscr{F}))_{\mathfrak{T}^{\rig}}
		\xrightarrow{\sim}	
		\varinjlim_{(\mathfrak{T}',\varphi)\in \MB_{\mathfrak{T}}^{\circ}} \mu_{\varphi}^{*}(\mathscr{F}_{\mathfrak{T}'}).
	\end{equation}
\end{prop}
\begin{proof}
	Let $\mathcal{U}$ be an object of $\Ad_{/\mathfrak{T}^{\rig}}$ that we consider as an object of $\RConv(X/\SS)$ via $r_{\mathfrak{T}^{\rig}}$ \eqref{RT to RConv}. By \ref{QXSS!}, $(\rho_{X/\mathfrak{S}}^{*}(\mathscr{F}))_{\mathfrak{T}^{\rig}}$ is the sheaf associated to the presheaf on $\Ad_{/\mathfrak{T}^{\rig}}$
	\begin{equation}
		\mathcal{U}\mapsto \varinjlim_{(\mathfrak{Z},g)\in (I^{\mathcal{U}}_{\QQ})^{\circ}} \mathscr{F}(\mathfrak{Z}).
	\end{equation}

	We denote by $J_{\QQ}^{\mathcal{U}}$ the category of quadruple $(\mathfrak{T}',\varphi,\UU,g)$ consisting of an admissible blow-up $(\mathfrak{T}',\varphi)$ of $\mathfrak{T}$, an open formal subscheme $\UU$ of $\mathfrak{T}'$ and an open immersion $g:\mathcal{U}\to \UU^{\rig}$ over $\mathfrak{T}^{\rig}$. 
	A morphism $(\mathfrak{T}'_{1},\varphi_{1},\UU_{1},g_{1})$ to $(\mathfrak{T}'_{2},\varphi_{2},\UU_{2},g_{2})$ is a morphism $\mathfrak{T}'_{1}\to \mathfrak{T}'_{2}$ of $\MB_{\mathfrak{T}}$ sending $\UU_{1}$ to $\UU_{2}$ compatible with $g_{1},g_{2}$. 
	The category $J_{\QQ}^{\mathcal{U}}$ is clearly fibered over $\MB_{\mathfrak{T}}$:
	\begin{displaymath}
		J_{\QQ}^{\mathcal{U}}\to \MB_{\mathfrak{T}},\qquad (\mathfrak{T}',\varphi,\UU,g)\mapsto (\mathfrak{T}',\varphi).
	\end{displaymath}
	For any admissible blow-up $(\mathfrak{T}',\varphi)$ of $\mathfrak{T}$, we denotes its fiber by $J_{\QQ,\mathfrak{T}'}^{\mathcal{U}}$. The sheaf $\mu_{\varphi}^{*}(\mathscr{F}_{\mathfrak{T}'})$ is associated to the presheaf on $\Ad_{/\mathfrak{T}^{\rig}}$ 
	\begin{equation}
		\mathcal{U}\mapsto \varinjlim_{(\UU,g)\in (J_{\QQ,\mathfrak{T}'}^{\mathcal{U}})^{\circ}} \mathscr{F}(r_{\mathfrak{T}}(\UU)).
	\end{equation}
	
	Then, the left hand side of \eqref{cal pullback by rho} is the sheaf on $\Ad_{/\mathfrak{T}^{\rig}}$ associated to the presheaf
	\begin{equation}
		\mathcal{U}\mapsto \varinjlim_{(\mathfrak{T}',\varphi,\UU,g)\in (J_{\QQ}^{\mathcal{U}})^{\circ}} \mathscr{F}(r_{\mathfrak{T}}(\UU)).
	\end{equation}

	We have a canonical functor \eqref{descent datum rig}
	\begin{displaymath}
		r: J_{\QQ}^{\mathcal{U}}\to I_{\QQ}^{\mathcal{U}} \qquad (\mathfrak{T}',\varphi,\UU,g)\mapsto (r_{\mathfrak{T}}(\UU),r_{\mathfrak{T}^{\rig}}(g))
	\end{displaymath}
	We denote by $J$ the full subcategory of $J_{\QQ}^{\mathcal{U}}$ consisting of objects such that $g$ is an isomorphism. Then each morphism of $J$ is Cartesian. 
	Each category $(J_{\QQ,\mathfrak{T}'}^{\mathcal{U}})^{\circ}$ is filtered by (\cite{SGAIV} I 5.2). We deduce that $(J_{\QQ}^{\mathcal{U}})^{\circ}$ is filtered. 
	It is clear that $J^{\circ}$ is cofinal in $(J_{\QQ}^{\mathcal{U}})^{\circ}$ and hence is filtered (\cite{SGAIV} I 8.1.3 a).
	
	To prove the assertion, it suffices to show that the induced functor $r:J^{\circ}\to (I_{\QQ}^{\mathcal{U}})^{\circ}$ is cofinal in the sense of (\cite{SGAIV} I 8.1.1). 
	By (\cite{Ab10} 4.2.2), for any object $(\mathfrak{Z},g:\mathcal{U}\to \mathfrak{Z}^{\rig})$ of $I_{\QQ}^{\mathcal{U}}$, there exists a morphism $h:\UU \to \mathfrak{Z}$ of $\Conv(X/\mathfrak{S})$ with an open formal subscheme $\UU$ of some admissible blow-up $\mathfrak{T}'$ of $\mathfrak{T}$, such that $g=h^{\rig}$, i.e. condition (F1) of (\cite{SGAIV} I 8.1.3) is satisfied.
	Given an object $(\mathfrak{Z},g)$ of $I_{\QQ}^{\mathcal{U}}$, an object $(\mathfrak{T}',\varphi,\mathfrak{U},h)$ of $J$ and two morphisms $f_{1},f_{2}:(r_{\mathfrak{T}}(\UU),r_{\mathfrak{T}^{\rig}}(h)) \to (\mathfrak{Z},g)$, then $f_{1}^{\rig}=f_{2}^{\rig}$ in $\RConv(X/\SS)$ since $h$ is an isomorphism. By (\cite{Ab10} 3.5.9), we deduce that $f_{1}=f_{2}$, i.e. condition (F2) of (\cite{SGAIV} I 8.1.3) is satisfied. Then, the assertion follows from (\cite{SGAIV} I 8.1.3 b).
\end{proof}

\begin{coro} \label{rconv description}
		\textnormal{(i)} The canonical morphism $\rho_{X/\SS}^{*}\rho_{X/\SS*}\to \id$ is an isomorphism.

		\textnormal{(ii)} The functor $\rho_{X/\SS*}$ is fully faithful and its essential image consists of sheaves $\mathscr{F}=\{\mathscr{F}_{\mathfrak{T}},\beta_{f}\}$ such that $\beta_{f}$ is an isomorphism for all morphism $f$ of $\MB_{X/\SS}$.  
\end{coro}

\begin{proof}
	(i) Let $\mathscr{F}$ be a sheaf of $(X/\SS)_{\rconv,\ad}$. Via \eqref{description pushforward rho} and \eqref{cal pullback by rho}, we consider the evaluation of $\rho_{X/\SS}^{*}\rho_{X/\SS*}(\mathscr{F}) \to \mathscr{F}$ at an object $\mathfrak{T}^{\rig}$ of $\RConv(X/\SS)$
	\begin{equation} \label{adjonction rho T}
		\varinjlim_{(\mathfrak{T}',\varphi)\in \MB_{\mathfrak{T}}^{\circ}} \mu_{\varphi}^{*}(\mu_{\varphi*}(\mathscr{F}_{\mathfrak{T}^{\rig}})) \to
		\mathscr{F}_{\mathfrak{T}^{\rig}}.
	\end{equation}
	In view of the proof of \ref{lemma pullback rhoXS}, the morphism $\mu_{\varphi}^{*}(\mu_{\varphi*}(\mathscr{F}_{\mathfrak{T}^{\rig}})) \to \mathscr{F}_{\mathfrak{T}^{\rig}}$ deduced from \eqref{adjonction rho T} is nothing but the adjunction morphism. Then the assertion follows from (\cite{Ab10} 4.5.27 and 4.5.28).

	(ii) By (i), the functor $\rho_{X/\SS*}$ is fully faithful. By \eqref{rho XS rconv}, the essential image of $\rho_{X/\SS*}$ has the desired property. 
	Let $\mathscr{G}$ be a sheaf of $(X/\SS)_{\conv,\zar}$ satisfying the desired property. By (\cite{Ab10} 4.5.22 and 4.5.27), we deduce that the evaluation of the canonical morphism
	\begin{equation} \label{adjoint map rho XSS G}
		\mathscr{G}\to \rho_{X/\SS*}\rho_{X/\SS}^{*}(\mathscr{G})
	\end{equation}
	at each object of $\RConv(X/\SS)$, is an isomorphism. Then the assertion follows.
\end{proof}

\begin{nothing} \label{setting functorial functor rig}
	Let $g:\SS'\to \SS$ be a morphism of $\MS^{\diamond}$, $X'$ an $S'$-scheme and $f:X'\to X$ a morphism compatible with $g$ as in \ref{def functorial functor}. 
	The canonical functor $\varphi:\Conv(X'/\SS')\to \Conv(X/\SS)$ defined by $(\mathfrak{T},u)\mapsto (\mathfrak{T},f\circ u)$ \eqref{functorial functor}, sends admissible blow-ups to admissible blow-ups. Then, $\varphi$ induces a functor that we denote by
	\begin{equation} \label{functorial functor rig}
		\psi: \RConv(X'/\SS')\to \RConv(X/\SS).
	\end{equation}
	
	Since $\varphi$ commutes with fiber products, the same holds for $\psi$ by \ref{def fibered product rig}.
	In view of \ref{cont cocont func}(i) and \eqref{Hom RConv}, one verifies that the functor $\psi$ is continuous and cocontinuous for admissible (resp. fppf) topology in the same way as in \ref{cont cocont func}.

	By \ref{generality morphism topos}, the functor $\psi$ \eqref{functorial functor rig} induces morphisms of topoi
	\begin{eqnarray}
		\label{morphism of topoi functorial rig} f_{\rconv,\tau}:(X'/\SS')_{\rconv,\tau}\to (X/\SS)_{\rconv,\tau}, \qquad \tau\in \{\ad,\fppf\}
	\end{eqnarray}
	such that the pullback functor is induced by the composition with $\psi$. For a sheaf $\mathscr{F}$ of $(X/\mathfrak{S})_{\rconv,\tau}$ and an object $\mathfrak{T}$ of $\RConv(X'/\mathfrak{S}')$, we have
	\begin{equation}
		(f_{\rconv,\tau}^{*}(\mathscr{F}))_{\mathfrak{T}^{\rig}}=\mathscr{F}_{\psi(\mathfrak{T}^{\rig})}.\label{description of pullback rig}
	\end{equation}
	For any morphism $g$ of $\RConv(X'/\mathfrak{S}')$, the transition morphism of $f_{\rconv,\tau}^{*}(\mathscr{F})$ associated to $g$ is equal to the transition morphism of $\mathscr{F}$ associated to $\psi(g)$.

	In view of the description of inverse image functors, we deduce the following result.
\end{nothing}

\begin{coro} \label{compatibility functorial morphism rho}
	Keep the assumption and notation of \ref{def functorial functor} and of \ref{setting functorial functor rig}. The diagram
	\begin{displaymath}
		\xymatrix{
			(X'/\SS')_{\rconv,\ad} \ar[r]^{f_{\rconv,\ad}} \ar[d]_{\rho_{X'/\SS'}} & (X/\SS)_{\rconv,\ad} \ar[d]^{\rho_{X/\SS}} \\
			(X'/\SS')_{\conv,\zar} \ar[r]^{f_{\conv,\zar}} & (X/\SS)_{\conv,\zar}
		}
	\end{displaymath}
	is commutative up to canonical isomorphisms.
\end{coro}

%

\begin{nothing}
	We set $\mathscr{O}_{X/\mathfrak{S}}^{\rig}=\rho_{X/\mathfrak{S}}^{*}(\mathscr{O}_{X/\mathfrak{S}}[\frac{1}{p}])$. By \ref{lemma pullback rhoXS}(i), for any object $\mathfrak{T}^{\rig}$ of $\RConv(X/\mathfrak{S})$, we have a canonical isomorphism
	\begin{equation}
		(\mathscr{O}_{X/\mathfrak{S}}^{\rig})_{\mathfrak{T}^{\rig}}\xrightarrow{\sim} \mathscr{O}_{\mathfrak{T}^{\rig}}.
	\end{equation}
	Then by fppf descent (\cite{Ab10} 5.11.11), the presheaf $\mathscr{O}_{X/\mathfrak{S}}^{\rig}$ is also a sheaf for the fppf (resp. admissible) topology.
	For $\tau\in \{\ad,\fppf\}$, if $\mathscr{F}$ is an $\mathscr{O}_{X/\mathfrak{S}}^{\rig}$-module of $(X/\mathfrak{S})_{\rconv,\tau}$, $\mathscr{F}_{\mathfrak{T}^{\rig}}$ is an $\mathscr{O}_{\mathfrak{T}^{\rig}}$-module. For any morphism $f:\mathfrak{T}'^{\rig}\to \mathfrak{T}^{\rig}$ of $\RConv(X/\mathfrak{S})$, the transition morphism $\gamma_{f}$ \eqref{descent data sheaf rig} extends to an $\mathscr{O}_{\mathfrak{T}'^{\rig}}$-linear morphism \eqref{fppf to ad and functorial}
	\begin{equation} \label{lin transition morphism rig}
		c_{f}: f^{*}_{\tau}(\mathscr{F}_{\mathfrak{T}^{\rig}})\to \mathscr{F}_{\mathfrak{T}'^{\rig}}.
	\end{equation}
	
	In view of \ref{descent data sheaf rig}, we deduce the following description for $\mathscr{O}_{X/\SS}^{\rig}$-modules.
\end{nothing}

\begin{prop} \label{description of mods rig}
		For $\tau\in\{\zar,\fppf\}$, an $\mathscr{O}_{X/\mathfrak{S}}^{\rig}$-module of $(X/\mathfrak{S})_{\rconv,\tau}$ is equivalent to the following data:

		\textnormal{(i)} For every object $\mathfrak{T}^{\rig}$ of $\RConv(X/\mathfrak{S})$, an $\mathscr{O}_{\mathfrak{T}^{\rig}}$-module $\mathscr{F}_{\mathfrak{T}}$ of $\mathfrak{T}_{\tau}$,

		\textnormal{(ii)} For every morphism $f:\mathfrak{T}'^{\rig}\to \mathfrak{T}^{\rig}$ of $\RConv(X/\mathfrak{S})$, an $\mathscr{O}_{\mathfrak{T}'^{\rig}}$-linear morphism $c_{f}$ \eqref{lin transition morphism rig}.

	which is subject to the following conditions

	\textnormal{(a)} If $f$ is the identity morphism, then $c_{f}$ is the identity.

	\textnormal{(b)} If the underlying morphism $f:\mathfrak{T}'^{\rig}\to \mathfrak{T}^{\rig}$ of coherent rigid spaces is a morphism of $\Ad_{/\mathfrak{T}^{\rig}}$ (resp. $\Rf_{/\mathfrak{T}^{\rig}}$), then $c_{f}$ is an isomorphism.

	\textnormal{(c)} If $f$ and $g$ are two composable morphisms, then we have $c_{g\circ f}=c_{f}\circ f^{*}_{\tau}(c_{g})$.
\end{prop}

\begin{definition} \label{def coh crystal rconv}
	Let $\mathscr{F}$ be an $\mathscr{O}_{X/\mathfrak{S}}^{\rig}$-module of $(X/\mathfrak{S})_{\rconv,\tau}$.

	\textnormal{(i)} We say that $\mathscr{F}$ is \textit{coherent} if for every object $\mathfrak{T}^{\rig}$ of $\RConv(X/\mathfrak{S})$, $\mathscr{F}_{\mathfrak{T}^{\rig}}$ is coherent \eqref{coherent fppf rig}.

	\textnormal{(ii)} We say that $\mathscr{F}$ is a \textit{crystal} if for every morphism $f$ of $\RConv(X/\mathfrak{S})$, $c_{f}$ is an isomorphism.	
\end{definition}

By fppf descent (\cite{Ab10} 5.11.11), the direct image and inverse image functors of $\alpha_{r}$ induce equaivalences of categories quais-inverse to each others between the category of coherent crystals of $\mathscr{O}_{X/\SS}^{\rig}$-modules of $(X/\mathfrak{S})_{\rconv,\ad}$ and of $(X/\mathfrak{S})_{\rconv,\fppf}$.

\begin{prop} \label{rho conv coh}
	The direct image and inverse image functors of $\rho_{X/\mathfrak{S}}$ induce equivalences of categories quasi-inverse to each other between the category of coherent crystals of $\mathscr{O}_{X/\mathfrak{S}}[\frac{1}{p}]$-modules of $(X/\mathfrak{S})_{\conv,\zar}$ and that of coherent crystals of $\mathscr{O}_{X/\mathfrak{S}}^{\rig}$-modules of $(X/\mathfrak{S})_{\rconv,\ad}$.
\end{prop}
\begin{proof}
	Let $\mathscr{F}$ be a coherent crystal of $\mathscr{O}_{X/\SS}^{\rig}$-modules of $(X/\mathfrak{S})_{\rconv,\ad}$.
	By \ref{rhoXX coh equi} and \eqref{description pushforward rho}, $\rho_{X/\mathfrak{S}*}(\mathscr{F})$ is coherent. In view of \eqref{base change rho coherent} and \eqref{description pushforward rho gamma}, we deduce that it is also a crystal. By \ref{rconv description}(i), $\rho_{X/\SS}^{*}\rho_{X/\SS*}(\mathscr{F})\to \mathscr{F}$ is an isomorphism.

	Let $\mathscr{G}$ be a coherent crystal of $\mathscr{O}_{X/\mathfrak{S}}[\frac{1}{p}]$-modules of $(X/\mathfrak{S})_{\conv,\zar}$ and $\mathscr{H}=\rho_{X/\SS}^{*}(\mathscr{G})$. 
	By Tate's acyclicity (\cite{Ab10} 3.5.5), $\mathscr{G}$ is contained in the essential image of $\rho_{X/\SS*}$ (\ref{rconv description}(ii)). Then we have a canonical isomorphism $\mathscr{G}\xrightarrow{\sim}\rho_{X/\SS*}(\mathscr{H})$ \eqref{adjoint map rho XSS G}. By \eqref{lemma pullback rhoXS}, \eqref{base change rho coherent} and \eqref{description pushforward rho gamma}, we deduce that $\mathscr{H}$ is coherent and is a crystal. 
	Then the assertion follows.
\end{proof}


\begin{nothing} \label{setting XtoY}
	Let $g:X\to Y$ be a morphism of $S$-schemes, $\mathfrak{T}$ be an object of $\Conv(Y/\mathfrak{S})$ and $\mathfrak{T}^{\rig}$ its image in $\RConv(Y/\SS)$. 
	By fppf descent for morphisms of coherent rigid spaces (\cite{Ab10} 5.12.1), the presheaf associated to $\mathfrak{T}^{\rig}$ is a sheaf for the fppf (resp. Zariski) topology that we denote by $\widetilde{\mathfrak{T}}^{\rig}$.
	We set $X_{T_{0}}=X\times_{Y}T_{0}$ and for $\tau\in \{\ad,\fppf\}$, we denote by 
	\begin{eqnarray*}
		&g_{X/\mathfrak{T},\tau}:&(X_{T_{0}}/\mathfrak{T})_{\rconv,\tau}\to (T_{0}/\mathfrak{T})_{\rconv,\tau},\\
		&\omega_{\mathfrak{T}^{\rig}}:&(X_{T_{0}}/\mathfrak{T})_{\rconv,\tau} \to (X/\mathfrak{S})_{\rconv,\tau}
	\end{eqnarray*}
	the functorial morphisms of topoi \eqref{morphism of topoi functorial rig}. 

	By repeating arguments of \S~\ref{Higher direct image}, we prove the following results in the rigid convergent topos.
\end{nothing}

\begin{lemma}[\ref{localisation by a sheaf}] \label{localisation by a sheaf rig}
	Keep the notation of \ref{setting XtoY}. There exists a canonical equivalence of topoi:
	\begin{equation} \label{localisation of rconv topos}
		(X/\mathfrak{S})_{\rconv,\tau/g^{*}_{\rconv,\tau}(\widetilde{\mathfrak{T}}^{\rig})}\xrightarrow{\sim} (X_{T_{0}}/\mathfrak{T})_{\rconv,\tau}
	\end{equation}
	which identifies the localisation morphism and $\omega_{\mathfrak{T}^{\rig}}$.
\end{lemma}

\begin{lemma}[\ref{evaluation direct image}] \label{evaluation direct image rig}
	For any $\mathscr{O}_{X/\mathfrak{S}}^{\rig}$-module $E$ of $(X/\mathfrak{S})_{\rconv,\tau}$, there exists a canonical isomorphism in $\rD^{+}(\mathfrak{T}^{\rig}_{\tau},\mathscr{O}_{\mathfrak{T}^{\rig}})$
	\begin{equation} \label{calcul of higher direct image evaluation rig}
		(\rR g_{\rconv,\tau*}(E))_{\mathfrak{T}^{\rig}} \xrightarrow{\sim} (\rR g_{X/\mathfrak{T},\tau*}(\omega_{\mathfrak{T}^{\rig}}^{*}(E)))_{\mathfrak{T}^{\rig}}.
	\end{equation}
\end{lemma}

\begin{coro}[\ref{base change coro}] \label{base change coro rig}
Let $\SS'\to \SS$ be a morphism of $\MS^{\diamond}$, $Y'$ an $S'$-scheme and $h:Y'\to Y$ a morphism compatible with $S'\to S$. We set $X'=X\times_{Y}Y'$ and we denote by $g':X'\to Y'$ and $h':X'\to X$ the canonical morphisms: 
	\begin{displaymath}
		\xymatrix{
			X'\ar[r]^{h'} \ar[d]_{g'} & X \ar[d]^{g}\\
			Y'\ar[r]^{h} & Y
		}
	\end{displaymath}
	Then, for any $\mathscr{O}_{X/\SS}[\frac{1}{p}]$-module $E$ of $(X/\SS)_{\conv,\tau}$, the base change morphism
	\begin{equation}
		h_{\rconv,\tau}^{*}(\rR g_{\rconv,\tau*}(E)) \xrightarrow{\sim} \rR g'_{\rconv,\tau*}(h'^{*}_{\rconv,\tau}(E)),
	\end{equation}
	is an isomorphism.
\end{coro}

\begin{coro}[\ref{coh zar fppf}] \label{fppf ad coh}
	Let $\mathscr{E}$ be a coherent crystal of $\mathscr{O}_{X/\SS}^{\rig}$-modules of $(X/\mathfrak{S})_{\rconv,\fppf}$. Then we have \eqref{alpha ad fppf}
	\begin{equation}
		\rR^{i}\alpha_{r*}(\mathscr{E})=0 \qquad \forall~ i\ge 1.
	\end{equation}
\end{coro}

\begin{coro} \label{coh zar ad iso}
	Let $\mathscr{E}$ be a coherent crystal of $\mathscr{O}_{X/\SS}^{\rig}$-modules of $(X/\mathfrak{S})_{\rconv,\ad}$. Then we have
	\begin{equation}
		\rR^{i}\rho_{X/\mathfrak{S}*}(\mathscr{E})=0 \qquad \forall~ i\ge 1.
	\end{equation}
\end{coro}
\begin{proof}
	By \ref{localisation by a sheaf rig}, the Zariski sheaf $\rR^{i}\rho_{X/\SS*}(\mathscr{E})$ on $\RConv(X/\mathfrak{S})$ is associated to the presheaf
	\begin{displaymath}
		\mathfrak{T}^{\rig}\mapsto \rH^{i}( (T_{0}/\mathfrak{T})_{\rconv,\ad}, \mathscr{E}|_{\widetilde{\mathfrak{T}}^{\rig}}).
	\end{displaymath}
	
	By (\cite{SGAIV} V 4.3 and III 4.1), we can replace $\RConv(X/\mathfrak{S})$ by the full subcategory of objects whose underlying rigid space is affinoid, and it suffices to show that for such an object $\mathfrak{T}^{\rig}$ 
	\begin{equation} \label{ad cohomology of a coh affine}
		\rH^{i}((T_{0}/\mathfrak{T})_{\conv,\ad}, \mathscr{E}|_{\widetilde{\mathfrak{T}}})
	\end{equation}
	vanishes for $i\ge 1$. 
	Let $\mathscr{U}=\{\mathfrak{Z}_{i}^{\rig}\to \mathfrak{Z}^{\rig}\}_{i=1}^{m}$ be an admissible covering by affinoids of an affinoids $\mathfrak{Z}^{\rig}$ in $\MR_{/\mathfrak{T}^{\rig}}$. The Čech cohomology $\check{\rH}^{i}(\mathscr{U},\mathscr{E}|_{\widetilde{\mathfrak{Z}}^{\rig}})$ is isomorphic to the cohomology $\rH^{i}(\mathfrak{Z}^{\rig}_{\ad},\mathscr{E}_{\mathfrak{Z}^{\rig}})$ which vanishes by (\cite{Ab10} 4.8.26). 
	Since each admissible covering of $\mathfrak{Z}^{\rig}$ admits a refinement by finitely many affinoids, the vanishing of \eqref{ad cohomology of a coh affine} follows from (\cite{Stacks} 21.11.9).
\end{proof}

\section{Higher direct images of a convergent isocrystal} \label{final sec}
\begin{nothing}
	Let $X$ be a $k$-scheme locally of finite type. 
	For $\tau\in \{\ad,\fppf\}$, the Frobenius homomorphism $\sigma:\rW\to \rW$ induces a morphism of topoi $(X'/\rW)_{\rconv,\tau}\to (X/\rW)_{\rconv,\tau}$ \eqref{morphism of topoi functorial rig}. For any sheaf $\mathscr{E}$ of $(X/\rW)_{\rconv,\tau}$, we denote by $\mathscr{E}'$ the inverse image of $\mathscr{E}$ to $(X'/\rW)_{\rconv,\tau}$. 

	As in \ref{def F-isocrystal}, we call \textit{convergent $F$-isocrystal of $(X/\rW)_{\rconv,\tau}$} a pair $(\mathscr{E},\varphi)$ consisting of a coherent crystal of $\mathscr{O}_{X/\SS}^{\rig}$-modules $\mathscr{E}$ of $(X/\rW)_{\rconv,\tau}$ \eqref{def coh crystal rconv} and an isomorphism
	\begin{equation}
		\varphi:F_{X/k,\rconv,\tau}^{*}(\mathscr{E}')\xrightarrow{\sim} \mathscr{E}.
	\end{equation}

	In this section, we prove the following result about the higher direct image of a convergent ($F$-)isocrystal of rigid convergent topos.
\end{nothing}

\begin{theorem} \label{main theorem}
	Let $g:X\to Y$ be a smooth proper morphism of $k$-schemes locally of finite type and $\mathscr{E}$ (resp. $(\mathscr{E},\varphi)$) a convergent isocrystal (resp. $F$-isocrystal) of $(X/\rW)_{\rconv,\tau}$. Then, $\rR^{i}g_{\rconv,\tau*}(\mathscr{E})$ (resp. $(\rR^{i}g_{\rconv,\tau*}(\mathscr{E}),\rR^{i}g_{\rconv,\tau*}(\varphi))$) is a convergent isocrystal (resp. $F$-isocrystal) of $(Y/\rW)_{\rconv,\tau}$. 
\end{theorem}

By \ref{coh zar ad iso}, \ref{main theorem} and arguments of \ref{pf of higher direct image of con Fiso weak}, we deduce the variant for the convergent topos. 

\begin{coro}\label{coro main thm}
	Keep the assumption of \ref{main theorem}. The higher direct image of a convergent isocrystal (resp. $F$-isocrystal) of $(X/\rW)_{\conv,\zar}$ \eqref{def F-isocrystal} is a convergent isocrystal (resp. $F$-isocrystal) of $(Y/\rW)_{\conv,\zar}$.
\end{coro}

\begin{prop} \label{main theorem weak}
	Keep the notation and assumption of \ref{main theorem}. 
	If $Y$ is moreover smooth over $k$, then $\rR^{i}g_{\rconv,\tau*}(\mathscr{E})$ is a coherent crystal of $\mathscr{O}_{Y/\rW}^{\rig}$-modules.
\end{prop}
\begin{proof}
	We first prove the assertion for the admissible topology. The sheaf $\mathscr{F}=\rho_{X/\rW*}(\mathscr{E})$ is a coherent crystal of $\mathscr{O}_{X/\rW}^{\rig}$-modules of $(X/\rW)_{\conv,\zar}$ and $\rho_{X/\rW}^{*}(\mathscr{F})\xrightarrow{\sim} \mathscr{E}$ \eqref{rho conv coh}. 
	By \ref{higher direct image of con Fiso weak}, $\rR^{i}g_{\conv,\zar*}(\mathscr{F})$ is a coherent crystal of $\mathscr{O}_{Y/\rW}^{\rig}$-modules of $(Y/\rW)_{\conv,\zar}$. 
	We consider the canonical morphisms
	\begin{equation} \label{composition base change rel coh zar ad}
		\rR^{i}g_{\conv,\zar*}(\mathscr{F})\xrightarrow{\sim} \rho_{Y/\rW*}\rho_{Y/\rW}^{*}(\rR^{i}g_{\conv,\zar*}(\mathscr{F})) \to \rho_{Y/\rW*}(\rR^{i}g_{\rconv,\ad*}(\mathscr{E}))
	\end{equation}
	where the first arrow is an isomorphism by \ref{rho conv coh} and second arrow is induced by the base change morphism. 
	By \ref{evaluation direct image} and \ref{evaluation direct image rig}, $\rR^{i}g_{\conv,\zar*}(\mathscr{F})$ (resp. $\rho_{Y/\rW*}(\rR^{i}g_{\rconv,\ad*}(\mathscr{E}))$) is the sheaf associated to the presheaf on $\Conv(Y/\rW)$
	\begin{eqnarray*}
		&\mathfrak{T} \mapsto& \rH^{i}( (X_{T_{0}}/\mathfrak{T})_{\conv,\zar}, \omega_{\mathfrak{T}}^{*}(\mathscr{F})), \\
		\textnormal{(resp.} &\mathfrak{T} \mapsto& \rH^{i}( (X_{T_{0}}/\mathfrak{T})_{\rconv,\ad}, \omega_{\mathfrak{T}^{\rig}}^{*}(\mathscr{E})) ).
	\end{eqnarray*}
	
	By \ref{coh zar ad iso}, the canonical morphism
	\begin{displaymath}
		\rH^{i}( (X_{T_{0}}/\mathfrak{T})_{\conv,\zar}, \omega_{\mathfrak{T}}^{*}(\mathscr{F})) \xrightarrow{\sim}
		\rH^{i}( (X_{T_{0}}/\mathfrak{T})_{\rconv,\ad}, \omega_{\mathfrak{T}^{\rig}}^{*}(\mathscr{E}))
	\end{displaymath}
	is an isomorphism.
	The composition \eqref{composition base change rel coh zar ad} is induced by above morphisms and hence is an isomorphism. In view of the definition of $\rho_{Y/\rW*}$ \eqref{rho XS rconv}, we deduce that $\rho_{Y/\rW}^{*}(\rR^{i}g_{\conv,\zar*}(\mathscr{F}))\xrightarrow{\sim} \rR^{i}g_{\rconv,\ad*}(\mathscr{E})$ by \eqref{composition base change rel coh zar ad}. Then the assertion for admissible topology follows from \ref{rho conv coh}.

	Using \ref{fppf ad coh}, one verifies the proposition for fppf topology by comparing $\rR^{i}g_{\rconv,\ad}(-)$ and $\rR^{i}g_{\rconv,\fppf}(-)$ in a similar way as above.
\end{proof}

\begin{nothing}\label{pre proper descent}
	To prove \ref{main theorem}, we use the proper descent for convergent isocrystals developed by Ogus in \cite{Ogus84}. 
	Let $\mathfrak{T}$ be a formal $\rW$-scheme of finite type and $f:Z\to T_{0}$ a projective and surjective morphism. Then $f$ factors through a closed immersion $Z\to \mathbb{P}_{T_{0}}^{N}$ for some integer $N\ge 1$. 
	We denote by $\mathbb{P}_{\mathfrak{T}}^{N}$ the formal $\rW$-scheme associated to the inductive system $(\mathbb{P}_{\mathfrak{T}_{n}}^{N})_{n\ge 1}$.
	By \ref{def dilatation}, we can construct a family of adic formal $\mathbb{P}^{N}_{\mathfrak{T}}$-schemes $\{\mathfrak{T}_{Z,n}(\mathbb{P}_{\mathfrak{T}}^{N})\}_{n\ge 0}$. 
	Based on the following result, Ogus showed that a proper surjective $k$-morphism satisfies descent for convergent isocrystals (\cite{Ogus84} 4.6).
\end{nothing}

\begin{theorem}[\cite{Ogus84} 4.7, 4.8] \label{thm proper descent Ogus}
	For $n$ large enough, the morphism $\mathfrak{T}_{Z,n}(\mathbb{P}^{N}_{\mathfrak{T}})\to \mathfrak{T}$ is faithfully rig-flat \eqref{def rigid points}. 
\end{theorem}

A variant of \ref{Frob descent fppf topos} holds for rigid convergent topos:

\begin{prop} \label{Frob descent fppf topos rig}
	For every locally of finite type $k$-scheme $X$, the morphism
	\begin{equation}
		F_{X/k,\rconv,\fppf}:(X/\rW)_{\rconv,\fppf}\to (X'/\rW)_{\rconv,\fppf}
	\end{equation}
	is an equivalence of topoi.
\end{prop}
	
By \ref{functor rho fully faithful} and \eqref{Hom RConv}, the canonical functor $\RConv(X/\rW)\to \RConv(X'/\rW)$ induced by $F_{X/k}$ \eqref{functorial functor rig} is fully faithful.
In view of \ref{lemma adjunction iso}, \ref{lemma condition 4} and \ref{rho XS rconv}, the assertion follows.

\begin{nothing} \textit{Proof of \ref{main theorem}}.
	We prove the assertion for convergent isocrystals. Then the assertion for convergent $F$-isocrystals follows from \ref{Frob descent fppf topos rig} and a similar argument as in \ref{higher direct image of con Fiso weak Frob}. 
	The question being local \eqref{base change coro rig}, we may assume that $Y$ is separated and of finite type by \ref{base change coro rig}. Moreover, we may assume that $Y$ is reduced.

	By applying alteration to each irreducible component of $Y$ (\cite{dJ96} 4.1), there exists a smooth $k$-scheme $\widetilde{Y}$ and a proper surjective $k$-morphism $\widetilde{Y}\to Y$. By Chow's lemma (\cite{EGAII} 5.6.1), there exists a surjective $k$-morphism $Z\to \widetilde{Y}$ such that the composition $f:Z\to \widetilde{Y}\to Y$ is projective and surjective.
	We set $\mathscr{F}=\rR^{i}g_{\rconv,\fppf*}(\mathscr{E})$. 
	In view of \ref{base change coro rig} and \ref{main theorem weak}, the inverse image of $\mathscr{F}$ to $(\widetilde{Y}/\rW)_{\rconv,\fppf}$ is a coherent crystal. Then, so is $f^{*}_{\rconv,\fppf}(\mathscr{F})$.

	Let $(\mathfrak{T},u)$ be an object of $\Conv(Y/\rW)$. The morphism $f$ factor through a closed immersion $Z\to \mathbb{P}_{Y}^{N}$ for some integer $N\ge 1$. We set $T_{Z}=T_{0}\times_{Y}Z$.
	We take again the notation of \ref{pre proper descent} for the projective and surjective $k$-morphisms $T_{Z}\to T_{0}$. We choose an integer $n$ such that the morphism $\mathfrak{T}_{T_{Z},n}(\mathbb{P}_{\mathfrak{T}}^{N})\to \mathfrak{T}$ is faithfully rig-flat \eqref{thm proper descent Ogus}. We set $\mathfrak{R}=\mathfrak{T}_{T_{Z},n}(\mathbb{P}_{\mathfrak{T}}^{N})$, $\mathfrak{R}^{(1)}=\mathfrak{R}\times_{\mathfrak{T}}\mathfrak{R}$ and denote by $p_{1},p_{2}:\mathfrak{R}^{(1)}\to \mathfrak{R}$ two projections.

	Note that $\mathfrak{R}$ and $\mathfrak{R}^{(1)}$ define objects of $\Conv(Z/\rW)$ by \eqref{diagram red to S} and then of $\Conv(Y/\rW)$. Moreover, $\{\mathfrak{R}^{\rig}\to \mathfrak{T}^{\rig}\}$ defines a fppf covering of $\RConv(V/\rW)$. Since $f_{\rconv,\fppf}^{*}(\mathscr{F})$ is a coherent crystal of $\mathscr{O}_{Z/\rW}^{\rig}$-modules, following modules are coherent
	\begin{equation}
		\mathscr{F}_{\mathfrak{R}^{\rig}}=(f_{\rconv,\fppf}^{*}(\mathscr{F}))_{\mathfrak{R}^{\rig}},\qquad \mathscr{F}_{\mathfrak{R}^{(1),\rig}}=(f_{\rconv,\fppf}^{*}(\mathscr{F}))_{\mathfrak{R}^{(1),\rig}},
	\end{equation}
	and we have isomorphisms 
	\begin{equation}
		p_{2}^{\rig*}(\mathscr{F}_{\mathfrak{R}^{\rig}})\xrightarrow{\sim} \mathscr{F}_{\mathfrak{R}^{(1),\rig}} \xleftarrow{\sim} p_{1}^{\rig*}(\mathscr{F}_{\mathfrak{R}^{\rig}}).
	\end{equation}
	Then we obtain a descent data on $\mathscr{F}_{\mathfrak{R}^{\rig}}$ for the fppf covering $\{u: \mathfrak{R}^{\rig}\to \mathfrak{T}^{\rig}\}$. 
	There exists a coherent $\mathscr{O}_{\mathfrak{T}^{\rig}}$-module $\mathscr{M}$ and an isomorphism $u^{*}(\mathscr{M})\xrightarrow{\sim} \mathscr{F}_{\mathfrak{R}^{\rig}}$ by (\cite{Ab10} 5.11.11). 
	
	On the other hand, since $\mathscr{F}$ is a sheaf for fppf topology, we have an exact sequence
	\begin{equation}
		0\to \mathscr{F}(\mathfrak{T}^{\rig})\to \mathscr{F}(\mathfrak{R}^{\rig}) \to \mathscr{F}(\mathfrak{R}^{(1),\rig}).
	\end{equation}
	Then we deduce that $\mathscr{F}_{\mathfrak{T}^{\rig}}$ is isomorphic to $\mathscr{M}$ and hence is coherent.

	Let $g:\mathfrak{T}'\to \mathfrak{T}$ be a morphism of $\Conv(Y/\rW)$. Choose an integer $n$ large enough such that $\mathfrak{R}'=\mathfrak{T}_{T'_{Z},n}(\mathbb{P}_{\mathfrak{T}'}^{N})\to \mathfrak{T}'$ and $\mathfrak{R}=\mathfrak{T}_{T_{Z},n}(\mathbb{P}_{\mathfrak{T}}^{N})\to \mathfrak{T}$ are faithfully rig-flat.
	Since the construction of $\mathfrak{R}$ is functorial, we have a $\rW$-morphism $h: \mathfrak{R}'\to \mathfrak{R}$ compatible with $g$. Moreover, $h$ induces a morphism of $\Conv(Z/\rW)$.
	The transition morphism of $f_{\rconv,\fppf}^{*}(\mathscr{F})$ associated to $h^{\rig}$ is an isomorphism.
	By fppf descent, we deduce that the transition morphism $c_{g^{\rig}}$ of $\mathscr{F}$ associated to $g^{\rig}$ is an isomorphism (cf. \cite{Xu} 8.15). Then $\mathscr{F}$ is a crystal and the theorem follows.
\end{nothing}

\end{document}